\documentclass[12pt,leqno]{article}

\usepackage{amssymb,amsfonts,amsmath,amsthm,amscd,mathrsfs}
\usepackage{PSFigure,psfrag}
\usepackage{graphics,graphicx}
\def\IL{{\mathbb L}}
\def\IP{{\mathbb P}}
\def\IR{{\mathbb R}}
\def\IN{{\mathbb N}}

\def\IZ{{\mathbb Z}}

\def\n{\noindent}
\def\dsl{\textstyle\sum\limits}

\def\dis{\displaystyle}
\def\o{\omega}
\def\fr{\mbox{\footnotesize $\dis\frac{1}{2}$}}
\def\frvier{\mbox{\footnotesize $\dis\frac{1}{4}$}}
\def\ov{\overline}
\def\ve{\varepsilon}
\def\f{\footnotesize}
\def\r{\rightarrow}
\def\point{{\mbox{\large $.$}}}
\def\wh{\widehat}
\def\wt{\widetilde}

\def\cA{{\cal A}}
\def\cB{{\cal B}}
\def\cC{{\cal C}}

\def\cT{{\cal T}}

\def\cI{{\cal I}}
\def\cJ{{\cal J}}

\def\cP{{\cal P}}

\def\cS{{\cal S}}
\def\cG{{\cal G}}

\def\cV{{\cal V}}
\def\cW{{\cal W}}

\dimendef\dimen=0

\newtheorem{theorem}{Theorem}[section]
\newtheorem{lemma}[theorem]{Lemma}
\newtheorem{definition}[theorem]{Definition}
\newtheorem{corollary}[theorem]{Corollary}
\newtheorem{proposition}[theorem]{Proposition}
\newtheorem{remark}[theorem]{Remark}

\begin{document}

\baselineskip14pt
\noindent

\begin{center}
{\bf DECOUPLING INEQUALITIES AND\\ INTERLACEMENT PERCOLATION ON $\pmb{G \times \IZ}$}
\end{center}

\vspace{0.5cm}

\begin{center}
Alain-Sol Sznitman$^*$
\end{center}

\bigskip
\begin{center}
\end{center}

\bigskip
\begin{abstract}
We study the percolative properties of random interlacements on $G \times \IZ$, where $G$ is a weighted graph satisfying certain sub-Gaussian estimates attached to the parameters $\alpha > 1$ and $2 \le \beta \le \alpha + 1$, describing the respective polynomial growths of the volume on $G$ and of the time needed by the walk on $G$ to move to a distance. We develop decoupling inequalities, which are a key tool in showing that the critical level $u_*$ for the percolation of the vacant set of random interlacements is always finite in our set-up, and that it is positive when $\alpha \ge 1 + \frac{\beta}{2}$. We also obtain several stretched exponential controls both in the percolative and non-percolative phases of the model. Even in the case where $G = \IZ^d$, $d \ge 2$, several of these results are new.
\end{abstract}

\vspace{7cm}

\noindent
Departement Mathematik  \hfill  
\\
ETH-Zentrum\\
CH-8092 Z\"urich\\
Switzerland

\vfill
\noindent
$\overline{~~~~~~~~~~~~~~~~~~~~~~~}$\\
$^*$ {\small This research was supported in part by the grant ERC-2009-AdG  245728-RWPERCRI}

\newpage

\thispagestyle{empty}
~

\newpage
\setcounter{page}{1}

\setcounter{section}{-1}
\section{Introduction}

Random interlacements offer a microscopic model for the structure left at appropriately chosen time scales by random walks on large recurrent graphs, which are locally transient. In this work we investigate the percolative properties of random interlacements on transient weighted graphs $E$ of the form $G \times \IZ$, where the walk on the weighted graph $G$ satisfies certain sub-Gaussian estimates governed by two parameters $\alpha > 1$ and $\beta$ in $[2, \alpha + 1]$, respectively reflecting the volume growth of $G$, and the diffusive or sub-diffusive nature of the walk on $G$. Random interlacements on the weighted graphs considered in this article occur for instance in the description of the microscopic structure left by random walks on discrete cylinders $G_N \times \IZ$, with large finite bases $G_N$, which tend to look like $G$ in the vicinity of certain points, when the walk on $G_N \times \IZ$, runs for times comparable to the square of the number of points in $G_N$, see \cite{Szni09a}, \cite{Wind10}. In the spirit of \cite{Wind08}, they are expected to occur in the description of the microscopic structure left by random walks on suitable sequences of large finite graphs $E_N$, which tend to look like $E$ in the vicinity of certain points, when the walk runs for times proportional to the number of points in $E_N$. 

Here, our main interest lies in the percolative properties of $\cV^u$ the vacant set at level $u$ of random interlacements on $E$, as $u\ge0$ varies. This set is the complement in $E$ of the trace of the trajectories with labels at most $u$ in the interlacement point process. These percolative properties are naturally related to various disconnection and fragmentation problems, see \cite{DembSzni06}, \cite{Szni09c}, \cite{Szni09b}, \cite{AbetCandLairConi04}, \cite{BenjSzni06},  \cite{CernTeix11}, \cite{CernTeixWind09}, \cite{CoopFrie10}, \cite{TeixWind10}. For instance in the case of random walks on the cylinders $(\IZ/N\IZ)^d \times \IZ$, with $d \ge 2$, one can construct couplings of the trace in a box of size $N^{1- \varepsilon}$ of the complement of the trajectory of the walk after completion of a suitable number of excursions to the vicinity of the box, with the vacant set of random interlacements in a box of size $N^{1 - \varepsilon}$ in $\IZ^{d+1}$, see \cite{Szni09c}, \cite{Szni09b}, and \cite{Beli11}. These couplings enable one to show that the disconnection time $T_N$ of the cylinder by the walk has precise order $N^{2d}$. They also suggest a candidate limit distribution for $T_N / N^{2d}$ as $N$ goes to infinity, which brings into play the critical parameter $u_*$ for the percolation of $\cV^u$, see \cite{Szni09c}, \cite{Szni09b}. Proving that such a limit holds presently rests on being able to sharpen controls on the percolative properties of $\cV^u$ when $u$ is fixed but possibly close to $u_*$. In the case of random walks on $(\IZ/N\IZ)^d$, $d \ge 3$, similar couplings between the trace left by the complement of the trajectory of the walks at time $u N^d$ in boxes of size $N^{1-\varepsilon}$, and the trace of $\cV^u$ in boxes of size $N^{1-\varepsilon}$ in $\IZ^d$ can be constructed, cf.~\cite{TeixWind10}. They enable one to show that for small $u$ there typically is a giant component containing order $N^d$ sites in the complement of the trajectory of the walk at time $uN^d$, which is unique, at least when $d \ge 5$, and that for large enough $u$ there typically are only small components. The critical value $u_*$ for the percolation of $\cV^u$ is moreover conjectured to be the threshold for this fragmentation problem, separating the two behaviors mentioned above. The proof of such a conjecture analogously rests on the improvement of controls on the percolative properties of $\cV^u$ for $u$ fixed but possibly close to $u_*$. When $(\IZ/N\IZ)^d$ is replaced by a large $d$-regular graph on $N$ vertices, with $d \ge 3$, and random walk runs on the graph up to time $uN$, one can indeed show that when $u < u_*$, for large $N$ there typically is a component in the complement of the trajectory with order $N$ vertices, which is unique, whereas for $u > u_*$ all components are small. The relevant threshold for this fragmentation problem is now the critical value $u_*$ for the percolation of the vacant set of random interlacements on the $d$-regular tree, (the local model for typical points of large random $d$-regular graphs), cf.~\cite{CernTeix11}, \cite{CernTeixWind09}, \cite{CoopFrie10}.

\pagebreak
We now describe our results in more detail, but refer to Section 1 for precise definitions. We consider graphs of the form $E = G \times \IZ$, where $G$ is an infinite connected graph of bounded degree, endowed with bounded and uniformly positive weights along its edges. We assume that $G$ is $\alpha$-Ahlfors regular for some $\alpha > 1$, that is, see also (\ref{1.7}),
\begin{equation}\label{0.1}
\begin{array}{l}
\mbox{the volume of balls of radius $R$ in the graph-distance on $G$}
\\
\mbox{behaves as $R^\alpha$, up to multiplicative constants.}
\end{array}
\end{equation}

\n
We endow the product graph $E = G \times \IZ$ with weights, which take the value $1$ for ``vertical edges'' in the $\IZ$-direction, and agree with the corresponding weights on $G$ for ``horizontal edges'' in the $G$-direction. These weights naturally determine a random walk on $E$, which at each step jumps to one of its neighbors with a probability proportional to the weight of the edge leading to this neighbor. Due to (\ref{0.1}) the random walk on $E$ is in fact transient, see below(\ref{1.8}), and we assume that the Green density $g(\cdot,\cdot)$ has a power decay of the following kind: there is a $\beta$ in $[2,\alpha + 1]$, such that
\begin{equation}\label{0.2}
c(d(x,x') \vee 1)^{-\nu} \le g(x,x') \le c' (d(x,x') \vee 1)^{-\nu}, \;\mbox{for $x,x'$ in $E$}\,,
\end{equation}

\n
with $\nu = \alpha - \frac{\beta}{2}$ (a positive number due to the constraints on $\alpha, \beta$), and $d(\cdot,\cdot)$ the distance on $E$ defined by:
\begin{equation}\label{0.3}
d(x,x') = \max (d_G(y,y'), \,|z-z'|^{\frac{2}{\beta}}), \;\mbox{for $x = (y,z)$, $x' = (y',z')$ in $E$,}
\end{equation}

\n
where $d_G(\cdot,\cdot)$ stands for the graph-distance on $G$.

\medskip
 In fact, with the help of \cite{GrigTelc01}, \cite{GrigTelc02}, the above assumptions can be restated in terms the following sub-Gaussian estimates for the transition densities $p_n^G(y,y')$ of the walk determined by the weighted graph $G$, see Remark \ref{rem1.1} 2): 
\begin{equation}\label{0.4}
\begin{array}{rl}
{\rm i)} & p^G_n(y,y') \le c\,n^{-\frac{\alpha}{\beta}} \,\exp\{ - c(d_G(y,y')^\beta / n)^{\frac{1}{\beta - 1}}\}, \;\mbox{for $n \ge 1$, $y, y'$ in $G$},
\\[2ex]
{\rm ii)} &p^G_n(y,y') + p_{n+1}^G (y,y') \ge c\,n^{-\frac{\alpha}{\beta}} \,\exp\{-c(d_G(y,y')^\beta / n)^{\frac{1}{\beta - 1}}\}, 
\\[1ex]
&\mbox{for $n \ge 1 \vee d_G(y,y'), \; y,y'$ in $G$}\,,
\end{array}
\end{equation}

\medskip\n
and $c$ refers to positive constants changing from place to place.

\medskip
The classical example $G = \IZ^d$, $d \ge 2$, endowed with the natural weight equal to $1$ along all edges thus corresponds to $\alpha = d$, $\beta = 2$, whereas the case of $G$, the discrete skeleton of the Sierpinski gasket endowed with its natural weight, corresponds to $\alpha = \frac{\log 3}{\log 2}$, $\beta = \frac{\log 5}{\log 2} \;(> 2)$, see \cite{Jone96}, \cite{BarlCoulKuma05}. We also refer to \cite{Barl04a}, \cite{BarlCoulKuma05}, \cite{GrigTelc01}, \cite{GrigTelc02}, \cite{HambKuma04}, for many more examples and for equivalent characterizations of (\ref{0.4}).

\medskip
Random interlacements on the transient weighted graph $E$ consist in essence of a Poisson point process on the state space of doubly infinite trajectories on $E$ modulo time-shift, which tend to infinity at positive and negative infinite times, see \cite{Teix09b}, Section 1, and \cite{Szni10a}, Remark 1.4. A non-negative parameter $u$ plays the role of a multiplicative factor of the $\sigma$-finite intensity measure of this Poisson point process. In fact, one constructs ``at once'', on the same probability space $(\Omega, \cA, \IP)$, the whole family $\cI^u$, $u \ge 0$, of random interlacements at level $u$. These random subsets of $E$ are infinite when $u$ is positive, and come as unions of the ranges of the trajectories modulo time-shift with label at most $u$, in the canonical Poisson cloud constructed on $(\Omega, \cA, \IP)$. The law on $\{0,1\}^E$ of the indicator function of $\cI^u$ is characterized by the identity:
\begin{equation}\label{0.5}
\IP[\cI^u  \cap K = \emptyset] = \exp\{ - u \;{\rm cap}(K)\}, \;\mbox{for all finite subsets $K$ of $E$}\,,
\end{equation}
with ${\rm cap}(K)$ the capacity of $K$, see (\ref{1.21}).

\medskip
There is by now substantial evidence that the percolative properties of the vacant set $\cV^u = E \backslash \cI^u$ play an important role in understanding various disconnection and fragmentation problems for random walks on large recurrent approximations of $E$, see in particular  \cite{Szni09c}, \cite{Szni09b}, \cite{CernTeixWind09}, \cite{TeixWind10}.

\medskip
To further discuss these properties, we define for $x$ in $E$ and $u \ge 0$,
\begin{equation}\label{0.6}
\eta(x,u) = \IP[x \stackrel{\cV^u}{\longleftrightarrow} \infty]\,,
\end{equation}

\medskip\n
where the above notation refers to the event ``	$x$ belongs to an infinite connected component of $\cV^u$''. It is known from Corollary 3.2 of \cite{Teix09b} that the critical value
\begin{equation}\label{0.7}
u_* = \inf\{u \ge 0; \,\eta(x,u) = 0\} \in [0,\infty]\,,
\end{equation}

\n
does not depend on the choice of $x$. It is an important question whether $u_*$ is finite and positive. We show in Theorem \ref{theo4.1} that $\cV^u$ does not percolate for large $u$, that is:
\begin{equation}\label{0.8}
u_* < \infty\,.
\end{equation}

\n
When $\alpha \ge 1 + \frac{\beta}{2}$, i.e. $\nu \ge 1$ in (\ref{0.2}), we show in Theorem \ref{theo5.1} that for small $u > 0$, $\cV^u$ percolates in ``half-planes'' of $E = G \times \IZ$ that are product of a semi-infinite geodesic path in $G$ with $\IZ$, and in particular that
\begin{equation}\label{0.9}
u_* > 0, \;\mbox{when}\; \nu \ge 1\,.
\end{equation}

\n
Some special cases of (\ref{0.8}), (\ref{0.9}) are known, for instance when $E$ is $\IZ^{d+1}$, $d \ge 2$, endowed with its natural weight, see \cite{Szni10a}, \cite{SidoSzni09a}. Also when $\nu \ge 6$, the methods of Section 4 of \cite{Teix09b} can likely be adapted to prove (\ref{0.9}). But Theorems \ref{theo4.1} and \ref{theo5.1} yield further information. If $B(x,r)$ denotes the closed ball with center $x$ in $E$ and radius $r \ge 0$ in the $d(\cdot,\cdot)$-metric, see (\ref{0.3}), and $\partial_{\rm int} B(x,r)$ stands for the set of points in $B(x,r)$ neighboring $B(x,r)^c$, we introduce:
\begin{equation}\label{0.10}
u_{**} = \inf\{u \ge 0; \;\underset{L \r \infty}{\underline{\lim}}\;\sup\limits_{x \in E} \;\IP[B(x,L) \stackrel{\cV^u}{\longleftrightarrow} \partial_{\rm int} B(x,2L)] = 0\} \in [0,\infty]\,,
\end{equation}

\n
where the event under the probability refers to the existence of a nearest neighbor path in $\cV^u$ between $B(x,L)$ and $\partial_{\rm int} B(x,2L)$. We trivially have $u_* \le u_{**} \le \infty$, and Theorem \ref{theo4.1} actually shows that
\begin{equation}\label{0.11}
u_{**} <\infty\,,
\end{equation}

\n
and that for $u > u_{**}$, the connectivity function has a stretched exponential decay
\begin{equation}\label{0.12}
\IP[x  \stackrel{\cV^u}{\longleftrightarrow} \partial_{\rm int}  B(x,L)] \le c\,e^{-c'\,L^\gamma},
\end{equation}

\n
where $c(u)$, $c'(u) > 0$ and $0 < \gamma(u) < 1$ are constants possibly depending on $u$. We also introduce the value
\begin{equation}\label{0.13}
\wt{u} = \inf\{u \ge 0; \underset{L \r \infty}{\underline{\lim}} \;\sup\limits_{\cP, x \in \cP} \,\IP[B(x,L) \stackrel{*-\cI^u \cap \cP}{\longleftrightarrow} \partial_{\rm int}B(x,2L)] > 0\}\,,
\end{equation}

\n
where $\cP$ runs over all the half-planes in $E$, and the event under the probability refers to the existence of a $*$-path in $\cI^u \cap \cP$ between $B(x,L)$ and $\partial_{\rm int} B(x,2L)$, see below (\ref{3.9}) for the definition of a $*$-path in $\cP$. We show in Theorem \ref{theo5.1} and Corollary \ref{cor5.5} that
\begin{equation}\label{0.14}
\wt{u} \le u_*\,,
\end{equation}

\n
that for $u < \wt{u}$ and for any half-plane $\cP$, and $x$ in $\cP$, as above,
\begin{equation}\label{0.15}
\mbox{$\IP[$the component of $x$ in $\cV^u \cap \cP$ is finite and intersects $\partial_{\rm int} B(x,L)] \le ce^{-c'\,L^{\gamma'}}$},
\end{equation}

\n
with $c(u)$, $c'(u) > 0$ and $0 < \gamma(u)' < 1$, constants possibly depending on $u$, that in (\ref{0.13}) the quantity under the $\liminf$ satisfies a similar stretched exponential decay when $u < \wt{u}$, and importantly that
\begin{equation}\label{0.16}
\wt{u} > 0, \;\mbox{when $\nu \ge 1$}\,.
\end{equation}

\n
When $E = \IZ^{d+1}$, $d \ge 2$, for all the above results, the distance in (\ref{0.3}) and the associated balls can be replaced with the more common sup-norm distance and corresponding balls. Even in the special case $E = \IZ^{d+1}$, $d \ge 2$, the above results improve on present knowledge, see Remarks \ref{rem4.2} 1) and \ref{rem5.6} 1). 

\medskip
However in the case of a general $E$ with $\nu < 1$, (for instance when $G$ is the discrete skeleton of the Sierpinski gasket), it is a challenging question left open by the present work to understand whether $u_* > 0$, see Remark \ref{rem5.6} 2).

\medskip
We now provide some comments on the proofs. An important difficulty stems from the long range interaction present in random interlacements, (for instance the correlation of the events $\{x \in \cV^u\}$ and $\{x' \in \cV^u\}$ decays as $d(x,x')^{-\nu}$, when the distance between $x$ and $x'$ grows, see (\ref{1.38})). One important novelty of the present work, (even in the classical case $E = \IZ^{d+1}$, $d \ge 2$), is that we use the same renormalization scheme to treat both the percolative and non-percolative regimes of $\cV^u$. The heart of the matter is encapsulated in the {\it decoupling inequalities} stated in Theorem \ref{theo2.6}.

\medskip
The decoupling inequalities bound from above the probability of an intersection of $2^n$ decreasing events, or $2^n$ increasing events that depend on the respective traces of random interlacements in $2^n$ boxes in $E$, which are well ``spread out'', and can be thought of as the ``bottom leaves'' of a dyadic tree of depth $n$.

\medskip
More precisely, for $2^n$ decreasing, resp.~increasing, events $A_m$, resp.~$B_m$, on $\{0,1\}^E$, labelled by $m$, respectively adapted to $2^n$ balls of size of order $L_0$, which are well ``spread out'', this last feature involves a geometrically growing sequence of length scales $L_n = \ell^n_0 \,L_0$, see (\ref{2.1}) - (\ref{2.5}), the decoupling inequalities state that for $K > 0$, $0 < \nu' < \nu = \alpha - \frac{\beta}{2}$, when $\ell_0 \ge c(K,\nu')$, $L_0 \ge 1$, $u > 0$, one has
\begin{equation}\label{0.17}
\begin{array}{l}
\IP \Big[\bigcap\limits_m \,A_m^{u_+}\Big]  \le \prod\limits_m \,(\IP[A^u_m] + \ve(u)), \;\mbox{and} 
\end{array}
\end{equation}
\begin{equation}\label{0.18}
\begin{array}{l}
\!\!\!\!\!\!\! \IP \Big[\bigcap\limits_m \,B_m^{u_-}\Big]  \le \prod\limits_m \,(\IP[B^u_m] + \ve(u_-)),   
\end{array}
\end{equation}
where
\begin{equation}\label{0.19}
\begin{array}{l}
u_{\pm} = \prod\limits_{n \ge 0} \,\Big(1 + c_1\mbox{\f $\sqrt{\dis\frac{K}{(n+1)^3}}$}\; \;\ell_0^{-\frac{(\nu - \nu ')}{2}}\Big)^{\pm 1} u\,,
\end{array}
\end{equation}
and we have set for $v > 0$,
\begin{equation}\label{0.20}
\ve (v) = 2\,e^{- K v L_0^\nu\,\ell_0^{\nu '}} / (1 - e^{-K v \,L_0^\nu\,\ell_0^{\nu '}})\,.
\end{equation}

\medskip\n
The notation $A^v_m$, resp.~$B_m^v$, in (\ref{0.17}), resp.~(\ref{0.18}), refers to the event where the indicator function of $\cI^v$ belongs to $A_m$, resp.~$B_m$.

\medskip
The tuning of the level $u$ into $u_+$ or $u_-$ in (\ref{0.17}), (\ref{0.18}), corresponds to the ``sprinkling technique'', where throwing in some additional trajectories of the random interlacements provides a way to dominate long range interactions. Several variations of this strategy have already been employed, see~ \cite{Szni10a}, \cite{SidoSzni09a}, \cite{SidoSzni09b}. The novelty here is that percolative and non-percolative regimes are handled in a unified and more powerful fashion thanks to the decoupling inequalities, which themselves are the consequence of the renormalization step stated in Theorem \ref{theo2.1}.

\medskip
The strength of the decoupling inequalities comes to bear when applied to {\it cascading families} of events as discussed in Section 3. Informally, the occurrence of an event of such a family on a ``large scale'' trickles down to a ``lower scale'', and permits the construction of dyadic trees of finite depth, forcing the occurrence of events of the family at the bottom scale, in well spread out boxes corresponding to the leaves of the tree. Balancing out the combinatorial complexity of the possible dyadic trees arising in this construction, with the bounds coming from the decoupling inequalities (\ref{0.17}), (\ref{0.18}), is the key to the derivation of the stretched exponential bounds in (\ref{0.12}), (\ref{0.15}). The approach we develop here also gives a way to revisit \cite{Teix10} from a new perspective, see Remarks \ref{rem3.3} 2) and \ref{rem3.8}) 1).

\medskip
We will now explain the structure of this article.

\medskip
In Section 1 we introduce further notation and collect several useful results about the weighted graph $E$, random walk on $E$, and random interlacements on $E$. The Harnack inequality that appears in Lemma \ref{lem1.2} is instrumental in Section 2.

\medskip
Section 2 develops the renormalization scheme underpinning the decoupling inequalities of Theorem \ref{theo2.6}. The key renormalization step is carried out in Theorem \ref{theo2.1}.

\medskip
In Section 3 we bring into play the cascading property and give several examples in Proposition \ref{prop3.2} and Remark \ref{3.3}. The combination of this notion with the decoupling inequalities leads to the bounds stated in Theorem \ref{theo3.4}. Some consequences are stated in the Corollaries \ref{cor3.5} and \ref{cor3.7}.

\medskip
Section 4 applies Theorem \ref{theo3.4} and Corollary \ref{cor3.5} to prove the finiteness of $u_{**}$, (and hence of $u_*$), as well as the stretched exponential decay of the connectivity function when $u > u_{**}$, see (\ref{0.10}) - (\ref{0.12}). The main result appears in Theorem \ref{theo4.1}.

\medskip
In Section 5 we apply Corollary \ref{cor3.5} and \ref{cor3.7} to show in Theorem \ref{theo5.1} that $\wt{u} \le u_*$, derive for $u < \wt{u}$ a stretched exponential bound for the occurrence of a large finite cluster in the intersection of $\cV^u$ with a half-space, and establish that $\wt{u} > 0$, when $\nu \ge 1$.

\medskip
The Appendix provides the proof of the cascading property of the family of ``separation events" introduced in Remark \ref{rem3.3}) 2), in the spirit of \cite{Teix10}.

\medskip
Let us finally explain the convention we use concerning constants. We denote with $c,c', \wt{c}, \ov{c}$ positive constants changing from place to place, which only depend on the weighted graph $G$, (and in particular on $\alpha$ and $\beta$). Numbered constants $c_0,c_1,\dots$ refer to the value corresponding to their first appearance in the text. Finally dependence of constants on additional parameters appears in the notation. For instance $c(u)$ denotes a positive constant depending on the weighted graph $G$ and on $u$.

\bigskip\n
{\bf Acknowledgements:} We wish to thank Augusto Teixeira for pointing out reference \cite{AbetCandLairConi04}.

\section{Notation and some useful facts}
\setcounter{equation}{0}

In this section we introduce the precise set-up and collect various useful results about random walks and random interlacements on the type of weighted graphs we consider in this article. An important control on Harnack constants for harmonic functions in $d(\cdot,\cdot)$-balls is stated in Lemma \ref{lem1.2}. Further we collect some facts concerning capacity and entrance probabilities in Lemma \ref{lem1.3}. Several useful features of random interlacements appear in Lemma \ref{lem1.4}.

\medskip
We let $\IN = \{0,1,\dots\}$ denote the set of natural numbers. When $u$ is a non-negative real number, we let $[u]$ stand for the integer part of $u$. Given a finite set $A$ we denote with $|A|$ its cardinality. The graphs we consider have a countable vertex set and an edge set made of unordered pairs of the vertex set. With an abuse of notation we often denote a graph by its vertex set, when this causes no confusion. When $x,x'$ are distinct vertices of the graph, we write $x \sim x'$, if $x$ and $x$ are neighbors, i.e.~if $\{x,x'\}$ is an edge of the graph. The graphs we discuss here are connected and have bounded degree, i.e.~each vertex has a uniformly bounded number of neighbors. A finite path in a graph refers to a sequence $x_0,\dots,x_N$, $N \ge 0$, of vertices of the graph such that $x_i \sim x_{i+1}$, for $0 \le i < N$. We sometimes write path instead of finite path, when this causes no confusion. Given a graph $\Gamma$ as above, we denote with $d_\Gamma(\cdot,\cdot)$ the graph-distance on $\Gamma$, i.e.~the minimal number of steps of a finite path joining two given vertices of $\Gamma$. The graphs under consideration being connected, this number is automatically finite. We write $B_\Gamma(x,r)$ for the closed $d_\Gamma$-ball with center $x$ in $\Gamma$ and radius $r \ge 0$. When $U$ is a subset of vertices in $\Gamma$, we denote by $\partial U$, $\partial_{\rm int} U$, and $\ov{U}$, the vertex boundary, the interior vertex boundary, and the closure of $U$:
\begin{equation}\label{1.1}
\begin{split}
\partial U =& \; \{x \in U^c; \exists x' \in U, \; \mbox{with} \; x \sim x'\}, 
\\
\partial_{\rm int} U =&\; \{x \in U; \exists x' \in U^c, \;\mbox{with} \;x \sim x'\}, \;\mbox{and}\; \ov{U} = U \cup \partial U\,.
\end{split}
\end{equation}

\medskip\n
Further we denote by $\chi_U$ the indicator function of $U$, and write $U \subset \subset \Gamma$ to express that $U$ is a finite subset of the vertex set of $\Gamma$. A weight on $\Gamma$ is a symmetric non-negative function $\rho_{x,x'}$ on $\Gamma \times \Gamma$ such that $\rho_{x,x'} > 0$ if and only if $x \sim x'$. A weight $\rho$ induces a measure on the vertex set of $\Gamma$ via
\begin{equation}\label{1.2}
\rho(x) = \dsl_{x' \sim x} \; \rho_{x, x'} \;\; \mbox{and} \;\; \rho(U) = \dsl_{x \in U} \, \rho(x), \; \mbox{for $U \subseteq \Gamma$}\,.
\end{equation}

\n
The natural weight on $\Gamma$ refers to the choice $\rho_{x,x'} = 1_{x \sim x'}$, and in this case $\rho(x)$ coincides with the degree of the vertex $x$ in $\Gamma$. The set $\IZ^m$, $m \ge 1$, throughout this work is tacitly endowed with its usual graph structure and its natural weight, unless explicitly stated otherwise.

\medskip
A weighted graph $(\Gamma, \rho)$ induces a random walk on the vertex set of $\Gamma$ having transition probability 
\begin{equation}\label{1.3}
p_{x,x'} = \rho_{x,x'} / \rho(x), \;\mbox{for $x,x'$ in $\Gamma$}\,.
\end{equation}

\n
It satisfies the detailed balance equations relative to $\rho$:
\begin{equation}\label{1.4}
\rho(x) \,p_{x,x'} = \rho(x')\,p_{x'x}, \; \mbox{for $x,x'$ in $\Gamma$}\,.
\end{equation}

\n
We let $P_x$ stand for the canonical law of the walk on $\Gamma$ starting at $x$, with transition probability as in (\ref{1.3}), and denote by $(X_n)_{n \ge 0}$ the canonical process. The transition density of the walk is defined as follows:
\begin{equation}\label{1.5}
p_n(x,x') =P_x[X_n = x'] / \rho(x'), \; \mbox{for $n \ge 0$, $x,x'$ in $\Gamma$}\,.
\end{equation}

\n
It is a symmetric function of $x,x'$ thanks to (\ref{1.4}).

\medskip
As explained in the Introduction our main interest lies in graphs of the form $E = G \times \IZ$, where $G$ is an infinite connected graph of bounded degree endowed with a weight $\rho^G$ such that $\rho^G_{x,x'}$ is uniformly bounded, and uniformly positive, when $x \sim x'$. We then endow $E$ with the weight:
\begin{equation}\label{1.6}
\begin{split}
\rho_{x,x'} & = \rho^G_{y,y'}, \;\mbox{if $z = z'$}, \qquad  \qquad \mbox{with $x = (y,z), \; x' = (y',z')$ in $E$},
\\ 
& = 1, \quad\;\; \mbox{if $|z-z'| = 1$ and $y = y'$},
\\ 
& = 0, \quad\;\;  \mbox{otherwise.}
\end{split}
\end{equation}

\n
In particular for $x$ as above, $\rho(x) = \rho^G(y) + 2$, see (\ref{1.2}).

\medskip
We further assume that for some $\alpha > 1$, $(G, \rho^G)$ is $\alpha$-Ahlfors regular, that is
\begin{equation}\label{1.7}
c\, R^\alpha \le \rho^G (B_G(y,R)) \le c' R^\alpha, \; \mbox{for all $R \ge \frac{1}{2}$ and $y$ in $G$}\,.
\end{equation}

\n
When referring to the random walk in the weighted graph $(E,\rho)$, we use the notation introduced below (\ref{1.4}). The Green density is defined as follows:
\begin{equation}\label{1.8}
g(x,x') = \dsl_{n \ge 0} p_n(x,x'), \;\mbox{for $x,x'$ in $E$}\,.
\end{equation}

\n
It is a symmetric function. It is in fact finite. To see this point, one uses (\ref{1.7}) combined with the upper bound from Proposition 3.13 of \cite{Kuma10} on the transition density of the continuous time walk on $G$ using the weights on $G$ as jump rates. The transition density of the analogous continuous time walk on $E$ is then the product of the transition densities of the continuous time walk on $G$ and of the continuous time walk on $\IZ$. Relating $g(\cdot,\cdot)$ to the Green density of the continuous time walk on $E$, the finiteness of $g(\cdot,\cdot)$ readily follows. We also refer to the Appendix of \cite{Szni08} for related controls. In a more substantive fashion, we crucially assume the existence of a number $\beta$ in $[2,1+ \alpha]$ such that 
\begin{equation}\label{1.9}
c(d(x,x') \vee 1)^{-\nu} \le g(x,x') \le c'(d(x,x') \vee 1)^{-\nu}, \; \mbox{for $x,x'$ in $E$}\,,
\end{equation}

\n
where $\nu = \alpha - \frac{\beta}{2} \;(> 0)$, and $d(\cdot,\cdot)$ denotes the distance function on $E$ as in (\ref{0.3}), that is:
\begin{equation}\label{1.10}
d(x,x')= \max\{d_G(y,y'), \; |z-z'|^{\frac{2}{\beta}}\}, \; \mbox{for $x = (y,z)$, $x' = (y',z')$ in $E$}\,.
\end{equation}

\begin{remark}\label{rem1.1} \rm ~

\medskip\n
1) The condition $\beta \ge 2$ is natural since we want that $d(\cdot,\cdot)$ in (\ref{1.10}) defines a metric. The inequality $\beta \le 1 + \alpha$ is then a consequence of (\ref{1.7}) - (\ref{1.10}). Indeed by (\ref{1.9}), (\ref{1.11}), bringing into play the killed Green densities, (see also for instance (\ref{1.18}) below), one sees that $E_x[T_{B(x,R)}] \ge c\,R^\beta$, see  (\ref{1.12}) for notation, and a similar inequality then holds for the walk on $G$, with $B_G(y,R)$ in place of $B(x,R)$. The inequality $\beta \le 1 + \alpha$ now follows from a similar argument as in Lemma  1.2 of \cite{Barl04a}. 

\bigskip\n
2) Under the assumptions on $(G, \rho^G)$ stated above  (\ref{1.6}), for given $\alpha > 1$, $2 \le \beta \le \alpha + 1$, the conditions (\ref{1.7}) and (\ref{1.9}) are actually equivalent to the the sub-Gaussian bounds (\ref{0.4}) on  the probability density $p_n^G(y,y')$ of the walk on $G$, as we now explain.

\medskip
Under (\ref{0.4}), condition (\ref{1.7}) follows from Grigoryan-Telcs \cite{GrigTelc01}, p.~503-504. The proof of Lemma 5.4 of \cite{Szni08}, with obvious modifications since  $\rho^G$ is the natural weight in \cite{Szni08}, shows that (\ref{1.9}) holds. 

\medskip
Conversely given (\ref{1.7}) and (\ref{1.9}), the controls on Harnack constants for positive harmonic functions in $d(\cdot,\cdot)$-balls from Lemma \ref{lem1.2} below, imply, in the terminology of \cite{GrigTelc02}, the elliptic Harnack inequality on $(G, \rho^G)$. In addition, for the walk on $G$, when starting at $y$, the expected exit time from $B_G(y,R)$ is at least $c\,R^\beta$, as mentioned in 1) above. It is also at most $c'\,R^\beta$. To see this last point, one can for instance use the following fact, which is a consequence of (\ref{1.9}) and of lower bounds on the killed Green densities as in (\ref{1.20}) below,  that given $x$ and $x'$ in $E$ with $d(x,x') \le 3R$, the walk on $E$ enters $B(x',R)$ before exiting $B(x,cR)$, with a probability uniformly bounded away from $0$, when $c$ is chosen large. Condition (\ref{0.4}) now follows from Theorem 3.1 of Grigoryan-Telcs \cite{GrigTelc02}.

\medskip  
We refer to \cite{Barl04a}, \cite{BarlCoulKuma05},  \cite{GrigTelc01},  \cite{GrigTelc02}, \cite{HambKuma04}, for further examples and equivalent characterizations of (\ref{0.4}) combined with the assumptions stated above (\ref{1.6}). \hfill $\square$
\end{remark}

\medskip
The metric $d(\cdot,\cdot)$ in (\ref{1.10}) is well adapted to the possibly different natures of the horizontal and vertical displacements of the walk on $E$. It plays an important role throughout this work. We write $B(x,r)$ for the corresponding closed ball with center $x \in E$ and radius $r \ge 0$. When $K,K'$ are subsets of $E$ we write $d(K,K') = \inf\{d(x,x')$; $x \in K, x' \in K'\}$ for the mutual distance of $K$ and $K'$. When $K = \{x\}$ is a singleton, we simply write $d(x,K')$. As a direct consequence of (\ref{1.6}), (\ref{1.7}), (\ref{1.10}) we have the following control on the volume of balls in $E$:
\begin{equation}\label{1.11}
c\,R^{\alpha + \frac{\beta}{2}} \le \rho(B(x,R)) \le c'\,R^{\alpha + \frac{\beta}{2}}, \;\mbox{for $x \in E$, $R \ge \fr$}\,.
\end{equation}

\n
In the important special case $E = \IZ^{d+1}$, $d \ge 2$, (i.e. when $G = \IZ^d$, $d \ge 2$), we denote with $d_\infty(\cdot,\cdot)$ the sup-norm distance on $\IZ^{d+1}$ and write $B_\infty(x,r)$ for the corresponding closed ball with center $x$ and radius $r$. Let us already point out that when $E = \IZ^{d+1}$ we can replace $d(\cdot,\cdot)$ with the more common metric $d_\infty(\cdot,\cdot)$ in the various constructions and results we develop in this work.

\medskip
We denote by $W_+$ the set of nearest neighbor trajectories on $E$, defined for non-negative times and tending to infinity, and by $\cW_+$ the canonical $\sigma$-algebra on $W_+$ induced by the canonical process $X_n$, $n \ge 0$. Due to (\ref{1.9}), the walk on $(E, \rho)$ is transient and the set $W_+$ has full measure under any $P_x$. From now on we will view $P_x$ as a measure on $(W_+, \cW_+)$. When $m$ is a measure on $E$, we denote by $P_m$ the measure $\sum_{x \in E} m(x)P_x$, and by $E_m$ the corresponding expectation. We let $\theta_n$, $n \ge 0$, stand for the canonical shift on $W_+$, so that $\theta_n(w)(\cdot) = w(n + \cdot)$, for $w \in W_+$, $n \ge 0$. Given $U \subseteq E$, we write
\begin{equation}\label{1.12}
\begin{split}
H_U & = \inf\{n \ge 0; X_n \in U\}, \; \wt{H}_U = \inf\{ n \ge 1; X_n \in U\}
\\[1ex]
T_U & =  \inf\{n \ge 0; X_n \notin U\}\,,
\end{split}
\end{equation}

\n
for the entrance time in $U$, the hitting time of $U$, and the exit time from $U$. In the case $U = \{x\}$, we write $H_x$ and $\wt{H}_x$ for simplicity.

\medskip
We now come to some important controls on Harnack constants. Given $U \subseteq E$, we say that a function on $\ov{U}$, see (\ref{1.1}) for notation, is harmonic in $U$, when in the notation of (\ref{1.3})
\begin{equation}\label{1.13}
\dsl_{x' \sim x} \,p_{x,x'} f(x') = f(x), \;\mbox{for all $x$ in $U$}\,.
\end{equation}

\n
The following lemma is a crucial ingredient for the implementation of the sprinkling technique in Section 2. It is based on (\ref{1.9}) and Lemma A.2 of \cite{Szni09d}, which is an adaptation of Lemma of 10.2 of \cite{GrigTelc01}.

\begin{lemma}\label{lem1.2}
There exist $c, c_0 > 1$ such that for $x \in E$, $L \ge 1$, and $v$ a non-negative function on $\ov{B(x,c_0 L)}$, harmonic in $B(x,c_0\,L)$ one has
\begin{equation}\label{1.14}
\max\limits_{B(x,L)} \;v \le c \,\min\limits_{B(x,L)} \, v\,.
\end{equation}
\end{lemma}

\begin{proof}
We define $U_i = B(x,L_i)$, for $i = 1,2,3$, where $L_1 = L$, $L_2 = 3L$, and $L_3 \ge L_2$ so that $U_1 \subseteq U_2 \subseteq U_3$. It follows from Lemma A.2 of \cite{Szni09d} that when $v'$ is a non-negative function on $\ov{U}_3$ harmonic in $U_3$, one has
\begin{equation}\label{1.15}
\max\limits_{U_1} \;v \le K \,\min\limits_{U_1} \, v\,.
\end{equation}
where
\begin{equation}\label{1.16}
K = \max\limits_{\wt{x},\ov{x} \in U_1} \;\max\limits_{x' \in \partial_{\rm int} U_2} \, g_{U_3}(\wt{x},x') / g_{U_3}(\ov{x},x')\,,
\end{equation}

\n
and $g_{U_3}(\cdot,\cdot)$ denotes the killed Green density:
\begin{equation}\label{1.17}
g_{U_3}(\wt{x},x') = E_{\wt{x}}\Big[\dsl_{k \ge 0} 1\{X_k = x', \,T_{U_3} > k\}\Big] / \rho(x'), \;\mbox{for $\wt{x},x'$ in $E$}\,.
\end{equation}

\n
Applying the strong Markov property at time $T_{U_3}$ we find that
\begin{equation}\label{1.18}
g_{U_3}(\wt{x},x') = g(\wt{x},x') - E_{\wt{x}}[g(X_{T_{U_3}}, x')], \;\mbox{for $\wt{x},x'$ in $E$}\,.
\end{equation}

\medskip\n
As a result when $\wt{x},\ov{x}$ belong to $U_1$ and $x'$ to $\partial_{\rm int} U_2$, it follows from (\ref{1.9}) that
\begin{align}
g_{U_3}(\wt{x},x') & \le g(\wt{x},x') \le c(L_2 - L_1 - 1)^{-\nu} \le c \,L^{-\nu}, \label {1.19}
\\[1ex]
g_{U_3}(\ov{x},x') & \ge c\,L_2^{-\nu} - c(L_3 - L_2)^{-\nu} \ge c\,L^{-\nu}, \; \mbox{when} \; L_3 = c_0 L\,.\label {1.20}
\end{align}

\n
The claim (\ref{1.14}) now follows from (\ref{1.15}), (\ref{1.16}), together with $B(x,c_0 L) = U_3$ and $B(x,L) = U_1$.
\end{proof}

\medskip
We now recall some facts concerning the equilibrium measure and the capacity of a finite subset of $E$. Given $K \subset \subset E$, (see below (\ref{1.1}) for the notation) we write $e_K$ for the equilibrium measure of $K$ and ${\rm cap}(K)$ for its total mass, the capacity of $K$:
\begin{equation}\label{1.21}
\begin{split}
e_K(x) & = P_x[\wt{H}_K = \infty] \,\rho(x) \,1_K(x), \;\mbox{for $x \in E$}\,,
\\[1ex]
{\rm cap}(K) & = e_K(E) = \dsl_{x \in K} \,P_x[\wt{H}_K = \infty] \,\rho(x)\,.
\end{split}
\end{equation}

\n
The subadditive property of the capacity easily follows from (\ref{1.21}):
\begin{equation}\label{1.22}
{\rm cap}(K \cup K') \le {\rm cap}(K) + {\rm cap}(K'), \;\mbox{for $K,K' \subset \subset E$}\,,
\end{equation}

\n
and classically, one can express the probability that the walk enters $K$ via:
\begin{equation}\label{1.23}
P_x[H_K < \infty] = \dsl_{x' \in K} \,g(x,x') \,e_K(x'), \;\mbox{for $x \in E$}\,.
\end{equation}

\n
The following lemma collects further useful facts.
\begin{lemma}\label{lem1.3}
\begin{equation}\label{1.24}
\begin{array}{l}
{\rm cap}(\{x\}) = g(x,x)^{-1}, 
\\[1ex]
{\rm cap}(\{x,x'\}) = \dis\frac{g(x,x) + g(x',x') - 2g(x,x')}{g(x,x) g(x',x') - g(x,x')^2} \,, \;\mbox{for $x \not= x'$ in $E$}\,.
\end{array}
\end{equation}

\n
For $K \subset \subset E$, $x \in E$, one has:
\begin{equation}\label{1.25}
\begin{array}{l}
\dsl_{x' \in K} \,g(x,x') / \sup\limits_{\ov{x} \in K} \; \dsl_{x' \in K} \,g(\ov{x},x') \le P_x[H_K < \infty] \le
\\[3ex]
\dsl_{x' \in K} \,g(x,x') / \inf\limits_{\ov{x} \in K} \; \dsl_{x' \in K} \,g(\ov{x},x') \,.
\end{array}
\end{equation}
\begin{equation}\label{1.26}
c\,L^\nu \le {\rm cap}(B(x,L)) \le c' L^\nu, \;\mbox{for $x \in E$, $L \ge \fr$}\,.
\end{equation}
\end{lemma}

\begin{proof}
The claim (\ref{1.24}) follows by writing $e_{\{x\}} = \gamma_x \,\delta_x$ and $e_{\{x,x'\}} = \lambda_x \delta_x + \lambda_{x'} \delta_{x'}$, with $\gamma_x, \lambda_x, \lambda_{x'} \ge 0$, and then using (\ref{1.23}) together with the identity $P_{\ov{x}}[H_K < \infty] = 1$, when $\ov{x} \in K$, to determine $\gamma_x, \lambda_x, \lambda_{x'}$.

\medskip
The bound (\ref{1.25}) follows in a classical fashion from the $L^1(P_x)$-convergence of the bounded martingale $M_n = \sum_{x' \in K} g(X_{m \wedge H_K},x')$, $n \ge 0$, towards $1\{H_K < \infty\}$ $\sum_{x' \in K}$ $g(X_{H_K},x')$, and the resulting equality of the $P_x$-expectation of this quantity with $\sum_{x' \in K}$ $g(x,x')$.

\medskip
As for (\ref{1.26}), a variation of the above argument yields the identity 
\begin{equation*}
g(x,x') = E_x[g(X_{H_B},x'), \,H_B < \infty], \;\mbox{for $x,x'$ in $E$ and $B = B(x',L)$}\,.
\end{equation*}

\n
By (\ref{1.9}) we thus see that for $x \notin B$ one has
\begin{equation*}
c(L/d(x,x'))^\nu \le P_x [H_B < \infty] \stackrel{(\ref{1.23})}{=} \dsl_{\ov{x} \in B} \,g(x,\ov{x}) \,e_B(\ov{x}) \le c'(L/d(x,x'))^\nu\,.
\end{equation*}

\n
Letting $x$ tend to infinity the claim (\ref{1.26}) now follows from (\ref{1.9}).
\end{proof}

We now collect some properties of random interlacements on the transient weighted graph $(E,\rho)$, which will be useful in the following sections. Random interlacements on $E$ are defined on a canonical space $(\Omega, \cA, \IP)$.  On this space a certain Poisson point process on $W^* \times \IR_+$, (with $W^*$ the space of doubly infinite $E$-valued trajectories tending to infinity at positive and negative infinite times, modulo time-shift), having intensity measure $\nu(dw^*) du$, can be constructed. We refer to Section 2 of Teixeira \cite{Teix09b}, and also to Remark 1.4 of \cite{Szni10a}, for the precise definition of this probability space and of the $\sigma$-finite measure $\nu(dw^*)$. For the purpose of the present work we only need to recall that one can define on $(\Omega, \cA, \IP)$ some families of Poisson point processes on the space $W_+$, see above (\ref{1.12}) for the notation, namely $\mu_{K,u}(dw)$ and $\mu_{K,u',u}(dw)$, for $K \subset \subset E$, and $u,u' \ge 0$, with $u' < u$, so that
\begin{equation}\label{1.27}
\begin{array}{l}
\mbox{$\mu_{K,u',u}$ and $\mu_{K,u'}$ are independent with respective intensity measures}
\\
\mbox{$(u-u') P_{e_K}$ and $u' P_{e_K}$},
\end{array}
\end{equation}
and
\begin{equation}\label{1.28}
\mu_{K,u} = \mu_{K,u'} + \mu_{K,u',u}\,.
\end{equation}

\n
In essence $\mu_{K,u}$, resp.~$\mu_{K,u',u}$, keep track of the part after entrance in $K$ of the doubly infinite trajectories modulo time-shift, with labels at most $u$, resp.~with labels in $(u',u]$, belonging to the canonical Poisson cloud, that do enter $K$. These finite point measures on $W_+$ satisfy certain compatibility relations, the {\it sweeping identities:}
\begin{equation}\label{1.29}
\mu_{K,u} = \dsl^m_{i=1}  1\{H_K(w_i) < \infty\}\,\delta_{\theta_{H_K(w_i)}}, \;\mbox{if} \; \mu_{K',u} = \dsl^m_{i=1}  \,\delta_{w_i}, \;\mbox{for} \; K \subset K' \subset \subset E, u \ge 0\,,
\end{equation}

\n
together with the analogous relations for $\mu_{K,u',u}$ and $\mu_{K',u',u}$, $0 \le u' < u$. We refer to (1.13) - (1.15) of \cite{SidoSzni09b} or to (1.18) - (1.21) and Proposition 1.3 of \cite{Szni10a} for more details.

\medskip
Given $\o \in \Omega$, the random interlacement at level $u \ge 0$ is the random subset of $E$ defined by
\begin{equation}\label{1.30}
\begin{array}{l}
\cI^u(\o) = \bigcup\limits_{K \subset \subset E} \;\bigcup\limits_{w \in {\rm Supp} \,\mu_{K,u}(\o)} {\rm range}(w)\,,
\end{array}
\end{equation}

\n
where the notation ${\rm Supp} \,\mu_{K,u}(\o)$ refers to the support of the finite point measure $\mu_{K,u}(dw)$ and ${\rm range}(w) = w(\IN)$ for $w \in W_+$. The vacant set at level $u$ is then defined as
\begin{equation}\label{1.31}
\cV^u(\o) = E \backslash \cI^u(\o), \;\mbox{for} \;\o \in \Omega, \,u \ge 0\,.
\end{equation}

\n
As a straightforward consequence of the compatibility relations (\ref{1.29}), see for instance (1.54) of \cite{Szni10a}, one finds that:
\begin{equation}\label{1.32}
\begin{array}{l}
\cI^u(\o) \cap K = \bigcup\limits_{w \in {\rm Supp} \,\mu_{K',u}(\o)} w(\IN) \cap K, \;\mbox{for any}\; K \subset K' \subset \subset E, \, u \ge 0, \, \o \in \Omega\,.
\end{array}
\end{equation}

\n
It also readily follows from (\ref{1.27}) that
\begin{equation}\label{1.33}
\IP[\cV^u \supseteq K] = \exp\{- u \,{\rm cap}(K)\}, \;\mbox{for all} \; K \subset \subset E, \, u \ge 0\,.
\end{equation}

\n
This identity determines the law $Q_u$ on $\{0,1\}^E$ of $\chi_{\cV^u}$, the indicator function of $\cV^u$, see Remark 2.2 2) of \cite{Szni10a}, or Remark 2.3. of \cite{Teix09b}. The law $Q_u$ satisfies the FKG Inequality, see Theorem 3.1 of \cite{Teix09b}: when $f,g$ are increasing, $\sigma(\Psi_x,x \in E$)-measurable functions on $\{0,1\}^E$, (with $\Psi_x$, $x \in E$, the canonical coordinates on $\{0,1\}^E$), which are $Q_u$-square integrable, then
\begin{equation}\label{1.34}
E^{Q_u} [fg] \ge E^{Q_u}[f] \,E^{Q_u} [g]\,.
\end{equation}

\n
The following notation will be quite handy and recurrently used in the sequel. Given $\cJ(\o)$ a random subset of $E$, and $A$ a subset of $\{0,1\}^E$, we define
\begin{equation}\label{1.35}
A(\cJ) = \{\o \in \Omega; \; \chi_{\cJ(\o)} \in A\}\,,
\end{equation}
and write for $u \ge 0$:
\begin{equation}\label{1.36}
A^u \stackrel{\rm def}{=} A(\cI^u)\,.
\end{equation}

\n
The following lemma attests the presence of a long range dependence in random interlacements, and states a zero-one law for the occurrence of an infinite cluster in $\cV^u$. We recall that $\nu = \alpha - \frac{\beta}{2} > 0$, see (\ref{1.9}).

\begin{lemma}\label{lem1.4}  ($u > 0, x,x'$ in $E$)
\begin{align}
& \IP[x \in \cV^u] = \exp\Big\{ - \dis\frac{u}{g(x,x)}\Big\}\,, \label{1.37}
\\[1ex]
& c(u) d(x,x')^{-\nu} \le {\rm cov}_\IP (1 \{x \in \cV^u\}, \,1\{x' \in \cV^u\}) \le c' (u) d(x,x')^{-\nu}, \; \mbox{for $x \not= x'$}\,. \label{1.38}
\\[1ex]
&\mbox{$\IP[\,\cV^u$ contains an infinite connected component$\,] \in \{0,1\}$}\,. \label{1.39}
\end{align}
\end{lemma}

\begin{proof}
The identity (\ref{1.37}) directly follows from (\ref{1.24}), (\ref{1.33}). Similarly the covariance in (\ref{1.38}) is equal to, (writing $g_x$ in place of $g(x,x)$):
\begin{equation}\label{1.40}
\exp\{ - u (g_x^{-1} + g_{x'}^{-1})\}  \Big(\exp\Big\{u \Big[\dis\frac{g_x + g_{x'}}{g_x g_{x'}} - \dis\frac{g_x + g_{x'} - 2 g(x,x')}{g_x g_{x'} - g(x,x')^2} \Big]\Big\} - 1\Big)\,.
\end{equation}

\n
The expression inside the square bracket is non-negative as a result of (\ref{1.24}) and the subadditivity property of the capacity, see (\ref{1.22}). The claim (\ref{1.38}) when $d(x,x') > c(u)$ is a simple consequence of (\ref{1.9}). The case $0 < d (x,x') \le c(u)$ follows from the inequality
\begin{equation}\label{1.41}
g(x,x') / g_{x'} \le \wt{c} < 1, \;\mbox{(and similarly for $g(x',x)/g_x)$}\,,
\end{equation}

\n
which straightforwardly implies that the expression in the square bracket remains uniformly bounded away from zero for such $x,x'$. To prove (\ref{1.41}) one simply notes that $P_x [H_{x'} < \infty] = g(x,x') / g_{x'}$, and uses an escape route from $x'$ for the walk starting at $x$, which first moves in the $\IZ$-direction and then takes advantage of (\ref{1.9}) to show that the above probability is smaller than $\wt{c} < 1$.

\medskip
Let us finally prove (\ref{1.39}). When $z \in \IZ$, the translation $t_z$ in the $\IZ$-direction on $\{0,1\}^E$ is defined for $\sigma \in \{0,1\}^E$ via $t_z(\sigma) = \sigma(x+z)$, for $x \in E$, where $x + z$ is the natural translation of $x$ by $z$ in $E = G \times \IZ$. As a direct consequence of (\ref{1.33}) and the fact that ${\rm cap}(K+z) = {\rm cap}(K)$ for any $K \subset \subset E$, we see that $t_z, z \in \IZ$, preserves $Q_u$. Similar arguments as in the proof of Theorem 2.1 of \cite{Szni10a}, see in particular (2.6), (2.15) of \cite{Szni10a}, show that the group $t_z, z \in \IZ$ of $Q_u$-preserving transformations is ergodic, (see also (\ref{1.42}) below). The claim (\ref{1.39}) readily follows. 
\end{proof}

\begin{remark}\label{rem1.5} \rm ~

\medskip\n
1) The arguments used in the proof of (\ref{2.15}) of \cite{Szni10a} show that when $K,K'$ are disjoint finite subsets of $E$, and $F,F'$ are bounded measurable functions on the set of finite point measures on $W_+$, then for $u \ge 0$, one has
\begin{equation}\label{1.42}
|{\rm cov}_{\IP} (F(\mu_{K,u}), \,F'(\mu_{K',u}))| \le c\,u \; \dis\frac{{\rm cap}(K)\,{\rm cap}(K')}{d(K,K')^\nu} \;\|F\|_\infty \, \|F'\|_\infty\,,
\end{equation}

\n
with $d(K,K')$ as above (\ref{1.11}), and $\|F\|_\infty$, $\|F'\|_\infty$ the respective sup-norms of $F$ and $F'$.

\bigskip\n
2) In the case $E = \IZ^{d+1}$, $d \ge 2$, the definition of the capacity in (\ref{1.21}) differs by a multiplicative factor $2(d+1)$ from the definition in \cite{Szni10a}, due to the presence of $\rho(\cdot)$. The convention used in \cite{Szni10a} correspond to assigning a weight $(2d+2)^{-1}$ to each edge of $\IZ^{d+1}$, so that the volume measure coincides with the counting measure. This multiplicative factor leads to the (in most respect unessential) fact that random interlacement at level $u$ in the present work correspond to random interlacements at level $2(d+1) u$ in the terminology of \cite{Szni10a}, \cite{SidoSzni09a}, \cite{SidoSzni09b}.

\bigskip\n
3) In the context of Bernoulli percolation two inequalities play an important role in the development of the theory, the BK Inequality and the FKG Inequality, see for instance \cite{Grim99}. We have seen in (\ref{1.34}) that the FKG Inequality holds for $Q_u$. However it follows from (\ref{1.38}) that for $x \not= x'$ the events $\{x \in \cV^u\}$ and $\{x' \in \cV^u\}$ are correlated, in fact positively correlated, when $u > 0$. As a result the BK Inequality does not hold for $Q_u$. The decoupling inequalities of Theorem \ref{theo2.6} in the following section play the role of a partial substitute for the absence of a BK Inequality. \hfill $\square$
\end{remark}

\section{Decoupling inequalities and the sprinkling \\technique}
\setcounter{equation}{0}

In this section we implement a certain renormalization scheme. The goal is to derive decoupling inequalities for the probability of intersections of $2^n$ events, which are either all increasing or all decreasing, and respectively pertain to the state of random interlacements on $E$ at a certain level, in $2^n$ boxes of typical length $L_0$. These boxes should be thought of as the bottom leaves of a dyadic tree of depth $n$, and a geometrically increasing sequence of length scales, $L_n = \ell_0^n \,L_0$, $n \ge 0$, is used to quantify the fact that these boxes are well spread out. The key idea to overcome the presence of long range interactions in the model is the sprinkling technique. Roughly speaking one slightly increases the level so as to throw in more trajectories and dominate terms containing interactions between distant boxes at a given scale. The fashion in which we conduct the sprinkling technique is new, even in the case of $E = \IZ^{d+1}$, $d \ge 2$, and differs markedly from \cite{Szni10a}, \cite{SidoSzni09a},  \cite{SidoSzni09b}. The key induction step appears in Theorem \ref{theo2.1} and the crucial decoupling inequalities are presented in Theorem \ref{theo2.6}. In the important special case $E = \IZ^{d+1}$, $d \ge 2$, one can replace the distance $d(\cdot,\cdot)$ with the sup-norm distance $d_\infty(\cdot,\cdot)$ and the balls $B(x,r)$ with the corresponding balls $B_\infty(x,r)$ throughout the article, simply adapting constants when necessary.

\medskip
We introduce a geometrically growing sequence of length scales, (we recall that  $c_0 > 1$ appears in Lemma \ref{lem1.2}):
\begin{equation}\label{2.1}
L_n = \ell^n_0 \,L_0, \; \mbox{for $n \ge 0$, where $L_0 \ge 1$ and $\ell_0 \ge 10^6 c_0$}\,.
\end{equation}

\n
We want to derive upper bounds on the probability of intersections of events, which occur in boxes of size $L_0$, which are well spread-out. Specifically this last feature will bring into play certain embeddings of dyadic trees of finite depth $n$, which we now describe.

\medskip
Given $n \ge 0$, we denote with $T_n = \bigcup_{0 \le k \le n} \{1,2\}^k$, the canonical dyadic tree of depth $n$ and with $T_{(k)} = \{1,2\}^k$, the collection of elements of the tree at depth $k$. We write $\Lambda_n$ for the set of {\it embeddings} of $T_n$ in $E$, that is~of maps $\cT: T_n \r E$, such that defining
\begin{equation}\label{2.2}
x_{m, \cT} = \cT(m), \; \wt{C}_{m, \cT} = B(x_{m, \cT}, \, 10L_{n-k}), \;\mbox{for} \; m \in T_{(k)}, \; 0 \le k \le n\,,
\end{equation}

\n
one has for any $0 \le k < n$, $m \in T_{(k)}$, and $m_1,m_2$ the two descendants of $m$ in $T_{(k+1)}$ obtained by respectively concatenating $1$ and $2$ to $m$:
\begin{align}
&\wt{C}_{m_1,\cT} \cup \wt{C}_{m_2,\cT} \subseteq \wt{C}_{m,\cT}\,,\label{2.3}
\\[1ex]
&d(x_{m_1,\cT}, x_{m_2,\cT}) \ge \mbox{\f $\dis\frac{1}{100}$} \;L_{n-k}\,. \label{2.4}
\end{align}

\n
Due to (\ref{2.1}) it follows that for each $0 \le k \le n$, the ``boxes at depth $k$'', $\wt{C}_{m,\cT}$, $m \in T_{(k)}$,  are pairwise disjoint.

\medskip
Given $n \ge 0$ and $\cT \in \Lambda_n$, a collection of measurable subsets of $\{0,1\}^E$ indexed by the ``leaves'' of $T_n$, $A_m$, $m \in T_{(n)}$,  is said $\cT${\it -adapted} when (see above (\ref{1.34}) for notation):
\begin{equation}\label{2.5}
\mbox{$A_m$ is $\sigma(\Psi_x, x \in \wt{C}_{m,\cT})$-measurable for each $m \in T_{(n)}$}\,.
\end{equation}

\n
We also recall that given $u \ge 0$, the events $A^u_m (\subseteq \Omega)$, $m \in T_{(n)}$, are defined in (\ref{1.36}).

\medskip
We still need some notation before stating the main induction step of the renormalization scheme we develop in this section. Given $n \ge 0$ and $\cT \in \Lambda_{n+1}$, we denote with $\cT_1, \cT_2 \in \Lambda_n$, the two embeddings of $T_n$ such that $\cT_i(m) = \cT((i,i_1,\dots,i_k))$ for $i=1,2$ and $m = (i_1,\dots,i_k)$ in $T_{(k)}$, $0 \le k \le n$. Loosely speaking $\cT_1$ and $\cT_2$ are the two natural embeddings corresponding to the respective restrictions of $\cT$ to the descendants of $1$ and of $2$ in $T_{n+1}$. Given a $\cT$-adapted collection $A_m$, $m \in T_{(n+1)}$, we then define two collections $A_{m,1}$, $m \in T_{(n)}$, and $A_{m,2}$, $m \in T_{(n)}$, respectively $\cT_1$- and $\cT_2$-adapted, as follows:
\begin{equation}\label{2.6}
\mbox{$A_{m,i} = A_{(i,i_1, \dots, i_n)}$, for $i = 1,2$ and $m = (i_1,\dots,i_n) \in T_{(n)}$}\,.
\end{equation}
We can now state the main induction step.

\begin{theorem}\label{theo2.1} $(K > 0$, $0 < \nu^\prime < \nu \stackrel{(\ref{1.9})}{=} \alpha - \frac{\beta}{2})$

\medskip
When $\ell_0 \ge c(K,\nu^\prime)$, then for all $n \ge 0$, $\cT \in \Lambda_{n+1}$, for all $\cT$-adapted collections $A_m, m \in T_{(n+1)}$, of decreasing events, respectively $B_m$, $m \in T_{(n+1)}$, of increasing events, on $\{0,1\}^E$, and for all $0 < u^\prime < u$ such that
\begin{equation}\label{2.7}
u \ge \big(1 + c_1 \sqrt{K} \;(n+1)^{-\frac{3}{2}} \,\ell_0^{-\frac{(\nu - \nu^\prime)}{2}}\big) \,u^\prime\,,
\end{equation}

\vspace{-1ex}\n
one has
\begin{equation}\label{2.8}
\begin{array}{l}
\quad \IP \Big[ \bigcap\limits_{m \in T_{(n+1)}}
 A^u_m\Big]  \le \IP \Big[\bigcap\limits_{\ov{m}_1 \in T_{(n)}} 
 A^u_{\ov{m}_1,1}\Big] \Big(\IP \Big[\bigcap\limits_{\ov{m}_2 \in T_{(n)}} 
 A^{u^\prime}_{\ov{m}_2,2}\Big] + 2 e^{-2u^\prime \, \frac{K}{(n+1)^3} \,L^\nu_n \, \ell_0^{\nu^\prime}}\Big)\,, 
 \end{array}
 \end{equation}
 
 \vspace{-1ex}\n
respectively,
\begin{equation} \label{2.9}
\begin{array}{l}
\IP \Big[ \bigcap\limits_{m \in T_{(n+1)}} 
B^{u'}_m\Big]  \le \IP \Big[\bigcap\limits_{\ov{m}_1 \in T_{(n)}} 
B^u_{\ov{m}_1,1}\Big]  \,\IP \Big[\bigcap\limits_{\ov{m}_2 \in T_{(n)}} 
B^u_{\ov{m}_2,2}\Big] + 2 e^{-2u^\prime \, \frac{K}{(n+1)^3} \,L^\nu_n \, \ell_0^{\nu^\prime}} \,.
\end{array}
\end{equation}
\end{theorem}

\medskip
\begin{remark}\label{rem2.2} \rm  ~

\medskip\n
1) The events $A_{\ov{m}_1,1}$, $\ov{m}_1 \in T_{(n)}$, are decreasing and (\ref{2.8}) immediately implies that
\begin{equation}\label{2.8'}\tag{2.8'}
\begin{array}{l}
\IP \Big[ \bigcap\limits_{m \in T_{(n+1)}} 
A^u_m\Big]  \le \IP \Big[\bigcap\limits_{\ov{m}_1 \in T_{(n)}} 
A^{u^\prime}_{\ov{m}_1,1}\Big] \Big(\IP \Big[\bigcap\limits_{\ov{m}_2 \in T_{(n)}} 
A^{u^\prime}_{\ov{m}_2,2}\Big] + 2 e^{-2u^\prime \, \frac{K}{(n+1)^3} \,L^\nu_n \, \ell_0^{\nu^\prime}}\Big)\,,
\end{array}
\end{equation}

\medskip\n
2) By the FKG Inequality (\ref{1.34}) we see that
\begin{equation}\label{2.10}
\begin{array}{l}
\IP \Big[ \bigcap\limits_{m \in T_{(n)}} 
A^u_m\Big]  \ge \IP \Big[\bigcap\limits_{\ov{m}_1 \in T_{(n)}} 
A^u_{\ov{m}_1,1}\Big]\, \IP \Big[\bigcap\limits_{\ov{m}_2 \in T_{(n)}} 
A^u_{\ov{m}_2,2}\Big]  \,,
\end{array}
\end{equation}

\n
and a similar inequality holds for $B^u_m$, $m \in T_{(n+1)}$, $B^u_{\ov{m}_i,i}$, $\ov{m}_i \in T_{(n)}$, $i = 1,2$. \hfill $\square$
\end{remark}

We provide a brief outline of the proof of Theorem \ref{theo2.1}. When bounding the left-hand side of (\ref{2.8}), (\ref{2.9}), we have to cope with the mutual dependence between what happens in the union of boxes $\wt{C}_{\ov{m}_1,\cT_1}$ (of size $10 L_0)$ and  in the union of boxes $\wt{C}_{\ov{m}_2,\cT_2}$ (of size $10 L_0$) due to the presence of trajectories in the interlacement touching both collections. For this purpose we work in the above renormalization step with the two levels $0 < u' < u$. The sprinkling of additional trajectories corresponding to the increase from $u'$ to $u$, is supposed to dominate for the most part interactions taking place at level $u'$. Specifically we consider two disjoint balls $U_1, U_2$ of radius $c L_{n+1}$, with $c$ small, containing the respective collections of boxes, see (\ref{2.13}), and we keep track of excursions of the trajectories in the interlacement point process at level $u'$ between some region $W$, to be later determined, containing the two collections of boxes of size $10 L_0$, and the complement of $U = U_1 \cup U_2$. We collect all excursions generated by trajectories with labels at most $u'$ that do reenter $W$ after leaving $U$, see (\ref{2.42}), and dominate with high probability this (non-Poissonian) point process in terms of the collection of excursions of trajectories with labels between $u'$ and $u$ that do not come back to $W$ after leaving $U$. This domination step involves controls on the entrance distribution in $W$, which rely on the Harnack inequality from Lemma \ref{lem1.2}, see Lemma \ref{lem2.3}, and on a coupling lemma for the two point processes, see Lemma \ref{lem2.4}. We then choose the set $W$: a ``large $W$'' offers a high success probability for the domination in the coupling, however requires that $u$ and $u'$ are sufficiently far apart (as a function of the ``size of $W$''), cf.~(\ref{2.54}). We optimize so that domination occurs with high probability, but $u$ and $u'$ remain close enough, cf.~(\ref{2.58}). Finally the coupling combined with the monotone character of the events in the left-hand side of (\ref{2.8}), (\ref{2.9}), enables us to derive the upper bounds (\ref{2.8}), (\ref{2.9}), see (\ref{2.65}) - (\ref{2.68}).

\bigskip\n
{\it Proof of Theorem \ref{theo2.1}:} We consider $n \ge 0$, $0 < u^\prime < u$, $\cT \in \Lambda_{n+1}$, and write for simplicity  $x_1,x_2$ and $\wt{C}_1, \wt{C}_2$, dropping the subscript $\cT$ in (\ref{2.2}), when $m = 1,2 \in T_{(1)}$. We then define
\begin{equation}\label{2.11}
\wh{C}_i = \underset{m \in T_{(n)}}{\mbox{\f $\dis\bigcup$}}
\wt{C}_{m,\cT_i}, \;\; \mbox{for} \;\; i = 1,2, \;\; \mbox{and} \;\; V = \wh{C}_1 \cup \wh{C}_2\,.
\end{equation} 
In other words $\wh{C}_1$ and $\wh{C}_2$ are the respective unions of boxes of size of order $L_0$ corresponding to the respective descendants of $1$ and $2$ at the bottom scale. In particular we see by (\ref{2.3}) that
\begin{equation}\label{2.12}
\wh{C}_i  \subseteq \wt{C}_i, \;\; \mbox{for} \;\; i = 1,2 \,.
\end{equation} 

\n
Moreover by (\ref{2.1}), (\ref{2.4}), the following disjoint union
\begin{equation}\label{2.13}
U = U_1 \cup U_2, \;\; \mbox{where} \;\; U_i = B\Big(x_i, \;\dis\frac{L_{n+1}}{1000}\Big), \; i = 1,2\,,
\end{equation} 
is  such that $\ov{U}_1 \cap \ov{U}_2 = \emptyset$, and
\begin{equation}\label{2.14}
\wt{C}_i \subseteq B\Big(x_i, \;\dis\frac{L_{n+1}}{2000M}\Big) \stackrel{\rm def}{=} \wt{B}_i \subseteq U_i, \;\; \mbox{for} \;\; i = 1,2\,,
\end{equation} 

\n
where $M$ is determined in (\ref{2.36}) below, and such that $1 \le M \le \ell_0/(2 \cdot 10^4)$. We then introduce a set $W$ such that
\begin{equation}\label{2.15}
V \subseteq W \subseteq \wt{B}_1 \cup \wt{B}_2 \subseteq U\,,
\end{equation} 

\n
as well as the sequence $R_k, D_k, k \ge 1$ of successive returns of the walk on $E$ to $W$ and departures from $U$:
\begin{equation}\label{2.16}
\begin{split}
R_1 & = H_W, \, D_1 = T_U \circ \theta_{R_1} + R_1, \;\; \mbox{and by induction}
\\ 
R_{k+1} & = R_1 \circ \theta_{D_k} + D_k, \, D_{k+1} = D_1 \circ \theta_{D_k} + D_k , \;\mbox{for} \; k \ge 1\,.
\end{split}
\end{equation} 

\n
Our goal is to keep track of excursions between $W$ and $\partial U$ for the paths of the random interlacement entering $W$, and in essence dominate the effect of paths with labels at most $u^\prime$ that reenter $W$ for a second time, after leaving $U$, in terms of paths with labels between $u^\prime$ and $u$, that never reenter $W$ after leaving $U$. Optimizing on $W$ will then lead to the choice of $W$ in (\ref{2.58}) below.

\medskip
We thus introduce the Poisson point processes on $W^+$, (see (\ref{1.27}) for notation),
\begin{equation}\label{2.17}
\begin{split}
\zeta^\prime_\ell & = 1\{R_\ell < \infty = R_{\ell + 1}\} \,\mu_{W,u^\prime}, \;\; \mbox{for $\ell \ge 1$}\,,
\\ 
\zeta^*_\ell & = 1\{R_\ell < \infty = R_{\ell + 1}\} \,\mu_{W,u^\prime,u}, \;\; \mbox{for $\ell \ge 1$}\,.
\end{split}
\end{equation} 
By (\ref{1.27}) we see that
\begin{equation}\label{2.18}
\zeta^\prime_\ell , \, \ell \ge 1, \, \zeta^*_1 \;\; \mbox{are independent Poisson point processes on $W^+$}\,,
\end{equation} 
and their respective intensity measures are
\begin{equation}\label{2.19}
\begin{split}
\xi^\prime_\ell & = u^\prime \; 1\{R_\ell < \infty = R_{\ell + 1}\} \,P_{e_W}, \, \ell \ge 1\,,
\\ 
\xi^*_1& = (u - u^\prime) \; 1\{R_1< \infty = R_2\} \,P_{e_W}\,.
\end{split}
\end{equation} 

\n
Note also that the Poisson point processes on $W^+$
\begin{equation}\label{2.20}
1\{H_{\wh{C}_1} < \infty\}\,\mu_{W,u} \;\; \mbox{and} \;\; 1\{H_{\wh{C}_1} = \infty\} (\zeta^\prime_1 + \zeta^*_1) \;\;\mbox{are independent}\,.
\end{equation} 

\n
We then denote by $\cC$ the countable set of excursions from $W$ to $\partial U$:
\begin{equation}\label{2.21}
\begin{split}
\cC = \big \{ & \pi = (\pi(i))_{0 \le i \le N},\;\mbox{finite path, such that $\pi(0) \in W, \pi(N) \in \partial U$, and}
\\
&\mbox{$\pi(i) \in U$, for $0 \le i < N\big\}$}\,,
\end{split}
\end{equation} 
and with $\phi_\ell$, when $\ell \ge 1$, the map from $\{R_\ell < \infty = R_{\ell + 1}\} \subseteq W^+$ into $\cC^\ell$ such that:
\begin{equation}\label{2.22}
\begin{array}{l}
w \r \phi_\ell(w) = (w_1, \dots , w_\ell), \;\mbox{where}
\\
w_k (\cdot) = (X_{R_k + \cdot}(w))_{0 \le \cdot \le D_k(w) - R_k(w)}, \; 1 \le k \le \ell\,.
\end{array}
\end{equation} 

\n
The $\zeta^\prime_\ell$ in (\ref{2.17}) can be viewed as point processes on $\{R_\ell < \infty = R_{\ell + 1}\} (\subseteq W^+)$, and $\zeta^*_1$ can be viewed as a point process on $\{R_1 < \infty = R_2\}$. We then define 
\begin{equation}\label{2.23}
\begin{array}{l}
\mbox{$\wt{\zeta}^\prime_\ell$ the image of $\zeta^\prime_\ell$ under $\phi_\ell$, for $\ell \ge 1$, and}
\\
\mbox{$\wt{\zeta}^*_1$ the image of $\zeta^*_1$ under $\phi_1$}\,,
\end{array}
\end{equation} 
so that
\begin{equation}\label{2.24}
\begin{array}{l}
\mbox{$\wt{\zeta}^\prime_\ell, \, \ell \ge 1, \, \wt{\zeta}_1^*$ are independent Poisson point processes with intensity}
\\
\mbox{measures $\wt{\xi}^\prime_\ell, \, \ell \ge 1, \, \wt{\xi}^*_1$, which are respective images under $\phi_\ell, \, \ell \ge 1$, and}
\\
\mbox{$\phi_1$ of $\xi^\prime_\ell, \, \ell \ge 1$, and $\xi^*_1$}\,.
\end{array}
\end{equation} 

\n
The following lemma plays a crucial role for the domination procedure, i.e.~the sprinkling technique, we are currently setting into place.~It brings into play the control of Harnack constants in $d(\cdot,\cdot)$-balls from Lemma \ref{lem1.2}, together with some specific connectivity properties of the interior boundary of $d(\cdot,\cdot)$-balls owing to the special structure of $E = G \times \IZ$. As mentioned at the beginning of the section when $E = \IZ^{d+1}$, $d \ge 2$, our results also hold when we instead work with the more common sup-norm distance and the corresponding balls, possibly adapting constants.

\begin{lemma}\label{lem2.3} ($\ell_0 \ge c$)

\medskip
For any $\wt{W} \subseteq B(x_1,L_{n+1}/2000) \cup B(x_2,L_{n+1}/2000)$, $x \in \partial U \cup \partial_{\rm int} \,U$, $x^\prime \in \wt{W}$, one has
\begin{equation}\label{2.25}
c^\prime_2 \,L^{-\nu}_{n+1} \,e_{\wt{W}}(x') \le P_x [H_{\wt{W}} < \infty, \;X_{H_{\wt{W}}} = x^\prime] \le c_2 \,L^{-\nu}_{n+1} \,e_{\wt{W}}(x^\prime)\,.
\end{equation}
\end{lemma}

\begin{proof}
From the inclusion $\wt{W} \subseteq U$, see  (\ref{2.13}), we have the classical sweeping identity, (for instance resulting from (\ref{1.27}), (\ref{1.29})):
\begin{equation}\label{2.26}
e_{\wt{W}}(x^\prime) = P_{e_U} [H_{\wt{W}} < \infty, \; X_{H_{\wt{W}}} = x^\prime]\,,
\end{equation}
so that
\begin{equation}\label{2.27}
\begin{array}{l}
{\rm cap}(U) \, \inf\limits_{x \in \partial_{\rm int}U} \,P_x [H_{\wt{W}} < \infty, \,X_{H_{\wt{W}}} = x^\prime] \le e_{\wt{W}}(x^\prime) \le
\\[2ex]
 {\rm cap}\,(U)\, \sup\limits_{x \in \partial_{\rm int}U}  \,P_x [H_{\wt{W}} < \infty, \,X_{H_{\wt{W}}} = x^\prime] \,.
\end{array}
\end{equation}

\medskip\n
Observe that the function $f(x) = P_x [H_{\wt{W}}  < \infty$, $X_{H_{\wt{W}}} = x^\prime]$ is a non-negative function, which is harmonic in $\wt{W}^c$. Now $\wt{W} \subseteq B(x_1, \frac{L_{n+1}}{2000}) \cup B(x_2, \frac{L_{n+1}}{2000})$ and note that by (\ref{2.4}), (\ref{2.13}), (\ref{2.14}) one has:
\begin{equation}\label{2.28}
d(\partial_{\rm int} U, \; B(x_1,L_{n+1}/2000) \cup B(x_2,L_{n+1}/2000)) \ge \dis\frac{L_{n+1}}{1000} - \frac{L_{n+1}}{2000} - 1 > \frac{L_{n+1}}{5000} \;.
\end{equation}

\n
Moreover $\partial_{\rm int} U = \partial_{\rm int}\,U_1 \cup \partial_{\rm int} \,U_2$, and for $i=1,2$,
\begin{equation}\label{2.29}
\begin{array}{l}
\mbox{any two points on $\partial_{\rm int} \,U_i$ can be linked by a path in $\partial_{\rm int} \,U_i$ concatenation}
\\
\mbox{of at most three paths, which are either ``horizontal'' with at most $c \,L_{n+1}$}
\\
\mbox{steps or ``vertical'' with at most $c\, L_{n+1}^{\beta/2}$ steps.}
\end{array}
\end{equation}

\smallskip\n
The above statement is a straightforward consequence from the fact that for $x = (y,z)$ in $E$ and $R \ge 1$ one has the identity
\begin{align*}
\partial_{\rm int} \,B(x,R) =&\;  B_G(y,R) \times (z + \{-[R^{\beta/2}], [R^{\beta/2}]\}) \; 
\\
& \;\cup  \partial_{\rm int}B_G(y,R) \times (z + [-[R^{\beta/2}], [R^{\beta/2}]])\,.
\end{align*}

\medskip\n
Note that  $d(x_1,x_2) \le 20L_{n+1}$ due to (\ref{2.2}), (\ref{2.3}). It now follows that there is an integer constant $\ov{c}$ such that
\begin{equation}\label{2.30}
\begin{array}{l}
\mbox{for any two points $\ov{x}, \wt{x}$ in $\partial_{\rm int}\,U$, one can construct $\ov{x}(i)$, $0 \le i \le \ov{c}$,}
\\
\mbox{in $\partial_{\rm int}\,U \cup U^c$, such that $\ov{x}(0) = \ov{x}$, $\ov{x}(\ov{c}) = \wt{x}$, and}
\\
\mbox{$d(\ov{x}(i)$, $\ov{x}(i+1)) \le L \stackrel{\rm def}{=} L_{n+1} / (10^4 c_0)$}\,.
\end{array}
\end{equation}

\psfragscanon
\begin{center}
\includegraphics[width=15cm]{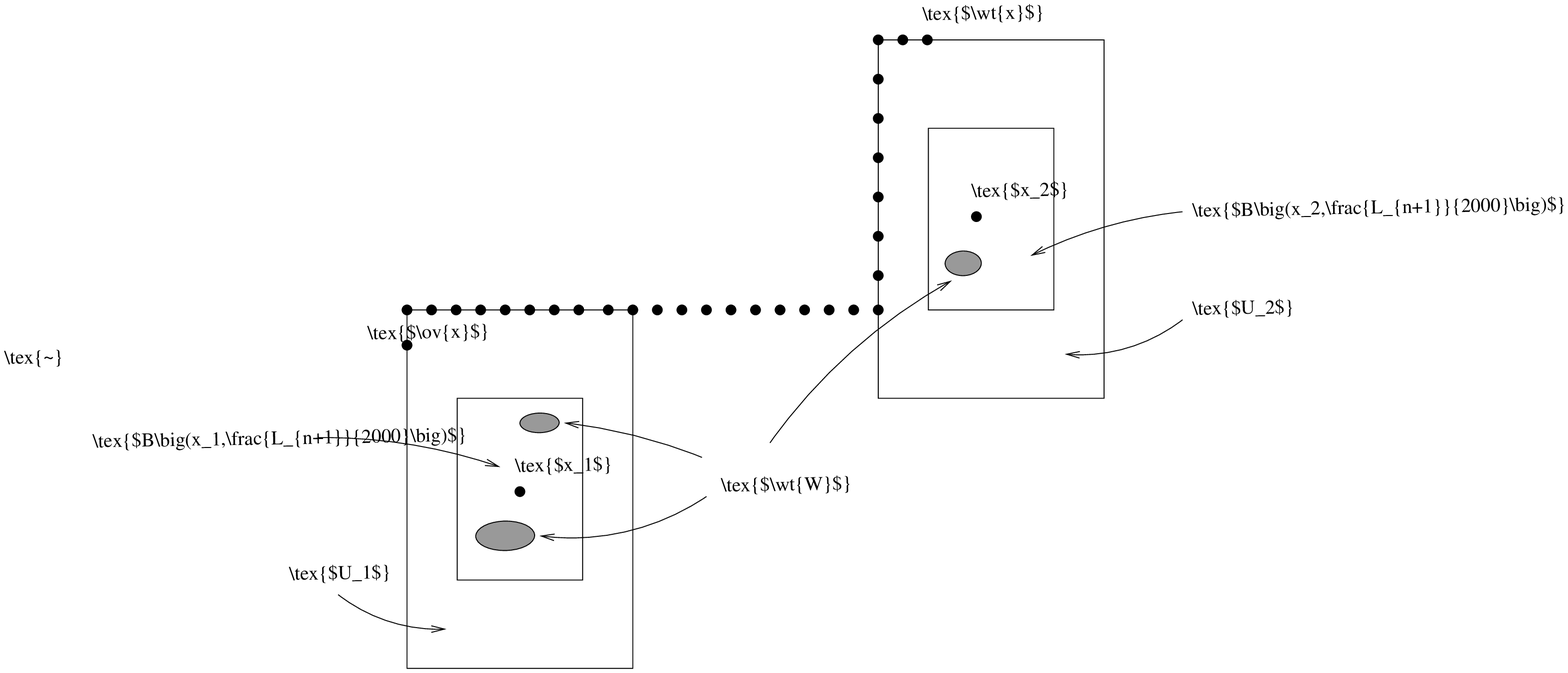}
\end{center}

\medskip
\begin{center}
\begin{tabular}{ll}
Fig.~1: & A schematic illustration of a possible sequence $\ov{x}(i)$, $0 \le i \le \ov{c}$, 
\\
&corresponding to the black dots.
\end{tabular}
\end{center}

\bigskip\n
For each $i$, the function $f$ is non-negative on $\ov{B(\ov{x}(i), c_0 L)}$ and harmonic in $B(\ov{x}(i), c_0L)$ $(\subseteq \wt{W}^c)$ due to (\ref{2.28}), (\ref{2.30}). Applying Lemma \ref{lem1.2} and a chaining argument, we see that:
\begin{equation}\label{2.31}
\sup\limits_{x \in \partial_{\rm int} \,U} \,P_x[H_{\wt{W}} < \infty, X_{H_{\wt{W}}} = x^\prime] \le c  \inf\limits_{x \in \partial_{\rm int}\, U} P_x[H_{\wt{W}} < \infty, X_{H_{\wt{W}}} = x']\,.
\end{equation}

\n
Moreover we know by (\ref{1.26}), (\ref{1.22}) that
\begin{equation}\label{2.32}
c\,L^\nu_{n+1} \le {\rm cap} (U) \le c^\prime \, L_{n+1}^\nu\,.
\end{equation}

\n
Coming back to (\ref{2.27}) we deduce (\ref{2.25}) for $x \in \partial_{\rm int}\,U$ and $x^\prime \in \wt{W}$. The extension to $x \in \partial U$ immediately follows thanks to the ellipticity condition satisfied by the walk, see above (\ref{1.6}). 
\end{proof}

\medskip
The next objective is to derive an upper bound on the intensity measures $\wt{\xi}^\prime_{\ell}$, $\ell \ge 2$, and a lower bound on $\wt{\xi}^*_1$, see (\ref{2.40}), (\ref{2.41}) below. We first note that for $\ell \ge 2$, $w_1,\dots,w_\ell \in \cC$, one has with the convention $R_0 = D_0 = 0$, as well as the notation $\ov{e}_W = e_W / {\rm cap}(W)$,
\begin{equation}\label{2.33}
\begin{array}{l}
\wt{\xi}^\prime_\ell (w_1,\dots,w_\ell) =
\\
 u^\prime \,P_{e_W} \big[R_\ell < \infty, (X_{R_k + \cdot})_{0 \le \cdot \le D_k - R_k} = w_k(\cdot), \, 1 \le k \le \ell, \,R_{\ell + 1} = \infty\big] \underset{\rm property}{\stackrel{\rm strong \; Markov}{\le}}
 \\
 u^\prime\, E_{e_W} \big[R_\ell  < \infty, (X_{R_k + \cdot})_{0 \le \cdot \le D_k - R_k} = w_k(\cdot), \,1 \le k < \ell, 
 \\[1ex]
 \qquad \quad  P_{X_{R_\ell}}[(X_\point)_{0 \le \cdot \le T_U} = w_\ell(\cdot)]\big] =
 \\[1ex]
 u^\prime\, E_{e_W} \big[D_{\ell - 1} < \infty, (X_{R_k + \cdot})_{0 \le \cdot \le D_k - R_k} = w_k(\cdot), \,1 \le k < \ell, 
  \\
 \qquad \quad E_{X_{D_{\ell -1}}}[H_W < \infty, \,P_{X_{H_W}} [(X_\point)_{0 \le \cdot \le T_U} = w_\ell(\cdot)]\big] \stackrel{(\ref{2.25})}{\le}
 \\[2ex]
 u^\prime \, \mbox{\f $\dis\frac{c_2}{L^\nu_{n+1}}$} \; {\rm cap}(W) \,E_{e_W}  \big[D_{\ell - 1} < \infty, (X_{R_k + \cdot})_{0 \le \cdot \le D_k - R_k} = w_k, \,1 \le k \le \ell - 1\big]  \; \times
 \\[2ex]
 P_{\ov{e}_W} [(X_\point)_{0 \le \cdot \le T_U} = w_\ell(\cdot)]\,,
 \end{array}
\end{equation}
and by induction
\begin{equation*}
\begin{array}{l}
\le u^\prime {\rm cap}(W) \Big(c_2 {\dis\frac{{\rm cap}(W)}{L^\nu_{n+1}}}\Big)^{\ell - 1} \; \prod\limits^\ell_{k = 1} \;\Gamma(w_k)\,,
\end{array}
\end{equation*}

\medskip\n
where $\Gamma$ is the law on $\cC$ of $(X_\point)_{0 \le \cdot \le T_U}$ under $P_{\ov{e}_W}$:
\begin{equation}\label{2.34}
\Gamma(\cdot) = P_{\ov{e}_W} [(X_\point)_{0 \le \cdot \le T_U} = \cdot ]\,.
\end{equation}

\n
As for the lower bound on $\wt{\xi}^*_1$, we note that for $w \in \cC$ one has
\begin{equation}\label{2.35}
\wt{\xi}^*_1(w) = (u-u^\prime) \,P_{e_W} [(X_\point)_{0 \le \cdot \le T_U} = w(\cdot), \, H_W \circ \theta_{T_U} = \infty]\,.
\end{equation}

\n
Observe that for $x \in \partial U$ we have by (\ref{2.14}), (\ref{2.15}), and (\ref{2.25})
\begin{equation}\label{2.36}
\begin{array}{l}
P_x[H_W < \infty] \le \dis\frac{c_2}{L^\nu_{n+1}} \,{\rm cap}(W)\stackrel{(\ref{1.26})}{\le} \dis\frac{c}{L^\nu_{n+1}} \,\Big(\dis\frac{L_{n+1}}{2000M}\Big)^\nu \le (2 e)^{-1}, 
\\[3ex]
\mbox{when we choose $M = c_3 \ge 1$}\,,
\end{array}
\end{equation}

\n
and we assume $\ell_0 \ge 2 \cdot 10^4 c_3$, so that $1 \le M \le \ell_0 / (2 \cdot 10^4)$. Coming back to (\ref{2.35}) we see applying the strong Markov property at time $T_U$ that for $w \in \cC$:
\begin{equation}\label{2.37}
\wt{\xi}^*_1 (w) \ge \dis\frac{(u-u^\prime)}{2} \;{\rm cap}(W) \, \Gamma (w)\,.
\end{equation}
We then introduce the notation
\begin{equation}\label{2.38}
\lambda^\prime_W = u^\prime {\rm cap}(W), \;\; \beta_W = \dis\frac{c_2\, {\rm cap}(W)}{L^\nu_{n+1}} \;\Big( \le \dis\frac{1}{2e}, \; \mbox{see (\ref{2.36})\Big)} \,,
\end{equation}
as well as
\begin{equation}\label{2.39}
\lambda^*_W =  \dis\frac{(u-u^\prime)}{2} \;{\rm cap}(W)\,.
\end{equation}
We have thus shown that
\begin{align}
\wt{\xi}^\prime_\ell & \le \lambda^\prime_W \; \beta_W^{\ell - 1} \;\Gamma^{\otimes \ell}, \;\mbox{for $\ell \ge 2$, and} \label{2.40}
\\[1ex]
\wt{\xi}^*_1 & \ge \lambda^*_W \;\Gamma\,. \label{2.41}
\end{align} 

\medskip\n
We now come to the construction of a coupling that will provide the domination argument hinted at in the comment below (\ref{2.16}). We denote by $s_\ell$ the map which sends a finite point measure $m$ on $\cC^\ell$ onto a finite point measure on $\cC$ via:
\begin{equation*}
s_\ell(m) = \dsl^\ell_{k=1} \;p_k^\ell \circ m\,,
\end{equation*}

\n
where $p^\ell_k$ stands for the $k$-th canonical coordinate on $\cC^\ell$, for $1 \le k \le \ell$. In other words $s_\ell(m) = \sum^N_{i=1} \delta_{w_1^i} + \dots + \delta_{w_\ell^i}$, when $m = \sum^N_{i=1} \delta_{(w^i_1,\dots,w_\ell^i)}$. We want to construct a coupling of the finite point measure on $\cC$:
\begin{equation}\label{2.42}
\wt{\sigma}^\prime = \dsl_{\ell \ge 2} \;s_\ell(\wt{\zeta}^\prime_\ell)\,,
\end{equation}

\n
which collects all excursions from $W$ to $\partial U$ induced by trajectories in the support of $\mu_{W,u^\prime}$ having at least once exited $U$ and reentered $W$, with the finite point measure $\wt{\zeta}^*_1$, so that with high probability $\wt{\zeta}_1^*$ dominates $\wt{\sigma}^\prime$. Incidentally let us mention that $\wt{\sigma}^\prime$ is not a Poisson point measure, see Remark \ref{rem2.5} 1) below. This feature makes the domination argument more delicate.

\medskip
With this in mind we consider an auxiliary probability space $(\wt{\Omega}, \wt{\cA}, \wt{\IP})$ endowed with a collection of independent Poisson variables $N^\prime_\ell$, $\ell \ge 2$, and with a Poisson variable $N^*_1$, with respective intensities,
\begin{align}
\lambda^\prime_\ell & = \lambda^\prime_W \,\beta_W^{\ell - 1}, \; \ell \ge 2, \;\mbox{and} \label{2.43}
\\[1ex]
\lambda^*_1 & = \lambda^*_W\,, \label{2.44}
\end{align} 

\n
and with an independent collection of i.i.d. $\cC$-valued random variables $\gamma_i, i \ge 1$, with distribution $\Gamma$, cf.~(\ref{2.34}). We then define
\begin{equation}\label{2.45}
N^\prime = \dsl_{\ell \ge 2} \;\ell \, N^\prime_\ell \,.
\end{equation}

\n
By (\ref{2.38}), we see that $N^\prime$ has finite expectation, (actually $N^\prime$ is distributed as a certain Poisson compound distribution, see below (\ref{2.51})), and we define the finite point processes on $\cC$:
\begin{equation}\label{2.46}
\Sigma^\prime = \dsl_{1 \le i \le N'}  \;\delta_{\gamma_i}, \; \Sigma^*_1 = \dsl_{1 \le i \le N^*_1}  \delta_{\gamma_i}\,.
\end{equation}

\begin{lemma}\label{lem2.4} (Coupling Lemma)

\medskip
One can construct on some probability space $(\ov{\Omega}, \ov{\cA}, \ov{\IP})$ variables $N^\prime, N_1^*, \Sigma^\prime, \Sigma^*_1$ as above, and $\ov{\sigma}\,'$ distributed as $\wt{\sigma}\,'$, and $\ov{\zeta}_1\,\!\!\! ^{*}$ distributed as $\wt{\zeta}^{*}_1$ so that
\begin{equation}\label{2.47}
\ov{\sigma}\,' \le \Sigma ' \; \mbox{and} \; \Sigma^*_1 \le \ov{\zeta}_1\,\!\!\! ^{*}\,,
\end{equation}
and in particular
\begin{equation}\label{2.48}
\ov{\IP}\mbox{-a.s. on} \; \{N' \le N^*_1\}, \; \ov{\sigma}\,' \le   \ov{\zeta}_1\,\!\!\! ^{*}\,.\end{equation}
\end{lemma}

\begin{proof} 
Observe that $\Sigma '$ has the same distribution as $\sum_{\ell \ge 2} s_\ell (\ov{\zeta}_\ell)$, where $\ov{\zeta}_\ell$, $\ell \ge 2$, are independent Poisson point processes on $\cC^\ell$ with respective intensities $\lambda^\prime_\ell \,\Gamma^{\otimes \ell} \ge \wt{\xi}'_\ell$, due to (\ref{2.40}), (\ref{2.43}). We construct a coupling of the $\wt{\zeta}'_\ell$, $\ell \ge 2$, and the $\ov{\zeta}_\ell$, $\ell \ge 2$, say by thinning the $\ov{\zeta}_\ell$, $\ell \ge 2$, so that $\ov{\zeta}_\ell \ge \wt{\zeta}'_\ell$, for each $\ell \ge 2$. This permits after ``enlarging the probability space $(\wt{\Omega}, \wt{\cA}, \wt{\IP})$'', to construct a $\ov{\sigma}\,'$ with same distribution as $\wt{\sigma}\,'$ and such that $\Sigma^\prime \ge \ov{\sigma}\,'$.

\medskip
Then note that $\Sigma^*_1$ is a Poisson point process on $\cC$ with intensity measure $\lambda^*_1 \,\Gamma \le \wt{\xi}^{*}_1$ by (\ref{2.41}), (\ref{2.44}). We can then construct, by ``further enlarging'' the probability space, a $\ov{\zeta}_1\,\!\!\! ^{*}$ having same law as $\wt{\zeta}^{*}_1$, such that $\Sigma^*_1 \le \ov{\zeta}_1\,\!\!\! ^{*}$. This yields (\ref{2.47}) and (\ref{2.48}) is an immediate consequence.
\end{proof}

\medskip
The next task is to bound the probability of the ``bad event'' $\{N^*_1 < N'\}$, cf.~(\ref{2.48}), and by appropriately selecting $W$ in (\ref{2.58}) obtain the bound (\ref{2.59}). We thus note that
\begin{equation}\label{2.49}
\ov{\IP}[N_1^* < N'] \le \ov{\IP} \Big[N^*_1 \le \dis\frac{\lambda^*_W}{2}\Big] + \ov{\IP} \Big[N' > \dis\frac{\lambda^*_W}{2}\Big]\,.
\end{equation}

\n
By classical bounds on the tail of a Poisson variable of intensity $\lambda^*_W$ we have:
\begin{equation}\label{2.50}
\ov{\IP}\Big[N_1^* \le \dis\frac{\lambda^*_W}{2}\Big] \le e^{-c \lambda^*_W}\,.\end{equation}

\n
As for the last term of (\ref{2.49}), the exponential Chebyshev inequality yields
\begin{equation}\label{2.51}
\ov{\IP}\Big[N' > \dis\frac{\lambda^*_W}{2}\Big] \le \exp\Big\{- a\dis\frac{\lambda^*_W}{2}\Big\} \; E^{\ov{\IP}}[e^{aN'}], \;\mbox{for $a > 0$}\,.
\end{equation}
We then observe that
\begin{equation*}
E^{\ov{\IP}} [e^{aN'}] \stackrel{(\ref{2.45})}{=} \exp\Big\{ \dsl_{\ell \ge 2} \,\lambda^\prime_\ell (e^{a \ell} -1)\Big\} = \exp\Big\{ \lambda^\prime_W \, \dsl_{\ell \ge 2}\,\beta_W^{\ell - 1} (e^{a \ell} - 1)\Big\}\,.
\end{equation*}

\n
This formula incidentically shows that $N'$ has the law of a compound Poisson variable with compounding distribution which is that of a geometric variable with parameter $1 - \beta_W$ conditioned to be at least 2. Note that $E[e^{a N'}] = \infty$, when $e^a \ge \beta_W^{-1}$, but that
\begin{equation}\label{2.52}
E^{\ov{\IP}} [e^{a N'}] \le \exp\Big\{\lambda^\prime_W \,\beta_W \, \dis\frac{e^{2a}}{1 - \beta_W e^a}\Big\}, \;\mbox{when $\beta_W \,e^a < 1$}\,.
\end{equation}
Coming back to (\ref{2.51}), we have shown that
\begin{equation}\label{2.53}
\ov{\IP}\Big[N' > \dis\frac{\lambda^*_W}{2}\Big] \le \exp\Big\{ - a\dis\frac{\lambda^*_W}{2} + \lambda^\prime_W \,\beta_W \,  \dis\frac{e^{2a}}{1 - \beta_W e^a}\Big\}, \;\mbox{for $\beta_W \,e^a < 1$}\,.
\end{equation}

\n
Choosing $a=1$, we know from (\ref{2.38}) that $2 e \beta_W \le 1$, so that when $\frac{\lambda^*_W}{4} \ge 2 \beta_W \,e^2 \,\lambda^\prime_W$ holds, one finds that $\ov{\IP} [N' > \frac{\lambda^*_W}{2}] \le \exp\{ - \frac{\lambda^*_W}{4}\}$. Thus coming back to (\ref{2.49}) we have shown that for any $W$ such that $V \subseteq W \subseteq \bigcup^2_{i=1} \,B(x_i,L_{n+1}/(2000 c_3)) = \wt{B}_1 \cup \wt{B}_2$, cf.~(\ref{2.15}), (\ref{2.36}), (\ref{2.39}), one has
\begin{equation}\label{2.54}
\ov{\IP}[N_1^* < N'] \le 2 \exp\{ - c \,\lambda^*_W\}, \; \mbox{when} \; (u-u') \ge 16e^2\,c_2 \; \dis\frac{{\rm cap}(W)}{L^\nu_{n+1}} \;u'\,.
\end{equation}

\n 
We will now select the set $W$. We want on the one hand ${\rm cap}(W)$ large, to take advantage of the left-hand inequality in (\ref{2.54}), and on the other hand not too large, so that we can pick $u$,$u'$ close enough (in our renormalization scheme, the sequence of levels will need to converge). 

\medskip
Due to (\ref{1.26}) and (\ref{2.11}), we have
\begin{equation}\label{2.55}
{\rm cap}(V) \le \wt{c} \;2^{n+1} \,L_0^\nu \,,
\end{equation}
as well as
\begin{equation}\label{2.56}
{\rm cap}(\wt{B}_1 \cup \wt{B}_2) \ge\ov{c} \,L^\nu_{n+1} = \ov{c}\,\ell_0^{(n+1)\nu} L_0^\nu \,.
\end{equation}

\medskip\n
Observe that for $A \subset \subset E$, $x \in E$, one has, cf.~(\ref{1.22})
\begin{equation}\label{2.57}
{\rm cap}(A) \le {\rm cap}(A \cup \{x\}) \le {\rm cap}(A) + c_*\,.
\end{equation}

\n
Thus when $\ell_0 \ge c(K,\nu^\prime)$, one can find $W$ such that $V \subseteq W \subseteq \wt{B}_1 \cup \wt{B}_2$ and
\begin{equation}\label{2.58}
\fr \; \mbox{\f $\dis\frac{\sqrt{K}}{(n+1)^{3/2}}$} \;L^\nu_n \, \ell_0^{\nu^{\prime\prime}} \le {\rm cap}(W) \le 2 \; \mbox{\f $\dis\frac{\sqrt{K}}{(n+1)^{3/2}}$} \;L_n^\nu \, \ell_0^{\nu^{\prime\prime}}, \;\mbox{where} \; \nu^{\prime\prime} = \dis\frac{\nu + \nu^\prime}{2}\;.
\end{equation}

\medskip\n
Indeed, when $\ell_0 \ge c(K,\nu ')$, with the same constants as in (\ref{2.55}) - (\ref{2.57}), one has for all $n \ge 0$, $\wt{c} \,2^{n+1} \,L_0^\nu \le \frac{\sqrt{K}}{(n+1)^{3/2}} \;\ell_0^{n \nu} \,L_0^\nu \,\ell_0^{\nu^{\prime\prime}} < \ov{c} \, \ell_0^{(n+1)\nu} \, L_0^\nu$, as well as $\frac{\sqrt{K}}{(n+1)^{3/2}} \,\ell_0^{n \nu} \,L_0^\nu \,\ell_0^{\nu^{\prime\prime}} > 2c_*$.

\medskip
So when $\ell_0 \ge c(K,\nu ')$, with the above choice of $W$ and (\ref{2.54}), we have shown that 
\begin{equation}\label{2.59}
\begin{array}{l}
\mbox{when $u-u' \ge c_1 \,\sqrt{K}/(n+1)^{3/2} \,\ell_0^{-(\nu - \nu^{\prime\prime})}u'$, then}
\\[2ex]
\ov{\IP}[N_1^* < N'] \le 2 \exp\Big\{ - c(u-u') \, \mbox{\f $\dis\frac{\sqrt{K}}{(n+1)^{3/2}}$} \;L^\nu_n \, \ell_0^{\nu^{\prime\prime}}\Big\} 
\\[2ex]
\qquad \qquad \quad\; \le 2 \exp\Big\{ - 2 u' \; \mbox{\f $\dis\frac{K}{(n+1)^3}$} \;L^\nu_n \, \ell^{\nu^\prime}_0\Big\}\,,
\\[3ex]
\mbox{(since $2 \nu^{\prime\prime} - \nu = \nu '$ and choosing $c_1$ large enough)}\,.
\end{array}
\end{equation}

\n
We now introduce the random finite subsets of $V(\subseteq W)$:
\begin{equation} \label{2.60}
\begin{array}{l}
\cI '_\ell =  V \cap \Big( \bigcup\limits_{(w_1,\dots,w_\ell) \in {\rm Supp} \,\wt{\zeta}^{'}_\ell} 
{\rm range} (w_1) \cup \dots \cup {\rm range} (w_\ell)\Big), \;\mbox{for $\ell \ge 1$}\,,
\end{array}
\end{equation}
\begin{equation}\label{2.61}
\hspace{-5cm}\begin{array}{l}
\cI^*_1  = V \cap \Big(\bigcup\limits_{w \in {\rm Supp} \, \wt{\zeta}^{*}_1} 
{\rm range} (w)\Big)\,. 
\end{array}
\end{equation}
By (\ref{2.23}), (\ref{2.24}) we see that
\begin{equation}\label{2.62}
\begin{array}{l}
\mbox{the random subsets $\cI^{'}_\ell, \, \ell \ge 1, \, \cI^*_1$ are independent,}
\\[1ex]
\cI^{u'} \cap V = \bigcup\limits_{\ell \ge 1} \;  \cI '_\ell, \; \mbox{and} \; \; \cI^u \cap V \supseteq \cI^*_1\,.
\end{array}
\end{equation}
We also note that
\begin{equation*}
\hspace{-1.2cm}\begin{array}{l}
\cI^u \cap \wh{C}_1 = \wh{C}_1 \cap  \Big( \bigcup\limits_{w \in {\rm Supp} \, (1_{\{H_{\wh{C}_1} < \infty\} \mu_{W,u}})} {\rm range} (w)\Big)\,,
\end{array}
\end{equation*}
and that
\begin{equation*}
\begin{array}{l}
(\cI^\prime_1 \cup \cI^*_1) \cap \wh{C}_2 = \wh{C}_2 \cap  \Big(\bigcup\limits_{w \in {\rm Supp} \, (1_{\{H_{\wh{C}_1} =\infty\}(\zeta '_1 +\zeta^*_1)})} 
{\rm range} (w)\Big)\,.
\end{array}
\end{equation*}
By (\ref{2.20}) we thus see that
\begin{equation}\label{2.63}
\cI^u \cap \wh{C}_1 \;\mbox{and} \; (\cI^\prime_1 \cup \cI^*_1) \cap \wh{C}_2 \; \mbox{are independent}\,.
\end{equation}

\medskip\n
By even easier arguments we also see that
\begin{equation}\label{2.64}
(\cI_1'  \cup \cI^*_1) \cap \wh{C}_1 \; \mbox{and} \;(\cI_1' \cup \cI_1^*) \cap \wh{C}_2 
\; \mbox{are independent}\,.
\end{equation}

\medskip\n
We now prove our main claims (\ref{2.8}), (\ref{2.9}). We recall the notation from (\ref{1.35}). Since the $A_m$, $m \in I_{(n+1)}$, are decreasing and $\cT$-adapted, cf.~(\ref{2.5}), we see that
\begin{equation}\label{2.65}
\begin{array}{l}
\IP \Big[\bigcap\limits_{m \in T_{(n+1)}} A^u_m\Big] \stackrel{(\ref{2.6})}{=} \IP \Big[\bigcap\limits_{\ov{m}_1 \in T_{(n)}} A^u_{\ov{m}_1,1} \cap \bigcap\limits_{\ov{m}_2 \in T_{(n)}} A^u_{\ov{m}_2,2} \Big] \le
\\[3ex]
\IP \Big[\bigcap\limits_{\ov{m}_1 \in T_{(n)}} A_{\ov{m}_1,1}(\cI^u \
 \cap \wh{C}_1) \cap  \bigcap\limits_{\ov{m}_2 \in T_{(n)}} A_{\ov{m}_2,2} ((\cI '_1 \cup \cI^*_1) \cap \wh{C}_2) \Big]  \stackrel{(\ref{2.63})}{=}
 \\[3ex]
 \IP \Big[\bigcap\limits_{\ov{m}_1 \in T_{(n)}} A^u_{\ov{m}_1,1}\Big] \,\IP  \Big[\bigcap\limits_{\ov{m}_2 \in T_{(n)}} A_{\ov{m}_2,2} (\cI '_1 \cup \cI^*_1)\Big]\,.
\end{array}
\end{equation}
Note that by (\ref{2.42}), (\ref{2.60}) we have
\begin{equation}\label{2.66}
\underset{\ell \ge 2}{\mbox{\f $\dis\bigcup$}}
\, \cI '_\ell = V \cap \Big( \underset{{w \in {\rm Supp} \,\wt{\sigma} '}}{\mbox{\f $\dis\bigcup$}}
{\rm range}(w)\Big)\,,
\end{equation}

\n
and it follows from Lemma \ref{lem2.4} and (\ref{2.62}) that
\begin{equation}\label{2.67}
\begin{array}{l}
\IP \Big[ \bigcap\limits_{\ov{m}_2 \in T_{(n)}} A_{\ov{m}_2,2} (\cI '_1 \cup \cI^*_1)\Big] \le \IP  \Big[ \bigcap\limits_{\ov{m}_2 \in T_{(n)}} A_{\ov{m}_2,2} \Big(\bigcup\limits_{\ell \ge 1}\,\cI '_\ell \Big)\Big] + \ov{\IP} [N^*_1 < N']  \stackrel{(\ref{2.59}),(\ref{2.62})}{\le}
 \\[3ex]
\IP \Big[ \bigcap\limits_{\ov{m}_2 \in T_{(n)}} A^{u'}_{\ov{m}_2,2} \Big] + 2 \exp\Big\{ -2 u' \;\mbox{\f $\dis\frac{K}{(n+1)^3}$} \;L^\nu_n \, \ell_0^{\nu '}\Big\},
\\[2ex]
\mbox{when} \; u-u' \ge c_1 \;\mbox{\f $\dis\frac{\sqrt{K}}{(n+1)^{3/2}}$} \;\ell_0^{-\frac{(\nu - \nu ')}{2}} u'\,.
\end{array}
\end{equation}

\n
Inserting this inequality in (\ref{2.65}) yields (\ref{2.8}) (and of course (\ref{2.8'}) as well). We now turn to the proof of (\ref{2.9}). Since the $B_m$, $m \in T_{(n+1)}$, are increasing $\cT$-adapted, we find with Lemma \ref{lem2.4} that
\begin{equation}\label{2.68}
\begin{array}{l}
\IP \Big[\bigcap\limits_{m \in T_{(n+1)}} B^{u^\prime}_m\Big] \stackrel{(\ref{2.62})}{=} \IP \Big[\bigcap\limits_{m \in T_{(n+1)}} B_m \Big(\bigcup\limits_{\ell \ge 1} \, \cI '_\ell\Big)\Big] \underset{{\rm Lemma}\,\ref{lem2.4}}{\stackrel{(\ref{2.62}), (\ref{2.66})}{\le}}
\\[3ex]
\IP \Big[\bigcap\limits_{m \in T_{(n+1)}} B_m(\cI^{'}_1  \cup \cI^{*}_1)\Big] + \ov{\IP} [N^*_1 < N'] \stackrel{(\ref{2.64})}{=}
\\[3ex]
\IP\Big[ \bigcap\limits_{\ov{m}_1 \in T_{(n)}} B_{\ov{m}_1,1} (\cI '_1 \cup \cI^*_1) \Big] \;\IP\Big[ \bigcap\limits_{\ov{m}_2 \in T_{(n)}} B_{\ov{m}_2,2} (\cI'_1 \cup \cI_1^*)\Big] +  \ov{\IP} [N^*_1 < N'] \stackrel{(\ref{2.62})}{\le}
 \\[3ex]
 \IP \Big[\bigcap\limits_{\ov{m}_1 \in T_{(n)}} B^u_{\ov{m}_1,1}\Big] \,\IP  \Big[\bigcap\limits_{\ov{m}_2 \in T_{(n)}} B^u_{\ov{m}_2,2}\Big] +  \ov{\IP} [N^*_1 < N]\,.
\end{array}
\end{equation}

\n
The claim (\ref{2.9}) now follows from (\ref{2.59}), and this completes the proof of Theorem \ref{theo2.1}. \hfill $\square$

\begin{remark}\label{rem2.5} \rm ~

\medskip\n
1) With the help of Lemma \ref{lem2.3} and (\ref{2.36}), one can also derive a companion lower bound to (\ref{2.33}) or (\ref{2.40}) showing that for $\ell \ge 2$, 
\begin{equation*}
\wt{\xi}^\prime_\ell \ge u' \,{\rm cap}(W) \Big(c \dis\frac{{\rm cap}(W)}{L^\nu_{n+1}}\Big)^{\ell - 1} \, \Gamma^{\otimes \ell} \,.
\end{equation*}

\n
Combined with (\ref{2.40}) we see that the total mass of $\wt{\sigma} '$, see (\ref{2.42}), has finite expectation, but that large enough exponential moments of the total mass of $\wt{\sigma} '$ are divergent, (this is very much in line with the identity below (\ref{2.51})).

\medskip
As a result one cannot hope to find a coupling of $\wt{\sigma} '$ and $\wt{\zeta}^*_1$ such that $\wt{\zeta}^*_1$ globally dominates $\wt{\sigma} '$, see Lemma \ref{lem2.4}.

\medskip 
Incidentally in the construction of the coupling from Lemma \ref{lem2.4}, one actually has flexibility in the choice of the joint distribution of $N'$ and $N^*_1$. One might possibly take advantage of this feature to improve the bound (\ref{2.49}), and as a result improve the quality of the remainder terms in (\ref{2.8}), (\ref{2.9}).

\medskip\n
2) The renormalization step conducted above differs from what was done in \cite{Szni10a}, \cite{SidoSzni09a}, \cite{SidoSzni09b}. In essence, the sprinkling technique we have employed here, aims at dominating for the most part, with the Poisson variable $N^*_1$, the ``long range interaction terms ", which are accounted for in the variable $N'$ of (\ref{2.45}). The variable $N'$ has a compound Poisson distribution, where the compounding law is supported on the set of integers at least equal to $2$, see below (\ref{2.51}). Quite naturally, the quality of our bounds deteriorates when the parameter of $N^*_1$ becomes small. This explains why in place of choosing $W=V$ in (\ref{2.54}), we instead pick $W$ rather big, at the expense of assuming $u-u'$ not too small. 

\medskip\n
3) Althoug Theorem \ref{theo2.1} will be general enough for our purpose here, let us note that the crucial Lemma \ref{lem2.4} offers a domination statement that goes beyond asserting that $\bigcup_{\ell \ge 2} \cI'_\ell$ is dominated with ``high probability'' by $\cI^*_1$. For instance the proof of Theorem \ref{theo2.1} yields similar inequalities as (\ref{2.8}), (\ref{2.9}) for decreasing events $A_m$, $m \in T_{(n+1)}$, or increasing events $B_m$, $m \in T_{(n+1)}$, pertaining to the occupancy of both sets of points and edges contained in $\wt{C}_{m,\cT}$. In place of (\ref{1.36}) one naturally defines in this context $A^u_m$, or $B^u_m$, by the consideration of the set of points visited and edges crossed by the trajectories of the interlacement point process with labels at most $u$. The proof of Theorem \ref{theo2.1} then yields 

\bigskip\n
{\bf Corollary 2.1'.} {\it ($K > 0$, $0 < \nu' < \nu$)

\medskip
When $\ell_0 \ge c(K,\nu')$, for all $n \ge 0$, $\cT \in \Lambda_{n+1}$, for all in the above generalized sense $\cT$-adapted collections $A_m$, $m \in T_{(n+1)}$, of decreasing events, respectively $B_m$, $m \in T_{(n+1)}$, of increasing events, and for all $0 < u' < u$ satisfying (\ref{2.7}), the corresponding inequalities to (\ref{2.8}), (\ref{2.9}) hold.}

\vspace{-1ex}
\hfill $\square$

\end{remark}

\medskip
We will now derive the key decoupling inequalities. Given $K > 0$, $0 < \nu ' < \nu$, and $\ell_0 \ge c(K,\nu ')$ such that Theorem \ref{theo2.1} applies we define for any $u_0 > 0$ the increasing and decreasing sequences of levels
\begin{equation}\label{2.69}
\begin{array}{l}
u^+_n  = \prod\limits_{0 \le k < n} 
\Big(1 + c_1 \,\mbox{\f $\dis\frac{\sqrt{K}}{(k+1)^{3/2}}$} \;\ell_0^{-\frac{(\nu - \nu^\prime)}{2}}\Big) \,u_0, 
\\[3ex]
u^-_n  =\prod\limits_{0 \le k < n} 
\Big(1 + c_1 \,\mbox{\f $\dis\frac{\sqrt{K}}{(k+1)^{3/2}}$} \;\ell_0^{-\frac{(\nu - \nu^\prime)}{2}}\Big)^{-1} \,u_0,
\end{array}
\end{equation}

\n
so that $u^+_n$, $n \ge 0$, is increasing, $u^-_n$, $n \ge 0$, is decreasing, and they satisfy 
\begin{equation*}
u^+_{n+1} - u^+_n = c_1 \,\mbox{\f $\dis\frac{\sqrt{K}}{(n+1)^{3/2}}$} \;\ell_0^{-\frac{(\nu - \nu^\prime)}{2}}\; u^+_n,  \;\; \mbox{and} \;\;   u^-_n - u^-_{n+1}  = c_1 \,\mbox{\f $\dis\frac{\sqrt{K}}{(n+1)^{3/2}}$} \;\ell_0^{-\frac{(\nu - \nu^\prime)}{2}} \,u_{n+1}^-\, ,
\end{equation*}
for $n \ge 0$, as well as $u^+_0 = u_0 = u^-_0$.

\medskip
These sequences also have positive finite limits respectively equal to
\begin{equation}\label{2.70}
\begin{array}{l}
u^+_\infty  = \prod\limits_{k \ge 0} 
\Big(1 + c_1 \,\mbox{\f $\dis\frac{\sqrt{K}}{(k+1)^{3/2}}$} \;\ell_0^{-\frac{(\nu - \nu^\prime)}{2}}\Big) \,u_0, 
\\[3ex]
u^-_\infty  = \prod\limits_{k \ge 0} 
\Big(1 + c_1 \,\mbox{\f $\dis\frac{\sqrt{K}}{(k+1)^{3/2}}$} \;\ell_0^{-\frac{(\nu - \nu^\prime)}{2}}\Big)^{-1} \,u_0,
\end{array}
\end{equation}

\n
We can now state and prove the main result of this section.

\begin{theorem}\label{theo2.6} (Decoupling Inequalities; $K > 0$, $0 < \nu^\prime < \nu$, $\ell_0 \ge c(K,\nu^\prime)$)

\medskip
For any $u_0 > 0$, $n \ge 0$, $\cT \in \Lambda_n$, and all $\cT$-adapted collections $A_m$, $m \in T_{(n)}$, of decreasing events on $\{0,1\}^E$, respectively $B_m$, $m \in T_{(n)}$, of increasing events on $\{0,1\}^E$, one has:
\begin{equation}\label{2.71}
\begin{array}{l}
\IP\Big[\bigcap\limits_{m \in T_{(n)}}\,
A_m^{u^+_\infty}\Big]  \le \IP \Big[\bigcap\limits_{m \in T_{(n)}} \,A_m^{u^+_n}\Big] \le \prod\limits_{m \in T_{(n)}}\, (\IP[A_m^{u_0}] + \ve(u_0))\,,
\end{array}
\end{equation}
\begin{equation}\label{2.72}
\begin{array}{l}
\;\; \IP\Big[\bigcap\limits_{m \in T_{(n)}}\, B_m^{u^-_\infty}\Big]  \le \IP \Big[\bigcap\limits_{m \in T_{(n)}}\, B_m^{u^-_n}\Big] \le \prod\limits_{m \in T_{(n)}}\, (\IP[B_m^{u_0}] + \ve(u_\infty^-))\,, 
\end{array}
\end{equation}
where we have set
\begin{equation}\label{2.73}
\ve(u) =  \dis\frac{2e^{-K u L_0^\nu \,\ell_0^{\nu^\prime}}}{1 - e^{-K u \,L_0^\nu\,\ell_0^{\nu^\prime}}}, \;\; \mbox{for $u > 0$, (note that  $v > 0 \rightarrow \frac{e^{-v}}{1-e^{-v}}$ is decreasing)}\,.
\end{equation}
\end{theorem}

\begin{proof}
We begin with the proof of (\ref{2.72}). The first inequality is immediate since $u^-_\infty \le u^-_n$, and the events $B_m$ are increasing. As for the second inequality, we first prove by induction on $n$ that
\begin{equation}\label{2.74}
\begin{array}{l}
\IP\Big[\bigcap\limits_{m \in T_{(n)}} \,B_m^{u^-_n}\Big]  \le \prod\limits_{m \in T_{(n)}}\,\big(\IP[B_m^{u_0}]  + \dsl_{0 \le k  < n} \,\big(2 \,e^{-\frac{2K}{(k+1)^3}\,u^-_{k+1}\,L_k^\nu \, \ell_0^{\nu^\prime}}\big)^{\frac{1}{2^{k+1}}}\big)\,.
\end{array}
\end{equation}

\n
The claim trivially holds when $n=0$. If it is true for $n$, we find by (\ref{2.9}) that for $\cT \in \Lambda_{n+1}$, and $B_m$, $m \in T_{(n+1)}$, an increasing $\cT$-adapted collection, one has, with hopefully obvious notation:
\begin{equation}\label{2.75}
\begin{array}{l}
\IP\Big[\bigcap\limits_{m \in T_{(n+1)}} B_m^{u^-_{n+1}}\Big]  \le \IP\Big[\bigcap\limits_{\ov{m}_1 \in T_{(n)}} B_{\ov{m}_1,1}^{u^-_n}\Big] \; \IP\Big[\bigcap\limits_{\ov{m}_2 \in T_{(n)}} B_{\ov{m}_2,2}^{u^-_n}\Big] + 2 \,e^{-\frac{2K}{(n+1)^3}\,u^-_{n+1}\,L_n^\nu \, \ell_0^{\nu^\prime}}
\\[3ex]
\underset{\rm hypothesis}{\stackrel{\rm induction}{\le}} \prod\limits_{m \in T_{(n+1)}} \big(\IP[B_m^{u_0}] + \dsl_{0 \le k < n} \big(2 \,e^{-\frac{2K}{(k+1)^3}\,u^-_{k+1}\,L_k^\nu \, \ell_0^{\nu^\prime}}\big)^{\frac{1}{2^{k+1}}}\big) + 2 \,e^{-\frac{2K}{(n+1)^3}\,u^-_{n+1}\,L_n^\nu \, \ell_0^{\nu^\prime}}
\\[3ex]
\le \prod\limits_{m \in T_{(n+1)}} \big(\IP[B_m^{u_0}] + \dsl_{0 \le k < n+1} \big(2 \,e^{-\frac{2K}{(k+1)^3}\,u^-_{k+1}\,L_k^\nu \, \ell_0^{\nu^\prime}}\big)^{\frac{1}{2^{k+1}}}\big)\,,
\end{array}
\end{equation}
and this completes the proof of (\ref{2.74}) by induction.

\pagebreak
To finish the proof of (\ref{2.72}), we now observe that when $\ell_0 \ge c(K,\nu ')$, then for all $k \ge 0$, $(\frac{\ell_0^\nu}{2})^k \ge 2^{4k} \ge (k+1)^4$. Hence the series in the product in the right-hand side of (\ref{2.74}) is smaller than
\begin{equation*}
\dsl_{k \ge 0} \,2 \,e^{-\frac{K}{(k+1)^3} \;u_\infty^- (\frac{\ell_0^\nu}{2})^k\, L_0^\nu \,\ell_0^{\nu^\prime}} \le 2 \dsl_{k \ge 0} \,e^{- (k+1)\, K \,u_\infty^- \, L_0^\nu \,\ell_0^{\nu^\prime}} = \ve (u^-_\infty)\,,
\end{equation*}

\n
and this completes the proof of (\ref{2.72}).

\medskip
The proof of (\ref{2.71}) is analogous. Instead of (\ref{2.74}), one shows by induction on $n$, with the help of (\ref{2.8'}), that
\begin{equation}\label{2.76}
\begin{array}{l}
\IP\Big[\bigcap\limits_{m \in T_{(n)}} A_m^{u^+_n}\Big]  \le \prod\limits_{m \in T_{(n)}} \, \Big(\IP [ A^{u_0}_m] + \dsl_{0 \le k < n} \,\big(2 \,e^{-\frac{2K}{(k+1)^3}\,u^+_k\,L_k^\nu \, \ell_0^{\nu^\prime}}\big)^{\frac{1}{2^k}}\Big)\,,
\end{array}
\end{equation}

\medskip\n
and obtains ({2.71}) as a consequence, (in fact the argument below (\ref{2.75}) yields that (\ref{2.71}) even holds with $2K$ in place of $K$ in the definition of $\ve(u_0)$ in (\ref{2.73})).
\end{proof}

\begin{remark}\label{rem2.7} \rm ~

\medskip\n
1) As already mentioned, in the important special case $E = \IZ^{d+1}$, $d \ge 2$, (so that $\alpha = d$, $\beta = 2$), one can replace the distance $d(\cdot,\cdot)$ with the sup-norm distance $d_\infty(\cdot,\cdot)$ and the balls relative to $d(\cdot,\cdot)$ with balls relative to $d_\infty(\cdot,\cdot)$, in the above theorem and of course in the definition of $\Lambda_n$, and of $\cT$-adapted collections when $\cT \in \Lambda_n$.

\bigskip\n
2) As companions to (\ref{2.71}), (\ref{2.72}), the FKG-Inequality, see (\ref{1.34}), implies that 
\begin{equation}\label{2.77}
\begin{array}{l}
\IP\Big[\bigcap\limits_{m \in T_{(n)}} A_m^{u^+_\infty}\Big]  \ge \prod\limits_{m \in T_{(n)}} \,\IP [A_m^{u^+_\infty}]\,,
\end{array}
\end{equation}
and
\begin{equation}\label{2.78}
\begin{array}{l}
\IP\Big[\bigcap\limits_{m \in T_{(n)}} B_m^{u^-_\infty}\Big]  \ge \prod\limits_{m \in T_{(n)}} \,\IP [B_m^{u^-_\infty}]\,.
\end{array}
\end{equation}

\medskip\n
Theorem \ref{theo2.6} provides an upper bound for the expressions in the left-hand side of (\ref{2.77}), (\ref{2.78}). In a sense it offers a partial substitute for the BK Inequality, which plays a key role for Bernoulli percolation, cf.~\cite{Grim99}, but  not for interlacement percolation, so far, see Remark \ref{rem1.5} 3).  
\hfill $\square$

\end{remark}

\section{Cascading events}
\setcounter{equation}{0}

We want to apply the decoupling inequalities of Theorem \ref{theo2.6} to control the probability of certain events on $\{0,1\}^E$, which pertain to the trace of the interlacement $\cI^u$ in a large box. For this purpose the cascading property plays an important role. In essence, it enables us to cover such an event concerning the state of $\cI^u$ in a ball of size of order $L_n$, with a not too large family of events, which come each as intersections of $2^n$ events depending on the respective traces of $\cI^u$ in well-separated balls of size of order $L_0$. The decoupling inequalities of Theorem \ref{theo2.6} will yield upper bounds on the probability of such intersections. These bounds will compete against the combinatorial complexity of the family of dyadic tree embeddings used to cover the original event. In Proposition \ref{prop3.2} we present two examples of cascading families of events, which play an important role in Section 4 and 5. We discuss in Remark \ref{rem3.3} another example in the case of $\IZ^{d+1}$, $d \ge 2$, which is due to Teixeira \cite{Teix10}, as well as a modification of this example adapted to the general context of the present work. The main consequences of the cascading property and the decoupling inequalities appear in Theorem \ref{theo3.4}, as well as Corollary \ref{cor3.5} and  \ref{cor3.7}. In the special case $E = \IZ^{d+1}$, $d \ge 2$, the distance $d(\cdot,\cdot)$ and the balls $B(x,r)$ can be replaced with the sup-norm distance $d_\infty(\cdot,\cdot)$ and the corresponding balls $B_\infty(\cdot,\cdot)$. 

\smallskip
\begin{definition}\label{def3.1}
A family $\cG = (G_{x,L})_{x \in E, L \ge 1 \;{\rm integer}}$ of events on $\{0,1\}^E$ has the cascading property (or cascades) with complexity at most $\lambda > 0$, when
\begin{equation}\label{3.1}
\mbox{$G_{x,L}$ is $\sigma(\Psi_{x'}, x' \in B(x,10L))$-measurable for each $x \in E$, $L \ge 1$}\,,
\end{equation}

\smallskip\n
(the notation $\Psi_x$ is defined above (\ref{1.34})), 

\medskip\n
and for each $\ell$ multiple of $100$, $x \in E$, $L \ge 1$, there exists  $\Lambda \subseteq E$, such that
\begin{align}
& \Lambda \subseteq B(x,9\ell L), \label{3.2}
\\[2ex]
& |\Lambda | \le c(\cG, \lambda) \, \ell^\lambda \,, \label{3.3}
\\[2ex]
& G_{x,\ell L} \subseteq \underset{x',x'' \in \Lambda; \,d(x',x'') \ge \frac{\ell}{100}\,L}{\mbox{\f $\dis\bigcup$}} \;G_{x',L} \cap G_{x'',L}\,.\label{3.4}
\end{align}
\end{definition}

\medskip
When the family of events depends on an additional parameter, we say that it has the {\it uniform cascading property} (or {\it cascades uniformly} ) {\it with complexity at most} $\lambda$, when for each fixed value of the parameter, the cascading property with complexity at most $\lambda$ holds, and in addition the constant in (\ref{3.3}) can be chosen uniformly in the parameter.

\medskip
We will now provide examples of such families, see Proposition \ref{prop3.2} and Remark \ref{3.3} below. The first two examples play an important role, respectively in Section 4 and 5. The first example corresponds to the family $\cA = (A_{x,L})_{x \in E, L \ge 1\;{\rm integer}}$, where
\begin{equation}\label{3.6}
\begin{split}
A_{x,L} = \big\{ &\, \mbox{$\sigma \in \{0,1\}^E; \; B(x,L)$ is linked to $\partial_{\rm int} \, B(x,2L)$}
\\
&\mbox{by a path where $\sigma$ vanishes\big\}}.
\end{split}
\end{equation}
In particular, we see that in the notation of (\ref{0.10}) and (\ref{1.36}), for $u \ge 0$, 
\begin{equation}\label{3.7}
A^u_{x,L} \stackrel{(\ref{1.36})}{=} A_{x,L} (\cI^u) = \{B(x,L) \stackrel{\cV^u}{\longleftrightarrow} \partial_{\rm int} B(x,2L)\}\,.
\end{equation}

\medskip\n
To describe the second example, we consider the family of ``half-planes'' $\cP$ in $E$ of the form
\begin{equation}\label{3.8}
\cP = \{y(n); n \ge 0\} \times \IZ \subseteq E\,,
\end{equation}

\n
where $y(n)$, $n \ge 0$, is a semi-infinite geodesic in $G$ (recall $E = G \times \IZ)$, that is:
\begin{equation}\label{3.9}
d_G (y(n),y(m)) = | n-m|, \;\mbox{for all} \; n,m \ge 0\,.
\end{equation}

\n
It is well-known, see for instance Theorem 3.1 of \cite{Watk86}, that for each $y \in G$, one can find such a $y(n), n \ge 0$, with $y(0) = y$. Given such a half-plane $\cP$ we say that a finite sequence $x_0,\dots , x_N$ in $\cP$ is a $*$-path, when for each $0 \le i < N$, $x_i \not= x_{i+1}$, and the $G$-projections of $x_i$ and $x_{i+1}$ lie at $d_G$-distance at most $1$, and a similar condition holds for the $\IZ$-projections.

\medskip
We then define the family of events depending on the additional parameter $\cP$ varying over all possible ``half-planes'' in $E$, $\cB = (B_{x,L,\cP})$, where:
\begin{equation}\label{3.10}
\begin{split}
B_{x,L,\cP} & = \big\{\mbox{$\sigma \in \{0,1\}^E; \; B(x,L)$ is linked to $\partial_{\rm int} \, B(x,2L)$ by a $*$-path in $\cP$,}
\\
&\qquad \qquad \qquad \qquad  \! \mbox{where $\sigma$ takes the value $1\big\}$, when $x \in \cP$},
\\[1ex]
& = \phi, \;\mbox{when $x \notin \cP$}\,.
\end{split}
\end{equation}

\n
In particular when $x \in \cP$, $B^u_{x,L,\cP}$ coincides with the event that appears in (\ref{0.13}).

\begin{proposition}\label{prop3.2}
\begin{align}
&\mbox{$\cA$ is a family of decreasing events on $\{0,1\}^E$ that cascades with} \label{3.11} 
\\[-0.5ex]
&\mbox{complexity at most $\alpha + \frac{\beta}{2}$}\,. \nonumber
\\[2ex]
&\mbox{$\cB$ is a family of increasing events on $\{0,1\}^E$ that cascades uniformly with} \label{3.12} 
\\[-0.5ex]
&\mbox{complexity at most  $\frac{\beta}{2}$}\,. \nonumber
\end{align}
\end{proposition}

\begin{proof}
We begin with the proof of (\ref{3.11}). The events $A_{x,L}$ are clearly decreasing and $\sigma(\Psi_{x'}$, $x' \in B(x,10L))$-measurable. Then note that given $x \in E$, $L \ge 1$, and $\ell$ multiple of $100$, one can find $x_1^i$, $1 \le i \le N_1$, in $\partial_{\rm int} B(x,\ell L)$ and $x_2^j$, $1 \le j \le N_2$, in $\partial_{\rm int} B(x,\frac{3}{2} \,\ell N)$, such that
\begin{equation}\label{3.13}
\begin{array}{l}
 N_1 \vee N_2 \le c \,\ell^{\alpha + \frac{\beta}{2}}, 
\end{array}
\end{equation}

\vspace{-5ex}
\begin{equation}\label{3.14}
\begin{array}{l}
 \partial_{\rm int} \,B(x,\ell L) \subseteq\bigcup\limits_{1 \le i \le N_1} B(x_1^i,L), \;\mbox{and} \;\partial_{\rm int} B\Big(x, \mbox{\f $\dis\frac{3}{2}$} \; \ell L\Big) \subseteq \bigcup\limits_{1 \le j \le N_2} \,B(x^j_2,L)\,.  
\end{array}
\end{equation}

\n
Indeed we first select a maximal collection in $\partial_{\rm int} B(x,\ell L)$ with mutual distance bigger than $L$, $x_1^i$, $1 \le i \le N_1$. The balls $B(x_1^i, \frac{L}{2})$ are thus pairwise disjoint, and each have, by (\ref{1.11}), volume at least $c \, L^{\alpha + \frac{\beta}{2}}$. Their union is contained in $B(x,2 \ell L)$ and thus has volume at most $c \,\ell^{\alpha + \frac{\beta}{2}} \,L^{\alpha + \frac{\beta}{2}}$, using (\ref{1.11}) once again. As a result $N_1 \le c \,\ell^{\alpha + \frac{\beta}{2}}$. In a similar fashion we can construct $x_2^j$, $1 \le j \le N_2$, so that (\ref{3.13}), (\ref{3.14}) holds. Note that by construction one has
\begin{equation}\label{3.15}
d(x^i_1, x^j_2) \ge \mbox{\f $\dis\frac{\ell}{2}$} \;L - 1 > \mbox{\f $\dis\frac{\ell}{5}$} \; L, \;\mbox{for} \; 1 \le i \le N_1, \, 1 \le j \le N_2\,,
\end{equation}

\n
and defining $\Lambda$ as the collection of points $x_1^i$ and $x_2^j$, we see that (\ref{3.2}), (\ref{3.3}), hold with $\lambda = \alpha + \frac{\beta}{2}$.

\medskip
Finally observe that any path from $B(x, \ell L)$ to $\partial_{\rm int} B(x,2 \ell L)$ must visit $\partial_{\rm int} B(x, \ell L)$ and thus enter one of the $B(x^i_1,L)$, then leave $B(x^i_1,2L)$, then visit $\partial_{\rm int} B(x, \frac{3}{2} \, \ell L)$ and thus enter one of the $B(x_2^j,L)$, and then leave $B(x^j_2, 2 L)$ before reaching $\partial_{\rm int} B(x,2 \ell L)$. This shows that
\begin{equation}\label{3.16}
\begin{array}{l}
A_{x_*,\ell L} \subseteq \bigcup\limits_{1 \le i \le N_1 \atop 1 \le j \le N_2} A_{x^i_1,L} \cap A_{x^i_2,L}\,,
\end{array}
\end{equation}
and due to (\ref{3.15}) thus completes the proof of (\ref{3.11}).

\pagebreak
Let us now turn to the proof (\ref{3.12}). The events $B_{x,L,\cP}$ are clearly increasing and $\sigma(\Psi_{x'},x' \in B(x,10L))$-measurable. When $x \in \cP$, $R \ge 1$, then $B(x,R) \cap \cP$ is a rectangle with horizontal sides of length comparable to $R$ up to a multiplicative constant, and vertical sides of height comparable to $R^{\frac{\beta}{2}}$ up to a multiplicative constant. The set $\partial^\cP_{\rm int} B(x,R)$ of vertices of $B(x,R) \cap \cP$ that neighbour $\cP \backslash B(x,R)$ is the union of two horizontal and one or two vertical sides of the rectangle $B(x,R) \cap \cP$. We then proceed as in the proof of (\ref{3.11}), in essence replacing $\partial_{\rm int} B(x,\ell L)$ and $\partial_{\rm int} B(x,\frac{3}{2} \, \ell L)$ by $\partial_{\rm int}^\cP \,B(x,\ell L)$ and $\partial^\cP_{\rm int} B(x,\frac{3}{2} \, \ell L)$. In a similar fashion we find $x^i_1$, $1 \le i \le N_1$, in $\partial^\cP_{\rm int} B(x,\ell L)$ and $x^j_2$, $1 \le j \le N_2$, in $\partial_{\rm int}^\cP\, B(x,\frac{3}{2} \, \ell L)$, such that
\begin{align}
& N_1 \vee N_2 \le c \,\ell^{\frac{\beta}{2}}, \label{3.17}
\\[2ex]
& \partial^\cP_{\rm int} \,B(x,\ell L) \subseteq \bigcup\limits_{1 \le i \le N_1} B(x_1^i,L),  \; \partial^\cP_{\rm int} B\Big(x, \mbox{\f $\dis\frac{3 \ell}{2}$} \;  L\Big) \subseteq \bigcup\limits_{1 \le j \le N_2} B(x^j_2,L)\,. \label{3.18}
\end{align}

\n
On the other hand when $x \notin \cP$, (and $B_{x,\ell L, \cP} = \phi$), we simply choose $N_1 = N_2 = 1$, and $x_1^i \in \partial_{\rm int} B(x, \ell L)$, $x^j_1 \in  \partial_{\rm int} B(x,\frac{3}{2} \, \ell L)$. Similar arguments as for the proof of (\ref{3.11}) show that (\ref{3.12}) holds. 
\end{proof}

\begin{remark}\label{rem3.3} \rm ~

\medskip\n
1) In the case $E = \IZ^{d+1}$, $d \ge 2$, $\alpha + \frac{\beta}{2} = d+1$ in (\ref{3.11}). But one easily sees that in fact $\cA$ cascades with complexity at most $d$. This fact was for instance implicitly used in the proof of Lemma 1.2 of \cite{SidoSzni09b}. 

\bigskip\n
2) In the case $E = \IZ^{d+1}$, $d \ge 2$, Teixeira introduced in \cite{Teix10} a  very interesting family of events having the cascading property, which we briefly describe below. We now use the sup-norm distance $d_\infty(\cdot,\cdot)$ in place of $d(\cdot,\cdot)$. For $x \in \IZ^{d+1}$, and $L \ge 1$ integer, we define the separation event:
\begin{equation}\label{3.19}
\begin{split}
\wt{S}_{x,L} = \big\{ & \sigma \in \{0,1\}^{\IZ^{d+1}}; \; \mbox{there exist two connected sets} 
\\
&\mbox{$A_1, A_2 \subset x + [-L,2L)^{d+1}$, with diameter at least $\mbox{\f $\frac{L}{2}$}$,} 
\\
&\mbox{separated by $\sigma$ in $x + [-2 L, 3L)^{d+1}\big\}$}\,,  
\end{split}
\end{equation}

\medskip\n
where the expression ``$\sigma$ separates $A_1, A_2$ in $x + [-2L, 3L)^{d+1}$'' means that the mutual graph-distance between $A_1$, and $A_2$ exceeds $1$, i.e.~$d_{\IZ^{d+1}}(A_1$, $A_2) > 1$, and that any path in $x + [-2L,3L)^{d+1}$ from $\partial A_1$ to $\partial A_2$ meets the set $\Sigma(\sigma) = \{x \in \IZ^{d+1}$, $\sigma (x) = 1\}$.

\medskip
In other words on the complement of the separation event $\wt{S}_{x,L}$, for any two connected subsets $A_1, A_2$ of $x + [-L,2L)^{d+1}$ with diameter at least $L/2$, and mutual graph-distance bigger than $1$, one can find a path from $\partial A_1$ to $\partial A_2$ in $x + [-2L, 3L)^{d+1}$, where $\sigma$ identically vanishes.

\medskip
It then follows from Theorem 5.2 of \cite{Teix10} that
\begin{equation}\label{3.20}
\wt{\cS} \stackrel{\rm def}{=} (\wt{S}_{x,L})_{x \in \IZ^{d+1}, L \ge 1} \; \mbox{cascades with complexity at most $d+1$}\,.
\end{equation}

\medskip\n
3) In analogy with \cite{Teix10}, we introduce in the general context of the present work the collection $\cS = (S_{x,L})_{x \in E, L \ge 1\,{\rm integer}}$ of separation events defined by
\begin{equation}\label{3.20a}
\begin{split}
S_{x,L} = \big\{ & \sigma \in \{0,1\}^E; \;\mbox{there exist connected subsets $A_1$ and $A_2$ of $\ov{B(x,3L)}$}
\\
&\mbox{with $d(\cdot,\cdot)$-diameter at least $L$, separated by $\Sigma(\sigma)$ in $B(x,5L)\big\}$}\,,
\end{split}
\end{equation}

\medskip\n
where $\Sigma(\sigma) = \{x \in E$; $\sigma(x) = 1\}$, and the above separation statement means that the mutual graph-distance $d_E(A_1,A_2)$ is bigger than $1$, and that any path from $\partial A_1$ to $\partial A_2$ in $B(x,5 L)$ meets $\Sigma(\sigma)$.

\medskip
We show in Proposition \ref{propA.2} of the Appendix that
\begin{equation}\label{3.21b}
\begin{array}{l}
\mbox{$\cS$ is a family of increasing events on $\{0,1\}^E$ that cascades with}
\\
\mbox{complexity at most  $\alpha + \frac{\beta}{2}$}.
\end{array}
\end{equation}

\medskip\n
As an aside, let us mention that making further progress on the question whether $u_*>0$, when $\nu<1$,  might possibly involve the application of Corollary  \ref{cor3.7} below, to a suitable variation on the above family $\cS$, see also Remark \ref{rem5.6} 2).

  \hfill $\square$
\end{remark}

We now come to the main result of this section, which combines the decoupling inequalities of Section 2 with the above notion of cascading property. We recall the sequence of length scales introduced in (\ref{2.1}), as well as the notation (\ref{1.36}), and (\ref{2.70}). We implicitly assume that $\ell_0 \ge 10^6 c_0$, see (\ref{2.1}), is divisible by $100$, and that $L_0 \ge 1$ is an integer.

\begin{theorem}\label{theo3.4} $(K > 0, \, 0 < \nu ' < \nu, \, \ell_0 \ge c(K, \nu '), \,L_0 \ge 1, \, \lambda > 0)$

\medskip
Consider $\cG = (G_{x,L})_{x \in E, L \ge 1 \,{\rm integer}}$, a collection of decreasing, resp.~increasing events on $\{0,1\}^E$, cascading with complexity at most $\lambda$. Then for any $n \ge 0$, $u_0 > 0$, one has for decreasing events
\begin{align}
\sup\limits_{x \in E} \, \IP[G_{x,L_n}^{u_\infty^+}] & \le (c(\cG,\lambda)\ell_0^{2 \lambda})^{2^n-1} \big(\sup\limits_{x \in E} \, \IP[G_{x,L_0}^{u_0}]   + \ve(u_0)\big)^{2^n}\,, \label{3.21}
\\[-1ex]
\intertext{resp.~for increasing events}
\sup\limits_{x \in E} \, \IP[G_{x,L_n}^{u_\infty^-}] & \le (c(\cG,\lambda)\ell_0^{2 \lambda})^{2^n-1} \big(\sup\limits_{x \in E} \, \IP[G_{x,L_0}^{u_0}]   + \ve(u_\infty^-)\big)^{2^n}\,, \label{3.22}
\end{align}
with $\ve(\cdot)$ as in (\ref{2.73}).

\medskip
When the family of events depends on a parameter and cascades uniformly with complexity at most $\lambda$, similar inequalities hold as in (\ref{3.21}), (\ref{3.22}), where the supremum appearing on both sides of the inequalities now runs over $x \in E$ and the parameter set.
\end{theorem}

\begin{proof}
It follows from (\ref{3.4}), and induction on $n$ that
\begin{equation}\label{3.23}
\begin{array}{l}
G^u_{x,L_n} \subseteq \bigcup\limits_{\cT \in \Lambda^\cG_n} \; \bigcap\limits_{m \in T_{(n)}} \,G^u_{x_{m,\cT}, L_0}\,,
\end{array}
\end{equation}

\n
where the subset $\Lambda_n^\cG$ of the collection $\Lambda_n$ of embeddings of $T_n$, see above (\ref{2.2}), has cardinality at most
\begin{equation}\label{3.24}
|\Lambda^\cG_n| \le (c(\cG, \lambda)  \ell_0^{\lambda})^2 (c(\cG, \lambda) \ell_0^{\lambda})^4 \dots (c(\cG, \lambda) \ell_0^{\lambda})^{2^n} = (c'(\cG, \lambda)  \ell_0^{2\lambda})^{2^n-1}\,.
\end{equation}

\n
In view of the decoupling inequalities (\ref{2.71}), (\ref{2.72}), our claims (\ref{3.21}), (\ref{3.22}) follow. The extension to families depending on a parameter and having the uniform cascading property is immediate. 
\end{proof}

\medskip
We now derive two corollaries, which we will later apply in Sections 4 and 5.

\begin{corollary}\label{cor3.5}
Consider a family $\cG = (G_{x,L})_{x \in E, L \ge 1 \,{\rm integer}}$ of events on $\{0,1\}^E$ cascading with complexity at most $\lambda > 0$, and $u > 0$ such that
\begin{equation}\label{3.25}
\lim\limits_{\ov{L \r \infty}} \;\sup\limits_{x \in E} \,\IP[G_{x,L}^u] = 0\,.
\end{equation}

\n
If the events in $\cG$ are decreasing, resp.~increasing, then for $\ov{u} > u$, resp.~$\ov{u} < u$, one can find integers $L_0 \ge 1$, $\ell_0 > 1$, such that with $L_n = \ell_0^n \,L_0$,
\begin{equation}\label{3.26}
\sup\limits_{x \in E} \,\IP[G^{\ov{u}}_{x,L_n}] \le 2^{-2^n}, \;\mbox{for all $n \ge 0$}\,.
\end{equation}

\n
When the family depends on a parameter and cascades uniformly with complexity at most $\lambda$, if (\ref{3.25}) holds with a joint supremum over $x$ in $E$ and the parameter set, then (\ref{3.26}) holds with a similar modification.
\end{corollary}

\begin{proof}
We pick $K = 2$ and $\nu ' = \frac{\nu}{2}$ in Theorem \ref{theo3.4}, and from now on assume that $\ell_0 > c = c(K = 2, \nu ' = \frac{\nu}{2}) \ge 10^6 c_0$ is a multiple of 100 such that Theorem \ref{theo3.4} applies. Setting $u_0 = u$,  we further assume that $\ell_0 \ge c(\ov{u}, \cG, \lambda)$ is large enough so that in the case of decreasing events, (cf.~(\ref{2.70}), (\ref{3.21}) for the notation):
\begin{equation}\label{3.27}
\begin{array}{rl}
{\rm i)} &u^+_\infty < \ov{u}, \;\mbox{and}
\\[1ex]
{\rm ii)} & c(\cG, \lambda) \,\ell_0^{2 \lambda} \ve(u) \stackrel{(\ref{2.73})}{\le} c(\cG, \lambda) \ell_0^{2 \lambda} \,2e^{-2 u \ell_0^{\nu '}} / (1 - e^{-2 u \ell_0^{\nu '}}) \le \frvier, \;\mbox{for all $L_0 \ge 1$}\,.
\end{array}
\end{equation}
and that in the case of increasing events:
\begin{equation}\label{3.28}
\hspace{-10ex}\begin{array}{rl}
{\rm i)} &u^-_\infty > \max\Big( \ov{u}, \mbox{\f $\dis\frac{u}{2}$}\Big)  
\\[1ex]
{\rm ii)} & c(\cG, \lambda) \,\ell_0^{2 \lambda} \ve(u^-_\infty)  \le c(\cG, \lambda) \ell_0^{2 \lambda}  \ve \Big(\mbox{\f $\dis\frac{u}{2}$}\Big) \le \frvier, \;\mbox{for all $L_0 \ge 1$}\,.
\end{array}
\end{equation}

\n
So in the case of decreasing events we see that for all $L_0 \ge 1$, $n \ge 0$, $x \in E$, 
\begin{equation}\label{3.29}
\IP [G^{\ov{u}}_{x,L_n}] \le \IP[G_{x,L_n}^{u^+_\infty}] \underset{(\ref{3.27}) {\rm ii)}}{\stackrel{(\ref{3.21})}{\le}} \Big[c(\cG,\lambda)\ell_0^{2\lambda} \;\sup\limits_{x \in E} \,\IP[G_{x,L_0}^u] + \frvier \Big]^{2^n}\,.
\end{equation}

\medskip
In view of (\ref{3.25}) we can pick a large enough $L_0$ so that (\ref{3.26}) holds. When the family consists of increasing events, the same argument with (\ref{3.22}) and (\ref{3.28}) in place of (\ref{3.21}), (\ref{3.22}) yields the claim. The case of a family depending on a parameter is handled in a similar fashion as above.
\end{proof}

\begin{remark}\label{rem3.6} \rm  ~

\medskip\n
1) Note that we can pick $\ell_0$ and $L_0$ in Corollary \ref{cor3.5} as $c(\ov{u},\cG)$, where the convention concerning constants can be found at the end of the Introduction. For instance we can select $\ell_0 > 1$ minimal such that (\ref{3.26}) holds for some $L_0 \ge 1$ and then given this choice of $\ell_0$ pick a minimal $L_0$ such that (\ref{3.26}) holds. 

\medskip\n
2) The above corollary could in fact accommodate the choice of a faster growth of complexity than what is imposed on cascading families in (\ref{3.3}). One could as well have chosen $c(\cG, \wt{\nu}) \, e^{\ell^{\wt{\nu}}}$, with $\wt{\nu}<\nu$, in place of $c(\cG, \lambda) \, \ell^\lambda$, and the above proof would have gone through with minor modifications. Apart from the various examples presented  in Proposition \ref{prop3.2} and Remark \ref{rem3.3}, one motivation for the choice that appears in (\ref{3.3}) stems from the next proposition, where the parameter $\lambda$ from (\ref{rem3.3}) plays an explicit role.
\hfill $\square$

\medskip
We are now ready to state the second corollary of Theorem \ref{theo3.4}, which will be used in Section 5. The following observation will be implicit in the interpretation of the expression in (\ref{3.30}) below: given $X_\point^1, \dots, X^M_\point$ independent canonical random walks on $E$, the indicator function $\chi_{\cup^M_1 {\rm range}(X_\point^i)}$ of the union of the ranges of the walks defines a $\{0,1\}^E$-valued random variable, (when $M=0$ this random variable is constant and equal to the function identically equal to $0$ on $E$).
\end{remark}

\begin{corollary}\label{cor3.7}
Consider a family $\cG = (G_{x,L})_{x \in E, L \ge 1 \,{\rm integer}}$ of increasing events on $\{0,1\}^E$ cascading with complexity at most $\lambda > 0$, and an integer $M \ge [\frac{2 \lambda}{\nu}]$, (recall $\nu = \alpha - \frac{\beta}{2} > 0$). Assume that
\begin{equation}\label{3.30}
\underset{L \r \infty}{\underline{\lim}} \;\sup\limits_{x,x_1,\dots,x_M} \;\bigotimes\limits^M_{i=1} \,P_{x_i} [\chi_{\cup^M_1 {\rm range}(X^i_\point)} \in G_{x,L}] = 0\,,
\end{equation}

\n
where the supremum runs over $x \in E$, $x_1,\dots,x_M \in \partial_{\rm int} B(x,20 L)$. Then one can find $u > 0$, and integers $L_0 \ge 1$, $\ell_0 > 1$, so that
\begin{equation}\label{3.31}
\sup\limits_{x \in E} \,\IP[G^u_{x,L_n}] \le 2^{-2n}, \;\mbox{for all $n \ge 0$}\,.
\end{equation}

\n
In addition when the family depends on a parameter and cascades uniformly with complexity at most $\lambda$, if (\ref{3.30}) holds with a joint supremum over $x,x_1,\dots, x_M$ as above and the parameter set, then the supremum in (\ref{3.31}) can be replaced by a joint supremum over $x$ and the parameter set.
\end{corollary}

\begin{proof}
We choose $K=2$ and $0 < \nu^{\prime\prime} < \nu^\prime < \nu = \alpha - \frac{\beta}{2}$, so that
\begin{equation}\label{3.32}
\nu^{\prime\prime} (M +1) > 2 \lambda \,.
\end{equation}

\n
We know from Theorem \ref{theo3.4} that for $\ell_0 \ge c(K=2, \nu^\prime)$ multiple of $100$, $u_0 > 0$, $L_0 \ge 1$, one has
\begin{equation}\label{3.33}
\sup\limits_{x \in E} \,\IP[G_{x,L_n}^{u^-_\infty}] \le [c(\cG,\lambda) \ell_0^{2 \lambda} \big(\sup\limits_{x \in E} \,\IP [G_{x,L_0}^{u_0}] + \ve (u_\infty^-))]^{2^n}\,.
\end{equation}

\n
From now on we assume $\ell_0$ large enough so that, cf.~(\ref{2.70}), for all $u_0 > 0$,
\begin{align} 
u^-_\infty &\ge \fr \;u_0 \,,\label{3.34}
\intertext{and such that setting}
u_0  &= L_0^{-\nu} \;\ell_0^{-\nu ''}, \label{3.35}
\end{align}
one also has, cf.~(\ref{2.73}):
\begin{equation}\label{3.36}
c(\cG,\lambda) \, \ell_0^{2 \lambda} \ve(u^-_\infty) \le  c(\cG, \lambda) \,\ell_0^{2 \lambda} \,2e^{-\ell_0^{\nu^\prime - \nu^{\prime\prime}}} / (1-e^{-\ell_0^{\nu^\prime - \nu^{\prime\prime}}}) \le \frvier, \;\; \mbox{for all $L_0 \ge 1$}\,.
\end{equation}

\n
The claim (\ref{3.31}) will then follow in view of (\ref{3.33}), (\ref{3.36}), once we show that
\begin{equation}\label{3.37}
\lim\limits_{\ell_0 \r \infty} \; \underset{L_0 \r \infty}{\underline{\lim}} \;\ell_0^{2 \lambda} \;\sup\limits_{x \in E} \;\IP[G^{u_0}_{x,L_0}] = 0\,,
\end{equation}

\n
with $u_0$ as in (\ref{3.35}) and $\ell_0$ multiple of $100$.

\medskip
For this purpose we observe that $G_{x,L_0}$ is $\sigma(\Psi_{x'}$,~$x'\! \in \! B(x,10L_0))$-measurable and $\cI^{u_0} \cap B(x,10 L_0)$ coincides by (\ref{1.32}) with the trace on $B(x,10L_0)$ of the union of the ranges of the trajectories in the support of the Poisson point measure $\mu  \stackrel{\rm def}{=} \mu_{B(x,20 L_0),u_0}$. The intensity of $\mu$ equals $u_0 \,P_{e_{B(x,20L_0)}}$ by (\ref{1.27}), and the law of $\mu$ can thus be generated as the sum of a Poisson number of point masses on $W^+$ located at independent random walk trajectories with common distribution $P_{\ov{e}}$, where the common starting distribution equals $\ov{e} = e_{B(x,20L_0)} / {\rm cap} (B(x,20L_0))$ and the Poisson variable has intensity
\begin{equation}\label{3.38}
u_0 \,{\rm cap}(B(x,20 L_0)) \stackrel{(\ref{1.26})}{\le} c \,u_0 \,L_0^\nu \stackrel{(\ref{3.35})}{\le} c\,\ell_0^{-\nu^{\prime\prime}} \stackrel{\rm def}{=} \kappa \,.
\end{equation}

\medskip\n
The probability that the Poisson variable exceeds $M$ is at most:
\begin{equation*}
e^{-\kappa} \dsl_{k \ge M + 1} \;\dis\frac{\kappa^k}{k!} \le \kappa^{M+1} = c^{(M+1)} \ell_0^{-\nu^{\prime\prime}(M+1)} .
\end{equation*}
As a result we see that
\begin{equation}\label{3.39}
\begin{array}{l}
\ell_0^{2 \lambda} \; \sup\limits_{x \in E} \;\IP[G_{x,L_0}^{u_0}] \le c^{(M+1)} \,\ell_0^{2 \lambda - (M+1) \nu^{\prime\prime}}  \; + 
\\[2ex]
\ell^{2 \lambda}_0 \;\sup\limits_{x,x_1,\dots,x_M} \;\bigotimes\limits^M_{i=1} \, P_{x_i} [\chi_{\cup^M_1 {\rm range}(X^i_\point)} \in G^{u_0}_{x,L_0}]\,,
\end{array}
\end{equation}

\medskip\n
where the supremum runs over the same collection as in (\ref{3.30}). By (\ref{3.30}) and (\ref{3.32}) we see that (\ref{3.37}) holds. This proves (\ref{3.31}). The case of a family $\cG$ depending on a parameter is handled in a similar fashion. This concludes the proof of Corollary \ref{cor3.7}.
\end{proof}

\begin{remark}\label{rem3.8} \rm ~

\medskip\n
1) In the case $E = \IZ^{d+1}$, $d \ge 2$, the family $\wt{\cS}$ of separation events, cf.~(\ref{3.19}), (\ref{3.20}) introduced by Teixeira \cite{Teix10} cascades with complexity at most $d+1$, and $\nu = d-1$, so that $[\frac{2 \lambda}{\nu}] = 2 + [\frac{4}{d-1}]$. When $d \ge 4$, we can choose $M=3$, and it follows from Theorem 6.11 of \cite{Teix10} that (\ref{3.30}) holds. As a result we see by (\ref{3.31}) that one can find $u > 0$, $\ell_0 > 1$, and $L_0 \ge 1$, such that
\begin{equation}\label{3.40}
\IP[\wt{S}^u_{0,L_n}] \le 2^{-2^n}, \; \mbox{for all $n \ge 0$}\,,
\end{equation}

\medskip\n
(the supremum over $x$ has been dropped due to translation invariance).

\bigskip\n
2) In view of the above remark, the family $\cS$ in (\ref{3.20a}) is also a natural candidate for which condition (\ref{3.30}) of Corollary \ref{cor3.7} ought to be tested. We hope this point will be tackled elsewhere.

\bigskip\n
3) By similar considerations as in Remark \ref{rem3.6}, we can pick $u$, $\ell_0$, $L_0$ as $c(\cG)$ in the above Corollary \ref{cor3.7}. \hfill $\square$
\end{remark}

\section{Finiteness of $\pmb{u_*}$ and connectivity bounds}
\setcounter{equation}{0}

In this section we explore the percolative properties of the vacant set of random interlacements on $E$, in the sub-critical phase of the model. We show that in our set-up the critical parameter $u_*$ in (\ref{0.7}) is always finite, and we derive stretched exponential bounds on the connectivity function $\IP[x \stackrel{\cV^u}{\longleftrightarrow} y]$, when $u > u_{**}$, with $u_{**} \ge u_*$,  a certain finite critical value, possibly equal to $u_*$, see (\ref{0.10}). The main result appears in Theorem \ref{theo4.1}, and comes as an application of Theorem \ref{theo3.4} and Corollary \ref{cor3.5}. Even in the case $E = \IZ^{d+1}$, $d \ge 2$, Theorem \ref{theo4.1} improves on what is presently known, see Remark \ref{rem4.2} 1) below.

\medskip
We recall the definition of $u_{**}$ from (\ref{0.10}),
\begin{equation}\label{4.1}
u_{**} = \inf\{ u \ge 0; \;\underset{L \r \infty}{\underline{\lim}} \;\sup\limits_{x \in E} \,\IP[B(x,L) \stackrel{\cV^u}{\longleftrightarrow} \partial_{\rm int} B(x,2L)] = 0\} \in [0,\infty]\,.
\end{equation}

\n
With the notation from (\ref{0.3}), it is plain that $\eta(x,u) = 0$ for all $x$ in $E$, when $u > u_{**}$, and therefore we see that
\begin{equation}\label{4.2}
0 \le u_* \le u_{**} \le \infty \,.
\end{equation}

\n
A routine covering argument further shows that $u_{**}$ does not change when we restrict $L$ to integer values in (\ref{4.1}), see also below (\ref{4.5}). The following theorem will in particular show that $u_{**}$ is finite.

\begin{theorem}\label{theo4.1}
\begin{equation}\label{4.3}
0 \le u_* \le u_{**} < \infty \,,
\end{equation}

\medskip\n
and for $u > u_{**}$ there exist $c_4(u)$, $c_5(u) > 0$, $0 < \kappa(u) < 1$, such that
\begin{equation}\label{4.4}
\sup\limits_{x \in E} \;\IP[B(x,L) \stackrel{\cV^u}{\longleftrightarrow} \partial_{\rm int} B(x,2L)] \le c_4 \,e^{-c_5 L^\kappa}, \;\mbox{for $L \ge 1$}\,.
\end{equation}
\end{theorem}

\begin{proof}
Once we know that $u_{**} < \infty$, it follows from Corollary \ref{cor3.5} applied to the collection $\cA$ of (\ref{3.6}), see also Remark \ref{rem3.6}, that for $u > u_{**}$, there is $\ell_0(u) > 1$ and $L_0(u) \ge 1$, such that:
\begin{equation}\label{4.5}
\sup\limits_{x \in E} \;\IP[A^u_{x,L_n}] \le 2^{-2^n}, \;\; \mbox{for all $n \ge 0$}\,.
\end{equation}

\n
If we now consider $L \ge 4 L_0$, and $L_n \le L/4 < L_{n+1}$, a similar argument as in the proof of Proposition \ref{prop3.2} enables us to cover $B(x,L)$ by a collection of a most $c(u)$ closed balls of radius $L_n$ with centers in $B(x,L)$. It then follows that
\begin{equation}\label{4.6}
\begin{array}{l}
\IP[B(x,L) \stackrel{\cV^u}{\longleftrightarrow} \partial_{\rm int} B(x,2L)] \le c(u) \;\sup\limits_{x \in E} \;\IP[A^u_{x,L_n}] \le c(u)\,2^{-2^n} =
\\[1ex]
c(u) \,\exp\Big\{- \Big(\mbox{\f $\dis\frac{L_{n+1}}{L_1}$}  \Big)^\kappa\Big\} \le c(u) \,\exp\{ - L^{-\kappa}_1 \,L^\kappa\}, \;\; 
\mbox{with $\kappa(u) = \dis\frac{\log 2}{\log \ell_0}$} \;,
\end{array}
\end{equation}

\medskip\n
and (\ref{4.4}) follows straightforwardly.

\medskip
We hence only need to show that $u_{**}$ is finite. For this purpose we use (\ref{3.21}) with $\cG = \cA$, $\lambda = \alpha + \frac{\beta}{2}$, $K = 1$, $\nu^\prime = \frac{\nu}{2}$, and find that for some $\ell_ 0 > 1$, $L_0 =1$, and all $n \ge 0$, $u_0 > 0$, one has
\begin{equation}\label{4.7}
\begin{array}{l}
\sup\limits_{x \in E} \;\IP[B(x,L_n)  \stackrel{\cV^{u^+_\infty}}{\longleftrightarrow} \partial_{\rm int} B(x,2L_n)]  \le 
\\[2ex]
\big[c\,\ell_0^{2 \lambda} \big(\sup\limits_{x \in E} \;\IP[B(x,1)  \stackrel{\cV^{u_0}}{\longleftrightarrow}  \partial_{\rm int} B(x,2)]  + \ve(u_0)\big)\big]^{2^n}\,.
\end{array}
\end{equation}

\medskip\n
Taking into account that the event under the probability in the right-hand side of (\ref{4.7}) is contained in the event $\{\cV^{u_0} \cap B(x,2) \not= \emptyset\}$, we see that
\begin{equation}\label{4.8}
\sup\limits_{x \in E} \;\IP[B(x,1)  \stackrel{\cV^{u_0}}{\longleftrightarrow}  \partial_{\rm int} B(x,2)]  \stackrel{(\ref{1.37})}{\le} c \,e^{-cu_0}\,.
\end{equation}

\n
Thus for large enough $u_0$, the left-hand side of (\ref{4.7}) is at most $2^{-2^n}$ for all $n \ge 0$. This proves that $u_{**} < \infty$, and concludes the proof of Theorem \ref{theo4.1}.
\end{proof}

\begin{remark}\label{rem4.2} \rm ~

\medskip\n
1)  In the case $E = \IZ^{d+1}$, $d \ge 2$, when we use the sup-norm distance in place of $d(\cdot,\cdot)$, the above result improves upon what is known from \cite{SidoSzni09b}. Indeed it shows that
\begin{equation}\label{4.9}
\begin{split}
u_{**} & = \inf \big\{ u \ge 0; \;  \underset{L \r \infty}{\underline{\lim}}
 \;\IP[B_\infty(0,L) \stackrel{\cV^u}{\longleftrightarrow} \partial_{\rm int} B_\infty(0,2L)]  = 0\big\}
\\ 
& = \inf \big\{ u \ge 0; \;\mbox{for some $\alpha > 0$},   \lim\limits_{L \r \infty} \; L^\alpha \, \IP[B_\infty(0,2L) \stackrel{\cV^u}{\longleftrightarrow} \partial_{\rm int} B_\infty(0,2L)]  = 0\big\}\,,
\end{split}
\end{equation}

\n
the second line being a priori bigger or equal to $u_{**}$, but in fact equal to $u_{**}$ thanks to (\ref{4.4}). The quantity in the second line of (\ref{4.9}) is used for the definition of $u_{**}$ in \cite{SidoSzni09b}, and the equality stated in (\ref{4.9}) is new.

\bigskip\n
2) Can one construct examples where $u_{**} > u_*$ holds? Can this bring some light concerning the open question whether $u_* = u_{**}$ in the case of $E = \IZ^{d+1}$, $d \ge 2$~? The absence of a true substitute for the BK Inequality in the context of interlacement percolation makes this last question hard to answer at present.

\bigskip\n
3) Note that for any two distinct points $x,x'$ in $E$, with $d(x,x') = L+1$, one has the inclusion 
\begin{equation*}
\{ x  \stackrel{\cV^u}{\longleftrightarrow} x' \} \subseteq \{ x  \stackrel{\cV^u}{\longleftrightarrow}  \partial_{\rm int} B(x,L)\}\,,
\end{equation*}

\medskip\n
because any path in $E$ from $x$ to $x'$ must at some point exit $B(x,L)$. This observation combined with (\ref{4.4}) implies that for $u > u_{**}$ the connectivity function has a stretched exponential decay:
\begin{equation}\label{4.10}
\IP[ x  \stackrel{\cV^u}{\longleftrightarrow} x' ] \le c(u)\,e^{-c'(u) d(x,x')^\kappa} \;\mbox{for $x,x'$ in $E$}\,,  
\end{equation}
with $\kappa$ as in (\ref{4.4}).

\bigskip\n
4) It is plain that the argument employed in the proof of Theorem \ref{theo4.1} is quite robust and could be adapted to different collections of decreasing events having the cascading property. \hfill $\square$

\end{remark}

\section{Positivity of $\pmb{u_*}$ and connectivity in half-planes}
\setcounter{equation}{0}

In this section we explore the percolative properties of the vacant set of random interlacements on $E$ in the super-critical phase of the model. We show that $u_* > 0$ when $\alpha \ge 1 + \frac{\beta}{2}$, i.e. when $\nu \ge 1$, and that for small $u > 0$, $\cV^u$ percolates in half-planes, see (\ref{3.8}) for the definition of half-planes. Additionally we derive stretched exponential bounds on the $*$-connectivity function of $\cI^u$ inside half-planes. Our main results appear in Theorem \ref{theo5.1} and Corollary \ref{cor5.5}. The proof of Theorem \ref{theo5.1} has some flavor of the proof of Theorem 5.3 of \cite{Szni08}, where a similar restriction $\alpha \ge 1 + \frac{\beta}{2}$ is present. This restriction rules out examples such as $E = G \times \IZ$, where $G$ is the discrete skeleton of the Sierpinski gasket, cf.~\cite{Jone96}, \cite{BarlCoulKuma05}. It is an interesting question whether $u_* > 0$ remains true in this case. Let us point out that when $E = \IZ^{d+1}$, $d \ge 2$, the condition $\nu \ge 1$ is automatically fulfilled and the results of this section go beyond present knowledge, cf.~\cite{SidoSzni09b},  \cite{Teix09b}, \cite{Teix10}.

\medskip
We recall the definition of half-planes in (\ref{3.8}) and of $*$-paths below (\ref{3.9}). We then introduce
\begin{equation}\label{5.1}
\wt{u} = \inf\big\{ u \ge 0; \;\underset{L \r \infty}{\underline{\lim}} \;\sup\limits_{\cP, x \in \cP} \;\IP\big[ B(x,L)  \stackrel{* - \cI^u \cap \cP}{\longleftrightarrow} \partial_{\rm int} B(x,2L)\big] > 0\big\}\,,
\end{equation}

\n
(when $L$ is a positive integer, the event inside the probability coincides with $B^u_{x,L,\cP}$, where $B_{x,L,\cP}$ has been defined in (\ref{3.10})). The main result of this section is the following.

\begin{theorem}\label{theo5.1}
\begin{equation}\label{5.2}
0 \le \wt{u} \le u_*\,,
\end{equation}
moreover for $u < \wt{u}$,
\begin{equation}\label{5.3}
\mbox{$\IP$-a.s., $\cV^u$ percolates in every half-plane $\cP$}\,,
\end{equation}

\n
and there exist positive $c_6(u)$, $c_7(u)$, and $0 < \wt{\kappa}(u) < 1$, such that
\begin{equation}\label{5.4}
\sup\limits_{\cP, x \in \cP} \;\IP\big[ B(x,L)  \stackrel{* - \cI^u \cap \cP}{\longleftrightarrow} \partial_{\rm int} B(x,2L)\big]  \le c_6 \,e^{-c_7 L^{\wt{\kappa}}}, \;\mbox{for $L \ge 1$}\,.
\end{equation}

\n
In addition when $\alpha \ge 1 + \frac{\beta}{2}$, i.e.~$\nu \ge 1$, one has
\begin{equation}\label{5.5}
0 < \wt{u} \le u_*\,.
\end{equation}
\end{theorem}

\begin{proof}
The claim (\ref{5.4}) is a direct application of Corollary \ref{cor3.5}, and similar considerations as below (\ref{4.5}). The claim (\ref{5.2}) follows immediately once we prove (\ref{5.3}). We now prove (\ref{5.3}). To this end we consider $u < \wt{u}$ and $\cP = \{y(n); n \ge 0\} \times \IZ$, where $y(n)$, $n \ge 0$, is a semi-infinite geodesic in $G$. We use the notation $x_0 = (y(0), 0)$, and $\IL_+ = \{y(0)\} \times [0,\infty)$, $\IL_- = \{y(0)\} \times (- \infty, -1]$. The claim (\ref{5.3}) is a consequence of
\begin{equation}\label{5.6}
\mbox{$\IP[x_0$ belongs to an infinite connected component of $\cV^u \cap \cP] > 0$}\,.
\end{equation}

\n
Indeed a similar argument as in the proof of (\ref{1.39}), see also (\ref{1.42}), shows that a zero-one law holds for the probability that $\cV^u$ percolates in $\cP$. Now for any integer $M > 0$, the probability in (\ref{5.6}) is at least
\begin{equation}\label{5.7}
\begin{array}{l}
\mbox{$\IP [B(x_0,M) \cap \cP \subseteq \cV^u$ and there is no $*$-path in $\cI^u \cap \cP$ from}
\\
\IL_+ \backslash B(x_0,M) \; \mbox{to} \; \IL_- \backslash B(x_0,M)] \stackrel{(\ref{1.34})}{\ge} 
\\[2ex]
\mbox{$\IP [B(x_0,M) \cap \cP \subseteq \cV^u] \;\IP[$ there is no $*$-path in $\cI^u \cap \cP$ from}
\\[1ex]
\IL_+ \backslash B(x_0,M) \; \mbox{to} \; \IL_- \backslash B(x_0,M)] \,.
\end{array}
\end{equation}

\n
The first factor in the last line equals $\exp\{ - u \, {\rm cap} (B(x_0,M) \cap \cP)\} > 0$. As for the last factor, we consider the complement of the event under the probability, and keep track of the largest $B(x_0,2^k M)$ not containing the rightmost point of the $*$-path on $\{(y(n), 0)$; $n \ge 0\}$. Setting $\wt{x}_k = (y(2^{k+1} M + 1), 0)$, for $k \ge 0$, we see that the last factor is at least
\begin{equation*}
\begin{array}{l}
1 - \dsl_{k \ge 0} \,\IP[B(\wt{x}_k,2^k M)  \stackrel{* - \cI^u \cap \cP}{\longleftrightarrow} \partial_{\rm int} B(\wt{x}_k,2^{k+1} M)] \stackrel{(\ref{5.4})}{\ge} 
\\[3ex]
1 - \dsl_{k \ge 0} \,c(u)\,e^{-c'(u)(2^k M)^{\wt{\kappa}}} > 0, \;\mbox{when $M$ is large}\,.
 \end{array}
\end{equation*}

\n
This proves (\ref{5.6}) and thus completes the proof of (\ref{5.3}).

\medskip
We now turn to the proof of (\ref{5.5}) and assume from now on unless otherwise specified that $\nu = \alpha - \frac{\beta}{2} \ge 1$. In essence this assumption makes it hard for the random walk on $E$ to have a trace in a half-plane that covers a large $*$-path. We define the integer $M$ via
\begin{equation}\label{5.8}
M = \Big[\dis\frac{\beta}{\nu}\Big]\,.
\end{equation}

\n
In view of (\ref{3.12}) and Corollary \ref{cor3.7}, (\ref{5.4}) will hold for small $u>0$, and our claim (\ref{5.5}) will follow once we prove that:
\begin{equation}\label{5.9}
\underset{L \r \infty}{\underline{\lim}} \;\wt{\sup} \; \bigotimes^M_{i=1} \;P_{x_i} [\chi_{\cup^M_1 \,{\rm range}(X_\point^i)} \in B_{x,L,\cP}] = 0\,,
\end{equation}

\n
where the notation $\wt{\sup}$ refers to a supremum over $\cP$, $x$ in $\cP$, $x_1,\dots,x_M$ in $\partial_{\rm int} B(x,20L)$, and $B_{x,L,\cP}$ appears in (\ref{3.10}).

\medskip
Given a half-plane $\cP$ and $x$ in $\cP$, we can find, depending on $x$ and $L\ge 1$, three or four rectangles $\wt{D} = \wt{W} \times \wt{J}$ in $\cP \cap (\partial_{\rm int} B(x,L) \cup (B(x,2L) \backslash B(x,L)))$, bordering $\partial_{\rm int} B(x,L) \cap \cP$, with $L \le |\wt{W}| \le 5L$, $[L^{\beta/2}] \le |\wt{J}| \le [(5 L)^{\beta/2}]$, and such that any $*$-path in $\cP$ from $B(x,L) \cap \cP$ to $\partial_{\rm int} B(x,2L) \cap \cP$ contains a $*$-path in $\cP$ joining the opposite sides of one of these rectangles, (see Figure 2).

\psfragscanon
\begin{center}
\includegraphics[width=11cm]{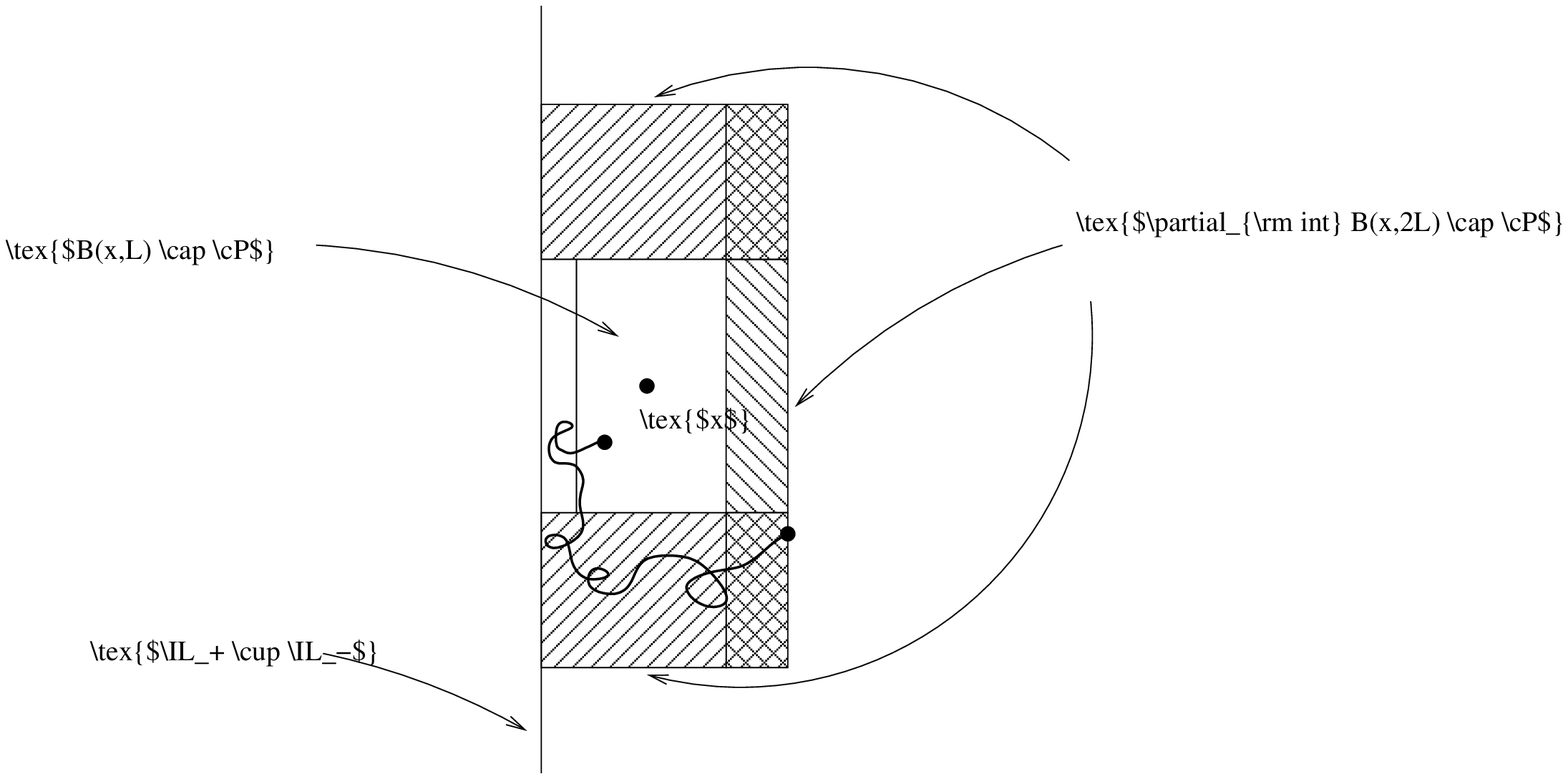}
\end{center}

\begin{center}
\begin{tabular}{ll}
Fig.~2: & An illustration with three rectangles bordering $\partial_{\rm int} B(x,L) \cap \cP$
\\
&and a $*$-path in $\cP$ from $B(x,L) \cap \cP$ to $\partial_{\rm int} B(x,2L) \cap \cP$
\\
&inducing a crossing of one of the three rectangles.
\end{tabular}
\end{center}

\medskip\n
We now define
\begin{equation}\label{5.11}
\begin{array}{l}
\mbox{$H = [(L/ N)^{\frac{\beta}{2}}]$, where $N = \log L$, and $L \ge 10^3$ is large enough}
\\
\mbox{so that $10^3 \le H < [L^{\frac{\beta}{2}}]$}\,.
\end{array}
\end{equation}

\medskip\n
We can cover each of the three or four above rectangles $\wt{D}$ by at most $c\, N^{\beta/2}$ overlapping rectangles $D$ having same horizontal projection as $\wt{D}$ and a vertical projection of cardinality $H$, so that every $*$-path joining opposite sides of $\wt{D}$ has a trace in at least one of the rectangles $D$ with either a horizontal projection containing a segment of at least $L/10$ points, or a vertical projection containing a segment of at least $H/10$ points. To see this, one for instance considers the collection of rectangles $D$ ``vertical translates'' of the form $\wt{W} \times ([1, H] + k[H / 100])$, $k \in \IZ$, that intersect $\wt{D} = \wt{W} \times \wt{J}$. One then looks at the location of the $*$-path when it reaches the ``middle'' of the side of $\wt{D}$ it crosses, and picks one of the rectangles $D$ with ``closest center'' to that location.

\medskip
From the above considerations we see that for large $L$, the expression inside the $\liminf$ in (\ref{5.9}) is bounded from above by
\begin{equation}\label{5.12}
\begin{split}
A \stackrel{\rm def}{=} c\,N^{\frac{\beta}{2}} \overline{\sup} \;\bigotimes\limits^M_{i=1} \,P_{x_i} \Big[ & \bigcup^M_{i=1} \,{\rm range} (X_\point^i) \cap D \;\mbox{has $G$-projection with at least $L/10$}
\\
& \mbox{points or $\IZ$-projection with at least $H/10$ points$\Big]$},
\end{split}
\end{equation}

\n
where the notation $\overline{\sup}$ refers to a supremum over $\cP$, $x_1,\dots,x_M$ in $\cP$, and $D = W \times J$, a rectangle in $\cP$, with $L \le |W| \le 5L$, and $|J| = H$.

\medskip
For a rectangle as above we introduce the notation
\begin{equation}\label{5.13}
J_y = \{y\} \times J, \;y \in G, \;W_z = W \times \{z\}, \, z \in \IZ\,.
\end{equation}

\n
The following lemma comes as a preliminary step in bounding $A$ in (\ref{5.12}). It will enable us to control the average number of vertical segments $J_y, y \in W$, and horizontal segments $W_z, z \in J$, of $D$ visited by a random walk on $E$ starting from an arbitrary location. The bound on $A$ will then follow by an application of Khashminskii's lemma, see~\cite{Khas59}. The calculations are similar to what appears in the proof of Corollary 5.3 of \cite{Szni08}.

\begin{lemma}\label{lem5.2} $(\nu \ge 1, L \ge c)$

\medskip
When $x = (y,z)$, $\ov{x} = (\ov{y}, \ov{z})$ are in $D$,
\begin{equation}\label{5.14}
P_x[H_{J_{\ov{y}}} < \infty]  \le \left\{ \begin{array}{ll} c(H \wedge d_G(y,\ov{y})^{\beta/2}) / d_G(y,\ov{y})^\nu, \;\mbox{when} \; \alpha > \beta\,,
\\[1ex]
 c\{1 + \log (H^{\frac{2}{\beta}}/ d_G(y,\ov{y}))\} / \log H, \;\mbox{when $\alpha = \beta$ and}
\\[1ex]
\hspace{6cm}  0 < d_G(y,\ov{y}) < H^{\frac{2}{\beta}},
\\[1ex]
  c\,H \,d_G (y,\ov{y})^{-\frac{\beta}{2}} / \log H, \;\mbox{when $\alpha = \beta$ and $d_G(y,\ov{y}) \ge  H^{\frac{2}{\beta}}$},
\\[1ex]
  c(1 \wedge (H^{\frac{2}{\beta}} / d_G(y,\ov{y}))^\nu), \;\mbox{when $\alpha < \beta$},
\end{array}\right.
\end{equation}
as well as
\begin{equation}\label{5.15}
\hspace{-10ex} P_x [H_{W_{\ov{z}}} < \infty]  \le \left\{\begin{array}{ll} c (1 \wedge |z-\ov{z}|^{\frac{2}{\beta} (1 - \nu)}), \;\mbox{if $\nu > 1$}\,,
\\[1ex]
  c (1 + \log (L / |z- \ov{z}|^{\frac{2}{\beta}})) / \log L, \;\mbox{if $\nu = 1$}\,.
\end{array}\right.
\end{equation}
\end{lemma}

\begin{proof}
We begin with the proof of (\ref{5.14}). We apply the right-hand inequality of (\ref{1.25}) with $K = J_{\ov{y}}$. With this in mind we introduce the function
\begin{equation}\label{5.16}
\hspace{-6.5cm}\cJ_{\ov{y}}(x') = \dsl_{\wt{x} \in J_{\ov{y}}}\,g(x',\wt{x}), \;\mbox{for $x' \in E$}\,.
\end{equation}
We note that by (\ref{1.9}) when $x' \in J_{\ov{y}}$, one has
\begin{equation}\label{5.17}
\hspace{-3.2cm} \cJ_{\ov{y}}(x') \ge c\,\dsl_{\ell = 1}^H \, \ell^{-(\frac{2\alpha}{\beta} - 1)}  \ge \left\{\begin{array}{l}  c, \; \mbox{when $\alpha > \beta$}\,,
\\[0.5ex]
 c \,\log H, \;\;\mbox{when $\alpha = \beta$}\,,
\\[0.5ex] 
c \,H^{2 - \frac{2 \alpha}{\beta}}, \;\mbox{when $\alpha < \beta$}\,.
\end{array}\right.
\end{equation}

\n
On the other hand when $x \notin J_{\ov{y}}$, with $x$ as in (\ref{5.14}), it follows from (\ref{1.9}) that
\begin{equation}\label{5.18}
\hspace{-6.2cm} \cJ_{\ov{y}}(x)  \le c\, \dsl_{\ell = 1}^H \,\big(d_G(y,\ov{y})^\nu \vee \ell^{\frac{2\alpha}{\beta}-1}\big)^{-1}\,,
\end{equation}

\n
so that when $d_G(y,\ov{y})^{\frac{\beta}{2}} \le H$, keeping track of whether $\ell \le d_G(y,\ov{y})^{\beta/2}$ or not, we find that
\begin{equation}\label{5.19}
\hspace{-3cm}\cJ_{\ov{y}}(x) \le \left\{\begin{array}{l} c\,d_G(y,\ov{y})^{\frac{\beta}{2} - \nu}, \;\; \mbox{when $\alpha > \beta$},
\\[1ex]
 c\,(1 + \log (H / d_G(y,\ov{y})^{\frac{\beta}{2}})) ,  \;\mbox{when $\alpha = \beta$}\,,
\\[1ex]
 c \,H^{2 - \frac{2 \alpha}{\beta}}, \;\mbox{when $\alpha < \beta$}\,.
\end{array}\right.
\end{equation}

\n
On the other hand when $d_G(y,\ov{y})^{\frac{\beta}{2}} > H$, we find instead in all cases that
\begin{equation}\label{5.20}
\cJ_{\ov{y}}(x) \le c\,H \,d_G(y,\ov{y})^{-\nu}\,.
\end{equation}

\n
The right-hand inequality in (\ref{1.25}) yields that $P_x[H_{J_{\ov{y}}} < \infty]$ is bounded by \\
 $\cJ_{\ov{y}}(x) /$ $\inf_{x' \in \cJ_{\ov{y}}} \cJ_{\ov{y}}(x')$, and (\ref{5.14}) now follows (note that $\nu = \frac{\beta}{2}$, when $\alpha = \beta$).

\medskip
We then continue with the proof of (\ref{5.15}), and now introduce the function
\begin{equation}\label{5.21}
\cW_{\ov{z}}(x') = \dsl_{\wt{x} \in W_{\ov{z}}} \,g(x',\wt{x}), \;\mbox{for $x' \in E$}\,.
\end{equation}

\n
We find by (\ref{1.9}) that when $x' \in W_{\ov{z}}$, (remember that $L \le |W| \le 5L)$, one has:
\begin{equation}\label{5.22}
\cW_{\ov{z}}(x') \ge c \,\dsl_{\ell=1}^{|W|} \, \ell^{-\nu}  \ge \left\{ \begin{array}{l}
c, \;\mbox{when $\nu > 1$}\,,
\\ 
c \,\log L, \;\mbox{when $\nu = 1$}\,.
\end{array}\right.
\end{equation}

\n
On the other hand, when $x \notin W_{\ov{z}}$, with $x$ as in (\ref{5.15}), we have by (\ref{1.9}):
\begin{equation}\label{5.23}
\begin{split}
\cW_{\ov{z}}(x) &  \le c \,\dsl_{\ell=1}^{|W|} \, \big(\ell^{\nu}  \vee |z - \ov{z}|^{\frac{2 \alpha}{\beta} - 1}\big)^{-1}
\\
& \le c \,|z - \ov{z}|^{\frac{2}{\beta} + 1 - \frac{2\alpha}{\beta}} + c \dsl_{|z - \ov{z}|^{\frac{2}{\beta}} < \ell \le |W|} \ell^{-\nu}
\\
& \le c \,|z - \ov{z}|^{\frac{2}{\beta} (1-\nu)}, \;\mbox{when $\nu > 1$}\,,
\\
& \le c(1 + \log (L / |z - \ov{z}|^{\frac{2}{\beta}})), \;\mbox{when $\nu =1$}\,.
\end{split}
\end{equation}

\medskip\n
Using once again the right-hand inequality of (\ref{1.25}) we find that $P_x[H_{W_{\ov{z}}} < \infty]$ is bounded by $\cW_{\ov{z}}(x) / \inf_{x' \in W_{\ov{z}}} \cW_{\ov{z}}(x')$, and (\ref{5.15}) follows.
\end{proof}

With the help of the above lemma, we will now derive an upper bound on the expected number of vertical or horizontal segments in $D$ visited by the random walk in $E$. We first introduce some notation. As a consequence of the strong Markov property at time $H_D$, we see that
\begin{align}
N_{\rm vert}& \stackrel{\rm def}{=} \sup\limits_{x \in E} \; E_x \Big[\dsl_{\ov{y} \in W}\,1\{H_{J_{\ov{y}}} < \infty\}\Big] = \sup\limits_{x \in D} \; E_x \Big[\dsl_{\ov{y} \in W} \,1\{H_{J_{\ov{y}}} < \infty\}\Big]\,,\label{5.24}
\\[-2ex]
\intertext{and that}
N_{\rm hor}& \stackrel{\rm def}{=} \sup\limits_{x \in E} \; E_x \Big[\dsl_{\ov{z} \in J}\,1\{H_{W_{\ov{z}}} < \infty\}\Big] = \sup\limits_{x \in D} \; E_x \Big[\dsl_{\ov{z} \in J} \,1\{H_{W_{\ov{z}}} < \infty\}\Big]\,,\label{5.25}
\end{align}

\medskip\n
where $D \subseteq \cP$ is a rectangle of the form $D= W \times J$, with $L \le |W| \le 5L$, and $|J| = H (< L^{\frac{\beta}{2}})$, see (\ref{5.11}). The following lemma encapsulates the estimates that we will use, and that are stated in a simplified form to avoid tracking all special values of the parameters.

\begin{lemma}\label{lem5.3} $(\nu \ge 1, L \ge c)$
\begin{align}
N_{\rm vert}&\le \; c\, H^{\frac{2}{\beta}} ( 1+ \log (L / H^{\frac{2}{\beta}})) /(1 + 1\{\alpha = \beta\} \log L)\,,\label{5.26}
\\
N_{\rm hor}& \le \; c\, H ( 1+ \log (L / H^{\frac{2}{\beta}})) / \log L\,.\label{5.27}
\end{align}
\end{lemma}

\begin{proof}
We begin with the proof of (\ref{5.26}). When $\alpha > \beta$, by the first line of (\ref{5.14}) we find 
\begin{equation}\label{5.28a}
\begin{split}
N_{\rm vert}& \le \dsl_{1 \le \ell < H^{\frac{2}{\beta}}} c \,\ell^{-(\alpha - \beta)} + \dsl_{ H^{\frac{2}{\beta}} \le \ell \le |W|} c\,H \ell^{-\nu}
\\
& \le c\,H^{\frac{2}{\beta} (1 + \beta - \alpha)_+} + c(\log H) \,1\{\alpha = \beta + 1\}\,,
\end{split}
\end{equation}
whence (\ref{5.26}).

\medskip
When $\alpha = \beta$, using the second and third line of (\ref{5.14}) we find that:
\begin{equation}\label{5.28}
\begin{split}
N_{\rm vert}&\le \dsl_{1 \le \ell < H^{\frac{2}{\beta}}} \;\dis\frac{c}{\log H} \;\Big(1 + \log \Big(\dis\frac{H^{\frac{2}{\beta}}}{\ell}\Big)\Big) + \dsl_{H^{\frac{2}{\beta}} \le \ell \le |W|} \; \dis\frac{c}{\log H} \;H\,\ell^{-\frac{\beta}{2}}
\\[1ex]
& \le c\,H^{\frac{2}{\beta}} / \log  H + c\,H^{\frac{2}{\beta}} ( 1 + 1\{ \beta = 2\} \log (L / H^{\frac{2}{\beta}})) / \log H,
\end{split}
\end{equation}
whence (\ref{5.26}).

\medskip
When $\beta > \alpha > 1 + \frac{\beta}{2}$, (so that $\nu > 1$ and $\beta > 2$), the last line of (\ref{5.14}) yields that:
\begin{equation}\label{5.29}
N_{\rm vert} \le \dsl_{\ell = 1}^{|W|} \, c\; \Big(1 \wedge  \Big(\dis\frac{H^{\frac{2}{\beta}}}{\ell}\Big)^\nu\Big) \le c\, H^{\frac{2}{\beta}}  + c\,H^{\frac{2 \nu}{\beta}} \; \dsl_{\ell \ge H^{\frac{2}{\beta}}} \; \ell^{-\nu} \le c\,H^{\frac{2}{\beta}}, 
\end{equation}
whence (\ref{5.26}).

\medskip
Finally when $\beta > \alpha = 1 + \frac{\beta}{2}$, (so that $\nu = 1$ and $\beta > 2$), the last line of (\ref{5.14}) now yields that:
\begin{equation}\label{5.30}
N_{\rm vert} \le \dsl_{\ell = 1}^{|W|} \, c\; \Big(1 \wedge  \Big(\dis\frac{H^{\frac{2}{\beta}}}{\ell}\Big)^\nu\Big) \le  c\, H^{\frac{2}{\beta}}  \big(1 +  \log ( L / H^{\frac{2}{\beta}})\big), 
\end{equation}
whence (\ref{5.26}).

\pagebreak
We then turn to the proof of (\ref{5.27}). When $ \alpha > 1 + \frac{\beta}{2}$, then $0 > \frac{2}{\beta} \,(1-\nu) =  \frac{2}{\beta} \,(1 + \ \frac{\beta}{2} - \alpha) = 1 -  \frac{2}{\beta}\,(\alpha - 1)$, and the first line of (\ref{5.15}) implies that:
\begin{equation}\label{5.31}
N_{\rm hor} \le \dsl_{k = 1}^{H} \; c\,k^{\frac{2}{\beta} (1 - \nu)} \le c \, H^{(1 + \frac{2}{\beta} (1 - \nu))_+} + c\,1\{\alpha = 1 + \beta\} \log H\,,
\end{equation}
whence (\ref{5.27}).

\medskip
Finally when $ \alpha = 1 + \frac{\beta}{2}$, then $\nu = 1$ and the second line of (\ref{5.15}) yields that
\begin{equation}\label{5.32}
N_{\rm hor} \le \dsl_{k = 1}^{H} \; \dis\frac{c}{\log L} \; \Big(1 + \log \Big(\dis\frac{L^{\frac{\beta}{2}}}{k}\Big)\Big) \le c \, H/\log L + c\,H \log (L^{\frac{\beta}{2}}/H) / \log L \,,
\end{equation}

\n
whence (\ref{5.27}). This concludes the proof of Lemma \ref{lem5.3}. 
\end{proof}

\medskip
We are now ready to prove (\ref{5.9}). For this purpose we bound $A$ in (\ref{5.12}) with the help of Khasminskii's Lemma, cf.~\cite{Khas59}, (2.46) of \cite{DembSzni06}, or \cite{ChungZhao95}, p.71. We thus find that for all rectangles $D = W \times J$ as stated above Lemma \ref{lem5.3} we have:
\begin{align}
\sup\limits_{x \in E}& \;E_x \Big[\exp\Big\{ \dis\frac{c}{N_{\rm vert}} \;\dsl_{\ov{y} \in W} \,1\{H_{J_{\ov{y}}} < \infty\}\Big\}\Big] \le 2, \;\mbox{and} \label{5.33}
\\[2ex]
\sup\limits_{x \in E} &\; E_x \Big[\exp\Big\{ \dis\frac{c}{N_{\rm hor}} \;\dsl_{\ov{z} \in J} \,1\{H_{W_{\ov{y}}} < \infty\}\Big\}\Big] \le 2\,.\label{5.34}
\end{align}

\medskip\n
As a result of the exponential Chebyshev inequality and (\ref{5.12}) we thus find:
\begin{equation}\label{5.35}
A \le c N^{\frac{\beta}{2}} \;2^M\Big[\exp\Big\{ - c\,\mbox{\f $\dis\frac{L}{N_{\rm vert}}$} \Big\} + \exp\Big\{ - c\; \mbox{\f $\dis\frac{H}{N_{\rm hor}}$}\Big\}\Big]\,.
\end{equation}

\medskip\n
However from (\ref{5.11}) and Lemma \ref{lem5.3} we find that for $L \ge c$,
\begin{align}
 \dis\frac{L}{N_{\rm vert}} &\ge c \; \dis\frac{L}{H^{\frac{2}{\beta}}} \; \dis\frac{(1 + 1\{\alpha = \beta\} \log L)}{1 + \log (L / H^{\frac{2}{\beta}})} \ge c \; \dis\frac{\log L}{\log \log L} , \; \mbox{and} \label{5.36}
\\[2ex]
 \dis\frac{H}{N_{\rm hor}}& \ge c \; \dis\frac{\log L}{1 + \log (L / H^{\frac{2}{\beta}})} \;  \ge c \; \dis\frac{\log L}{\log \log L} \;. \label{5.37}
\end{align}

\medskip\n
Inserting these bounds in (\ref{5.35}), (recall that $N = \log L)$, implies that $A$ tends to zero as $L$ goes to infinity. This proves (\ref{5.9}) and thus concludes the proof of Theorem \ref{theo5.1}.
\end{proof}

\begin{remark}\label{rem5.4} \rm In the above proof the methods of Section 3 make the case where $\nu > \beta$, i.e.~$\alpha > \frac{3}{2} \, \beta$, much simpler to treat. Indeed $M$ in (\ref{5.8}) vanishes and so does the expression under the supremum in (\ref{5.9}). This immediately yields the claim (\ref{5.5}).

\medskip
In the case $\alpha \le \frac{3}{2} \, \beta$, the choice of rectangles $D$ where the height $H$ is such that $H^{\frac{2}{\beta}}$ is slightly smaller (by a logarithmic factor) than the base $|W| \in [L,5L]$, see (\ref{5.11}) and below (\ref{5.12}), is truly useful when handling the situation where $\beta > \alpha$, (and $\nu \ge 1$), as can be seen from (\ref{5.29}), (\ref{5.30}). Such a procedure works because $N_{\rm hor}$ nevertheless remains ``negligible'' with respect to $H$, as can be seen in (\ref{5.31}), (\ref{5.32}). \hfill $\square$
\end{remark}

\medskip
We state a consequence of (\ref{5.4}) in terms of the size of the finite connected components of $\cV^u$ in half-planes when $u < \wt{u}$. The notation is the same as in Theorem \ref{theo5.1}, with $0 < \wt{\kappa}(u) < 1$ introduced in (\ref{5.4}).

\begin{corollary}\label{cor5.5}
When $u < \wt{u}$, there exist positive $c_8(u)$, $c_9(u)$ such that
\begin{equation}\label{5.38}
\begin{split}
\sup\limits_{\cP, x \in \cP} \;\IP\big[ & \mbox{the connected component of $x$ in $\cV^u \cap \cP$ is finite}
\\[-2ex]
&\mbox{and intersects $\partial_{\rm int} B(x,L)] \le c_8 \,e^{-c_9\, L^{\wt{\kappa}}}$, for $L \ge 1$}\,.
\end{split}
\end{equation}
\end{corollary}

\begin{proof}
On the event inside the probability one can find a $*$-path in $\cI^u \cap \cP$ separating in $\cP$ the connected component of $x$ in $\cV^u \cap \cP$ from infinity. The respective $G$- and $\IZ$-projections of the range of the $*$-path are ``intervals'' containing the respective $G$- and $\IZ$-projections of $x$, and either the $G$-projection contains at least $L$ points or the $\IZ$-projection contains at least $L^{\frac{\beta}{2}}$ points. One can then consider the rightmost point of the $*$-path having same height as $x$. A variation of the argument used to bound from below the last factor of (\ref{5.7}) combined with the inequalities $\sum_{k \ge 0} e^{-c(u)(2^kL)^{\wt{\kappa}}} \le \sum_{k \ge 0} \,e^{-c'(u)(k+1)L^{\wt{\kappa}}} \le c(u)\,e^{-c'(u) L^{\wt{\kappa}}}$ yields the claim (\ref{5.38}) in a straightforward fashion.
\end{proof}

\begin{remark}\label{rem5.6} \rm ~

\medskip\n
1)  In the important special case $E = \IZ^{d+1}$, $d \ge 2$, when the sup-norm distance $d_\infty(\cdot,\cdot)$ replaces $d(\cdot,\cdot)$, the results of Theorem \ref{theo5.1} and Corollary \ref{cor5.5}, show that the plane exponent:
\begin{equation}\label{5.39}
\wt{u}_{ {\rm pl}} = \inf\big\{ u \ge 0; \;\underset{L \r \infty}{\underline{\lim}} \;\IP[B_\infty(0,L) \stackrel{* - \cI^u \cap \IZ^2}{\longleftrightarrow} B(0,2L)] > 0\big\}\,,
\end{equation}

\medskip\n
(which coincides with $\wt{u}$), is such that for any $d \ge 2$,
\begin{equation}\label{5.40}
0 <\wt{u}_{{\rm pl}}  \le u_* ,
\end{equation}
that for $u < \wt{u}_{{\rm pl}}$,
\begin{equation}\label{5.41}
\mbox{$\IP$-a.s., $\cV^u$ percolates in all planes of $\IZ^d$},
\end{equation}

\n
and there exist $c(u), \,c'(u), \,0 < \gamma(u) < 1$, such  that for $L \ge 1$,
\begin{align}
\IP\big[ &B_\infty(0,L)  \stackrel{* - \cI^u \cap \IZ^2}{\longleftrightarrow} \partial_{\rm int}  B_\infty(0,2L)\big] \le c \,e^{-c' \,L^\gamma}, \;\mbox{and}\label{5.42}
\\[1ex]
\IP\big[&\mbox{the connected component of $0$ in $\cV^u \cap \IZ^2$ is finite and intersects} \label{5.43}
\\
&\partial_{\rm int} B_\infty(0,L)\big] \le c\,e^{-c'\,L^{\gamma}}\,. \nonumber
\end{align}

\n
The above sharpens previously known results of \cite{SidoSzni09a}, cf.~Remark 3.5 2), where arbitrary polynomial decay in $L$ for the left-hand side of (\ref{5.42}), when $u$ is sufficiently small has been established. It also complements the separation results from \cite{Teix09b}, for small $u > 0$ and $d+1 \ge 5$, see also (\ref{3.20}) and Remark \ref{rem3.8} 1).

\medskip\n
2) It is a very interesting question to understand what happens when $\alpha < 1 + \frac{\beta}{2}$. Does one of the values $\wt{u}$ or $u_*$ vanish, or both? Can one successfully apply Corollary \ref{cor3.7} to a cascading family in the spirit of the family $\cS$ of separation events, which has been introduced in Remark \ref{rem3.3} 3)? The special case of $E = G \times \IZ$, where $G$ is the discrete skeleton of the Sierpinski gasket endowed with its natural weights, corresponds to $\alpha = \frac{\log 3}{\log 2}$, $\beta = \frac{\log 5}{\log 2}$, cf.~\cite{Jone96},  \cite{BarlCoulKuma05}, and is somehow emblematic of this puzzle. \hfill $\square$
\end{remark}

\appendix
\section{Appendix}
\setcounter{equation}{0}

The main object of this appendix is to prove that the collection $\cS$ of ``separation events" introduced in Remark \ref{rem3.3} 3) has the cascading property. The proof is in essence an adaptation of the arguments that appear in Lemma 5.1 and Theorem 5.2 of Teixeira \cite{Teix10}, but the geometry is possibly quite different in the present set-up. Our main result is stated in Proposition \ref{propA.2}. We refer to the beginning of Section 1 for notation, and recall that $d_E(\cdot,\cdot)$ denotes the graph-distance on $E$, which should not be confused with $d(\cdot,\cdot)$, see (\ref{0.3}).

\medskip
We first note that for any subsets $A, B$ of $E$, one has
\begin{equation}\label{A.1}
d_E(A,B) > 1 \Longleftrightarrow \ov{A} \cap B = \emptyset  \Longleftrightarrow A \cap \ov{B} = \emptyset \,.
\end{equation}

\n
We say that $A$ and $B$ are separated by $C$ in $U$, all of them being subsets of $E$, when $d_E(A,B) > 1$ and any path from $\partial A$ to $\partial B$ remaining in $U$ meets $C$. The collection of events $\cS = (S_{x,L})_{x \in E, L \ge 1\;{\rm integer}}$ introduced in (\ref{3.20a}) and discussed in this appendix is defined by:
\begin{equation}\label{A.2}
\begin{split}
S_{x,L}  = \big\{ & \mbox{$\sigma \in \{0,1\}^E$; there exist connected subsets $A_1$ and $A_2$ of $\ov{B(x,3L)}$ with}
\\
& \mbox{$d(\cdot,\cdot)$-diameter at least $L$, separated by $\Sigma (\sigma)$ in $B(x,5L)\big\}$}\,,
\end{split}
\end{equation}

\n
where $\Sigma(\sigma) = \{x \in E; \sigma(x) = 1\}$. For simplicity we will often write $\Sigma$ in place of $\Sigma(\sigma)$, and diameter in place of $d(\cdot,\cdot)$-diameter. Our first result is:

\begin{lemma}\label{lemA.1} $(\ell$ multiple of $100$, $L \ge 1$ integer, $x_*$ in $E$)

\medskip
Assume $A_1,A_2 \subseteq \ov{B(x_*,3 \ell L)}$ are connected subsets of diameter at least $\ell L$, that $\sigma \in \{0,1\}^E$ is such that $\Sigma$ separates $A_1$ and $A_2$ in $B(x_*, 5 \ell L)$, and that $x_1,\dots, x_M$ is a sequence in $E$ such that:
\begin{equation}\label{A.3}
\begin{array}{rl}
{\rm i)} & B(x_i,5L) \subseteq B(x_*, 5 \ell L), \; 1 \le i \le M\,,
\\[1ex]
{\rm ii)} & d(x_i,x_{i+1}) \le L, \;\mbox{for} \; 1 \le i < M\,,
\\[1ex]
{\rm iii)} &A_1 \cap B(x_1,2L) \not= \emptyset \;\mbox{and} \;A_2 \cap B(x_M,2L) \not= \emptyset\,.
\end{array}
\end{equation}
It then follows that
\begin{equation}\label{A.4}
\mbox{for some $i \in \{1,\dots, M\}$, $\sigma \in S_{x_i,L}$}\,.
\end{equation}
\end{lemma}

\begin{proof}
Let $\cC$ denote the connected component of $A_1$ in $(A_1 \cup \Sigma^c) \cap B(x_*, 5 \ell L)$. By construction we have:
\begin{equation}\label{A.4a}
\partial \cC \cap B(x_*, 5 \ell L) \subseteq \Sigma\,.
\end{equation}
Moreover we also have
\begin{equation}\label{A.5}
\ov{A}_2 \cap \cC = \emptyset\,.
\end{equation}

\medskip\n
Indeed otherwise we could find a path in $B(x_*, 5 \ell L)$ from $A_1$ to $\ov{A}_2$ not meeting $\Sigma$ after its last visit to $A_1$, and as a result we could find a path from $\partial A_1$ to $\partial A_2$ in $B(x_*, 5 \ell L)$ not meeting $\Sigma$, a contradiction.

\medskip
By (\ref{A.5}) and (\ref{A.4a}) we now see that $d_E(\cC, A_2) > 1$, and any path from $\partial \cC$ to $\partial A_2$ in $B(x_*,5 \ell L)$ meets $\Sigma$, i.e.,
\begin{equation}\label{A.6}
\mbox{$\cC$ and $A_2$ are separated by $\Sigma$ in $B(x_*,5 \ell L)$}\,.
\end{equation}

\n
By (\ref{A.3}) iii) we know that $\cC \cap B(x_1,2L) \not= \emptyset$, and we denote by $i_*$ the largest $1 \le i \le M$, such that $\cC \cap B(x_i, 2L) \not= \emptyset$. We first note that
\begin{equation}\label{A.7}
\mbox{when $i_* = M$,  then $\sigma \in S_{x_M,L}$}\,.
\end{equation}

\n
Indeed $\cC$ and $A_2$ have diameter at least $\ell  L > 6 L + 2$ and thus both meet $\partial B(x_M,3L)$ and $B(x_M,2L)$. We can hence choose connected subsets $A'_1$ of $\cC \cap \ov{B(x_M,3L)}$ and $A'_2$ of $A_2 \cap \ov{B(x_M,3L)}$ with diameter at least $L$. Due to (\ref{A.6}) they are separated by $\Sigma$ in $B(x_*,5 \ell L)$ and hence in $B(x_M,5L)$. The claim (\ref{A.7}) now follows.

\medskip
We then observe that
\begin{equation}\label{A.8}
\mbox{when $1 \le i_* < M$, then $\sigma \in S_{x_{i_*},L}$}\,.
\end{equation}

\n
Indeed consider the ``vertical segment'', (with $\pi_G$ and $\pi_\IZ$ the respective $G$-projection and $\IZ$-projection on $E$):
\begin{equation*}
J = \{x \in E ; \;\pi_G(x) = \pi_G(x_{i_*+1}) \;\mbox{and} \; -1 \le \pi_\IZ(x) - \pi_\IZ(x_{i_*+1}) \le L^{\beta/2}\} \subseteq B(x_{i_*+1},L)\,.
\end{equation*}

\n
By assumption $\cC \cap B(x_{i_*+1}, 2L) = \emptyset$, so that $\ov{J} \cap \cC \subseteq B(x_{i_*+1},2L) \cap \cC = \emptyset$, and $d_E (J, \cC) > 1$. Suppose that $\sigma \notin S_{x_{i_*},L}$. Due to (\ref{A.3}) ii) we see that $J \subseteq \ov{B(x_{i_*},3L)}$, and we can extract a connected subset $\cC'$ of $\cC \cap \ov{B(x_{i_*},3L)}$ with diameter at least $L$ and a path from $\partial \cC'$ to $\partial J$ in $B(x_{i_*},5L) \backslash \Sigma$. This shows that $\cC$ extends to $\partial J$ and thus meets $B(x_{i_*+1},2L)$, a contradiction. The claim (\ref{A.4}) now follows from (\ref{A.7}), (\ref{A.8}), and Lemma \ref{lemA.1} is proved.
\end{proof}

We now come to the main result of this appendix.

\begin{proposition}\label{propA.2}
\begin{equation}\label{A.9}
\begin{split}
&\mbox{$\cS$ is a family of increasing events on $\{0,1\}^E$ that cascades with}
\\
&\mbox{complexity at most $\alpha + \frac{\beta}{2}$}\,.
\end{split}
\end{equation}
\end{proposition}

\begin{proof}
Clearly $\cS$ is a family of increasing events which satisfy (\ref{3.1}). For any given $\ell$ multiple of $100$, $x_*$ in $E$, and integer $L \ge 1$, we want to find a finite subset $\Lambda$ of $E$ for which (\ref{3.2}) - (\ref{3.6}) hold, with $\lambda = \alpha + \frac{\beta}{2}$ and $x_*$ in place of $x$ in (\ref{3.2}) - (\ref{3.6}). Without loss of generality and to simplify notation, we assume that $x_* = (y_*,0)$.

\medskip
We consider $\Lambda_G \subseteq B_G(y_*,5 \ell L)$ defined as follows. When $L \le 4$, $\Lambda_G = B_G(y_*, 5 \ell L)$ and when $L > 4$, $\Lambda_G$ is a maximal set of points in $B_G(y_*,5 \ell L)$ with mutual $d_G$-distance bigger than $L/4$. By construction we see that
\begin{equation}\label{A.10}
\begin{array}{l}
B_G(y_*, 5 \ell L) \subseteq \bigcup\limits_{y \in \Lambda_G} \,B(y,L/4)
\end{array}
\end{equation}

\n
and by the $\alpha$-Ahlfors regularity of $G$, see (\ref{1.7}), with similar arguments as in the proof of Proposition \ref{prop3.2}, we find that
\begin{equation}\label{A.11}
|\Lambda_G | \le c \,\ell^\alpha\,.
\end{equation}

\medskip\n
Further note that for $y \in G$, $R\ge 0$, the ball $B_G(y,R)$ is connected, and using (\ref{A.10}) when $L > 4$, we have the following local connectivity property of $\Lambda_G$:
\begin{equation}\label{A.12}
\begin{array}{l}
\mbox{when $B_G(y,R) \subseteq B_G(y_*,5 \ell L)$ and $y', y{''} \in \Lambda_G \cap B_G(y,R)$, there is a}
\\
\mbox{sequence $y_1,\dots ,y_M$ in $\Lambda_G \cap B_G(y,R + \frac{L}{4})$ such that $y_1 = y'$, $y_M = y{''}$,}
\\
\mbox{and $d_G(y_i, y_{i+1}) \le L$, for $1\le i < M$}\,.
\end{array}
\end{equation}

\medskip\n
We also introduce the finite subset of $\IZ$:
\begin{equation}\label{A.13}
\Lambda_\IZ = a \, \IZ \cap [-(5 \ell L)^{\beta/2}, (5 \ell L)^{\beta/2}], \;\mbox{where}\; a = \max (1, [(L/4)^{\frac{\beta}{2}}])\,.
\end{equation}

\medskip\n
We then define $\Lambda = \Lambda_G \times \Lambda_\IZ$, and note that
\begin{equation}\label{A.14}
\begin{array}{l}
\Lambda \subseteq B(x_*,5 \ell L), \; \mbox{and} \; B(x_*,5 \ell L) \subseteq \bigcup\limits_{x \in \Lambda} B\Big(x,\frac{L}{4}\Big)\,,
\\[1ex]
\mbox{(when $L \le 4$, $\Lambda$ actually coincides with $B(x_*,5 \ell L))$.}
\end{array}
\end{equation}
Further one has 
\begin{equation}\label{A.15}
|\Lambda| \le c \, \ell^{\alpha + \frac{\beta}{2}}\,.
\end{equation}

\medskip\n
Thus $\Lambda$ satisfies (\ref{3.2}), (\ref{3.3}), and Proposition \ref{propA.2} will follow once we prove (\ref{3.4}).

\medskip
To this end we consider $\sigma \in S_{x_*,5 \ell L}$ and $A_1, A_2$ connected subsets of $\ov{B(x_*,3 \ell L)}$ with diameter at least $\ell L$ separated by $\Sigma$ in $B(x_*,5 \ell L)$, see below (\ref{A.2}) for notation.

\medskip
By Lemma \ref{lemA.1}, the claim (\ref{3.4}) will follow once we find two sequences in $\Lambda$, $\pi = (\pi(i), 1 \le i \le M)$ and $\pi' = (\pi'(j), 1 \le j \le M')$, such that
\begin{equation}\label{A.16}
\begin{array}{rl}
{\rm i)} & \mbox{both $\pi$ and $\pi'$ satisfy (\ref{A.3})}.
\\[1ex]
{\rm ii)} & d(\pi(i), \, \pi'(j)) \ge \dis\frac{\ell}{100} \,L, \; 1 \le i \le M, \; 1 \le j \le M'\,.
\end{array}
\end{equation}

\n
We write $\wt{\ell} = \ell / 100 (\ge 1$ by assumption). As we now explain, we can find $x_1,x'_1,x_2,x'_2$ in $\Lambda$ so that
\begin{equation}\label{A.17}
\begin{array}{rl}
{\rm i)} & \mbox{the four points have mutual distance at least $24 \wt{\ell} \,L$}.
\\[2ex]
{\rm ii)} & A_1 \cap B(x_1, L/4) \not= \emptyset, \; A_1 \cap B(x'_1, L/4) \not= \emptyset\,,
\\[2ex]
{\rm iii)} & A_2 \cap B(x_2, L/4) \not= \emptyset, \; A_2 \cap B(x'_2, L/4) \not= \emptyset\,.
\end{array}
\end{equation}

\medskip\n
Indeed we can pick two points of $A_1$ with mutual distance at least $\ell L$. Since $A_2$ is connected and has diameter at least $\ell L$, it has at least one point outside the union of the two balls of radius $50 \wt{\ell} \,L-1$ centered at the above two points. Likewise $A_2$ must exit the ball of radius $25 \wt{\ell} \,L$ centered at this last point and hence meet the interior boundary of this ball. In this fashion we have constructed two points on $A_1$ and two points on $A_2$, so that these four points have mutual distance at least $25 \wt{\ell} \, L-1$. By (\ref{A.14}) we can thus find $x_1 = (y_1,z_1)$, $x'_1 = (y'_1, z_1')$, $x_2 = (y_2,z_2)$, $x'_2 = (y'_z, z'_2)$ in $\Lambda$ satisfying (\ref{A.17}). 

\bigskip
Up to relabelling we assume that $z_1 \ge z'_1$ and $z_2 \ge z'_2$.  We will construct the sequences $\pi$ and $\pi'$ satisfying (\ref{A.16}) when $z_1 \ge z_2$. The case where $z_1 < z_2$ is handled in a similar fashion exchanging the role of $1$ and $2$ in what follows. We from now on assume $z_1 \ge z_2$, and treat separately the case where $z_2$ lies in the interval $[z'_1,z_1]$ and the case where it does not.

\bigskip\n
1) $z'_1 \le z_2 \le z_1$:

\psfragscanon
\begin{center}
\includegraphics[width=4cm]{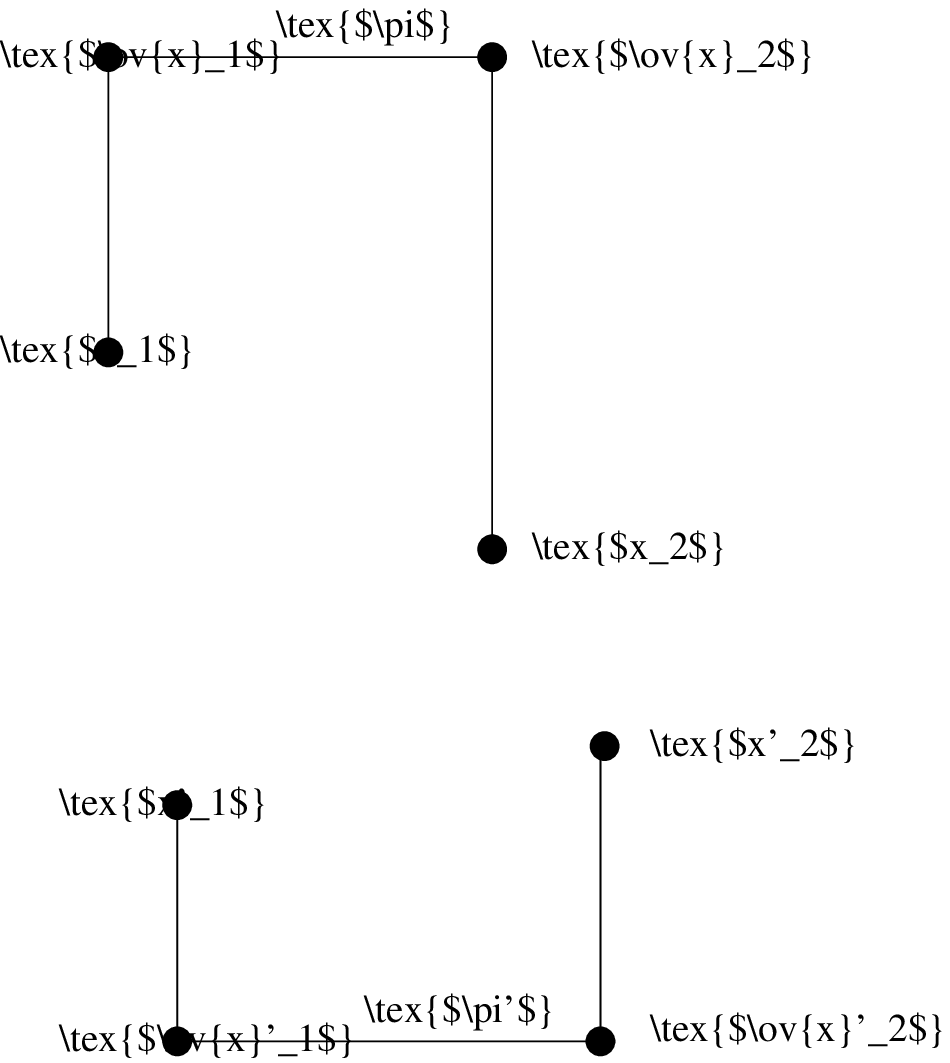}
\end{center}

\medskip\n
\begin{center}
Fig.~3: A schematic representation of $\pi$ and $\pi'$ in case 1)
\end{center}

\bigskip\n
We construct $\pi$ and $\pi'$ as follows. We pick $\ov{x}_1 = (y_1,\ov{z}_1)$ having same $G$-projection as $x_1$, with $\ov{z}_1$ the largest elements of $\Lambda_\IZ$ smaller or equal to $z_1 + (25 \wt{\ell} L)^{\beta/2}$, and define $\ov{x}_2 = (y_2,\ov{z}_1)$. By (\ref{A.12}) we can find a sequence in $\Lambda_G \cap B_G (y_*,3 \ell L + 1 + \frac{L}{4})$ linking $y_1$ and $y_2$ making steps of $d_G$-distance at most $L$. We thus define $\pi$ fulfilling (\ref{A.3}) which first moves up from $x_1$ to $\ov{x}_1$, then horizontally from $\ov{x}_1$ to $\ov{x}_2$, using the above mentioned sequence, and then down from $\ov{x}_2$ to $x_2$.

\medskip
Likewise to construct $\pi'$ we consider $\ov{x}'_1 = (y'_1, \ov{z}'_1)$ and $\ov{x}'_2 = (y'_2, \ov{z}'_1)$, where $ \ov{z}'_1$ denotes the smallest element of $\Lambda_\IZ$ bigger or equal to $z'_1 \wedge z'_2 - (25 \wt{\ell}\,L)^{\beta/2}$. We link $y'_1$ and $y'_2$ by a sequence in $\Lambda_G \cap B_G(y,3 \ell L + 1 + \frac{L}{4})$ as above, and construct $\pi'$ satisfying (\ref{A.3}), which first moves down from $x'_1$ to $\ov{x}'_1$, then horizontally from $\ov{x}'_1$ to $\ov{x}'_2$, and then vertically from $\ov{x}'_2$ to $x'_2$.

\medskip
The above constructed $\pi$ and $\pi'$, see Figure 3, thus satisfy (\ref{A.3}) and
\begin{equation}\label{A.18}
d(\pi(i), \pi(j)) \ge 24 \wt{\ell}\,L, \;\mbox{for} \; 1 \le i \le M, 1 \le j \le M'\,.
\end{equation}

\bigskip\n
2) $z_1\ge  z'_1 >z_2 (\ge z'_2)$:

\bigskip\n
We consider two sub-cases.

\medskip\n
a) $d_G(y'_1,y_2) \le 11 \wt{\ell}\,L$.

\medskip\n
We introduce $\ov{x}_2 = (y_2,z'_1)$ and link $y'_1$ and $y_2$ in $\Lambda_G \cap B_G(y'_1, 11 \wt{\ell}\,L + \frac{L}{4})$ by a sequence making steps of $d_G$-distance at most $L$. We then construct $\pi'$ which moves first horizontally from $x'_1$ to $\ov{x}_2$, using the above sequence and then downwards from $\ov{x}_2$ to $x'_2$.

\medskip
To construct $\pi$ we introduce $\ov{x}_1$ and $\ov{x}'_2$ in $\Lambda$ with same $G$-projection as $x_1$ and $x'_2$ respectively, with $\ov{x}_1$ the ``highest point'' in $\Lambda$ with distance at most $25 \wt{\ell}\,L$ from $x_1$ and $\ov{x}'_2$ the ``lowest point'' in $\Lambda$ with distance at most $25 \wt{\ell}\,L$ from $x'_2$, see Figure 4.

\medskip
We then pick $\ov{y}$ in $\Lambda_G \cap \ov{B_G(y_*, 3 \ell L)}$ having at least $d_G$-distance $50 \wt{\ell}\,L$ from $B_G(y'_1$,  $11 \wt{\ell}\,L + \frac{L}{4})$. This is possible because $\ov{B_G(y_*, 2\ell L)}$ has $d_G$-diameter at least $2\ell L$ and it suffices to pick $y \in  \ov{B_G(y_*, 2\ell L)}$ with $d_G(y,y'_1) \ge \ell L$ and then $\ov{y}$ in $\Lambda_G$ within $d_G$-distance $\frac{L}{4}$ from $y$.

\medskip
Then $\pi$ is constructed as follows. It first moves upwards from $x_1$ to $\ov{x}_1$ then horizontally to the point having $G$-projection $\ov{y}$, then downwards to the point with same $\IZ$-projection as $\ov{x}'_2$, then horizontally to $\ov{x}'_2$, and then upwards to $x'_2$. This construction can be performed so that the $G$-projection of $\pi$ lies in $\Lambda_G \cap B(y_*, 3 \ell L + 1 + \frac{L}{4})$.

\psfragscanon
\begin{center}
\includegraphics[width=7cm]{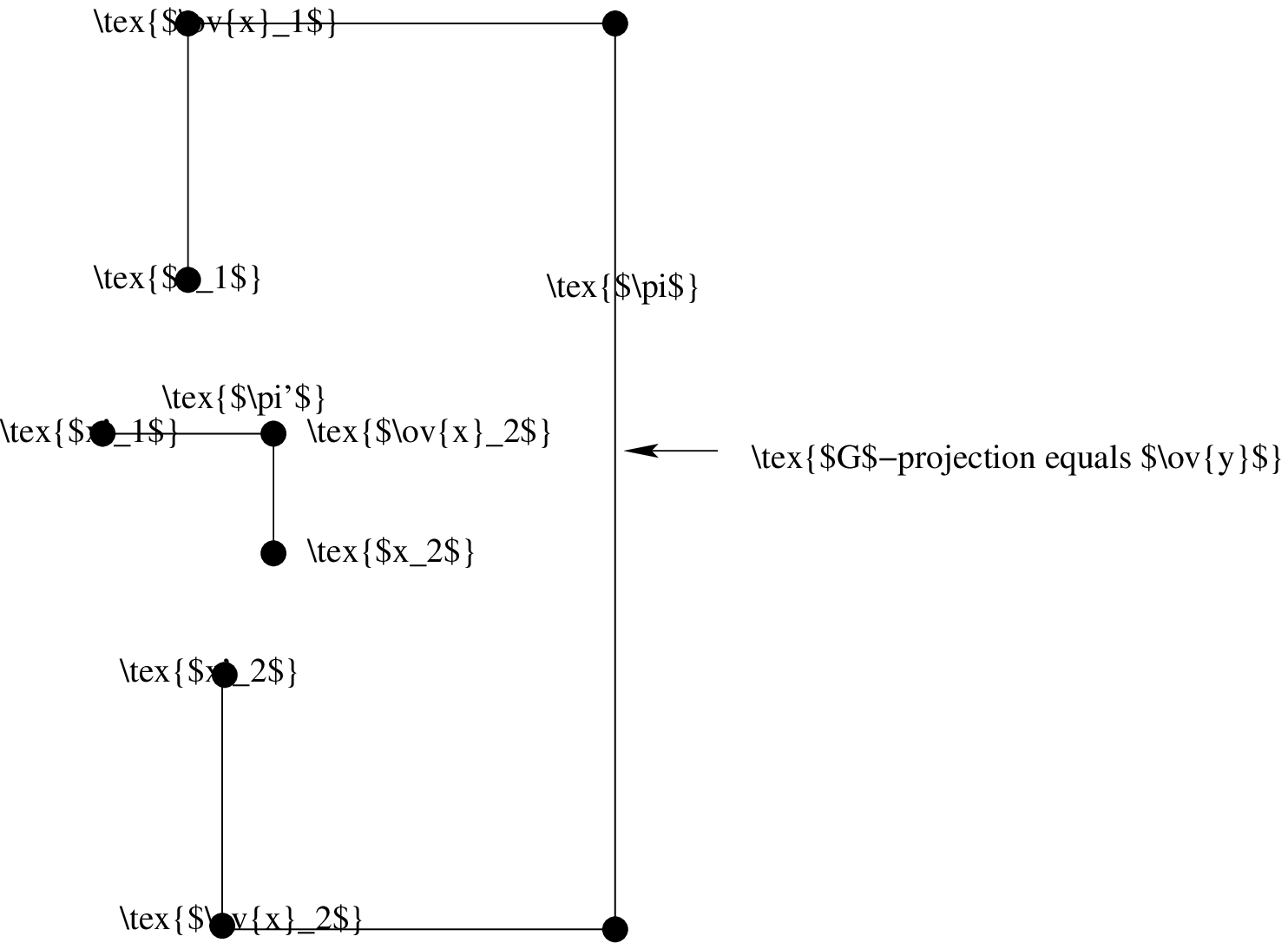}
\end{center}

\medskip\n
\begin{center}
Fig.~4: A schematic representation of $\pi, \pi'$ in case 2a)
\end{center}

\bigskip\n
The above constructed paths $\pi, \pi'$ satisfy (\ref{A.3}) and
\begin{equation}\label{A.19}
d(\pi(i), \pi(j)) \ge 12 \wt{\ell} \, L\,.
\end{equation}

\n
The remaining sub-case of 2) corresponds to 

\bigskip\n
b) $d_G(y'_1,y_2) > 11 \wt{\ell}\,L$.

\medskip
We then construct $\pi$ linking $x_1$ and $x_2$ and $\pi'$ linking $x'_1$ and $x'_2$ as in 1).

\psfragscanon
\begin{center}
\includegraphics[width=7cm]{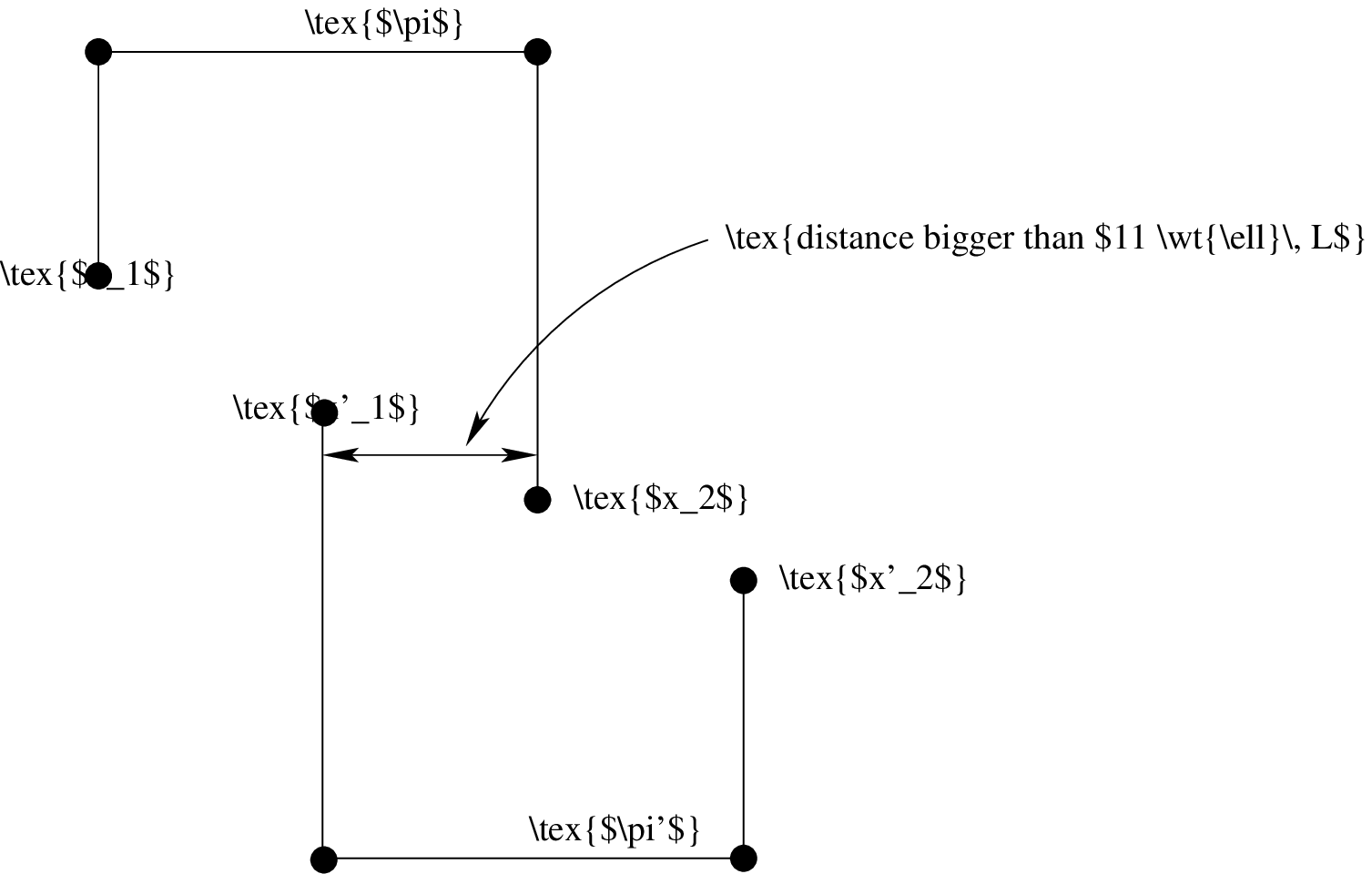}
\end{center}

\medskip\n
\begin{center}
Fig.~5: A schematic representation of $\pi$ and $\pi'$ in case 2) b)
\end{center}

\bigskip\n
The paths $\pi$ and $\pi'$ constructed in this fashion, see Figure 5, satisfy (\ref{A.3}) and
\begin{equation}\label{A.20}
d(\pi(i), \pi'(j)) \ge 11 \wt{\ell} \, L, \;\; 1 \le i \le M, \, 1 \le j \le M'\,.
\end{equation}

\n
Combining (\ref{A.18}) - (\ref{A.20}), we have thus completed the construction of $\pi, \pi'$ satisfying (\ref{A.16}) and Proposition \ref{A.2} follows.
\end{proof}


\end{document}